\documentclass[reqno, 10pt]{amsart}

\usepackage{epsf,amssymb,times, graphicx, amscd, verbatim, amsxtra}
\usepackage[all]{xypic}
\usepackage[mathscr]{eucal}
\usepackage{enumerate}

\newtheorem{thm}{Theorem}[section]

\newtheorem{prop}[thm]{Proposition}

\newtheorem{lemma}[thm]{Lemma}
\newtheorem{cor}[thm]{Corollary}
\newtheorem{conj}[thm]{Conjecture}

\theoremstyle{definition}
\newtheorem{rmk}[thm]{Remark}

\newcommand{\A}{\mathbb{A}}

\newcommand{\C}{\mathbb{C}}
\newcommand{\R}{\mathbb{R}}

\newcommand{\ra}{\rightarrow}

\newcommand{\taut}{\tilde{\tau}}

\newcommand{\Sw}{\mathcal{S}}
\newcommand{\Aut}{\mathcal{A}}
\newcommand{\E}{\mathcal{E}}
\newcommand{\V}{\mathcal{V}}
\newcommand{\Ov}{\mathcal{O}_v}
\newcommand{\Pv}{\mathcal{P}_v}

\newcommand{\sR}{\mathcal{R}}

\newcommand{\Gm}{\mathbb{G}_m}

\newcommand{\Ik}{\operatorname{Ik}}
\newcommand{\Ind}{\operatorname{Ind}\,}
\newcommand{\Res}{\operatorname{Res}\,}

\newcommand{\im}{\operatorname{Im\,}}

\newcommand{\Val}{\operatorname{Val}}
\newcommand{\Span}{\operatorname{span}}

\newcommand{\GSp}{\operatorname{GSp}}
\newcommand{\GSO}{\operatorname{GSO}}
\newcommand{\GO}{\operatorname{GO}}
\newcommand{\Sp}{\operatorname{Sp}}
\newcommand{\SO}{\operatorname{SO}}
\newcommand{\OO}{\operatorname{O}}
\newcommand{\GL}{\operatorname{GL}}
\newcommand{\Spt}{\widetilde{\operatorname{Sp}}}
\newcommand{\Mp}{\operatorname{Mp}}
\newcommand{\U}{\operatorname{U}}

\newcommand{\sign}{\operatorname{sign}}
\newcommand{\LO}{\operatorname{LO}_{\psi}}
\newcommand{\FO}{\operatorname{FO}_{\psi}}

\newcommand{\Pole}{\operatorname{Pole}}
\newcommand{\Zero}{\operatorname{Zero}}

\newcommand{\ie}{{\em i.e. }}

\begin{document}

\title[Siegel-Weil: second term identity]
{On the regularized Siegel-Weil formula\\ (the second term identity) \\and non-vanishing of theta lifts from orthogonal groups}


\author{Wee Teck Gan and Shuichiro Takeda}
\address{Mathematics Department, University of California, San Diego, 9500 Gilman Drive, La Jolla,
92093, U.S.A.}
\address{Department of Mathematics
Purdue University
150 N. University Street
West Lafayette, IN 47907-2067}

\email{wgan@math.ucsd.edu} \email{stakeda@math.purdue.edu}


\date{\today}

\begin{abstract}
We derive a (weak) second term identity for the regularized Siegel-Weil formula for
the even orthogonal group, which is used to obtain a Rallis inner product formula in the ``second term range". As an application, we show the following non-vanishing result of global theta lifts from orthogonal groups. Let $\pi$ be a cuspidal automorphic representation of an orthogonal group $\OO(V)$ with $\dim V=m$ even and $r+1\leq m\leq 2r$. Assume further that there is a place $v$ such that $\pi_v\cong\pi_v\otimes\det$. Then the global theta lift of $\pi$ to $\Sp_{2r}$ does not vanish up to twisting by automorphic determinant characters if the (incomplete) standard $L$-function $L^S(s,\pi)$ does not vanish at $s=1+\frac{2r-m}{2}$. Note that we impose no further condition on $V$ or $\pi$.
We also show analogous non-vanishing results when $m > 2r$ (the ``first term range") in terms of poles of $L^S(s,\pi)$ and consider the ``lowest occurrence" conjecture of the theta lift from the orthogonal group.
\end{abstract}

\maketitle


\section{\bf Introduction}\label{S:intro}


Let $F$ be a number field with ring of adeles $\A$.
Let $\OO(\V) = \OO_{m,m}$ be the split orthogonal group of rank $m$ over $F$ and let $\Sp(W) = \Sp_{2r}$ be the symplectic group of rank $r$. The group $\OO(\V) \times \Sp(W)$ is a dual reductive pair in the symplectic group $\Sp(\V \otimes W)$ and possesses a distinguished representation known as the Weil representation (modulo the choice of a non-trivial character of $F \backslash \A$). If one chooses a maximal isotropic subspace $\V^+$ of $\V$, then this Weil representation can be realized on the space $\Sw((\V^+ \otimes W)(\A))$ of Schwarz functions on $(\V^+ \otimes W)(\A)$. Moreover, it has a natural ``automorphic" realization

\[  \theta: S((\V^+ \otimes W)(\A) \longrightarrow  \{ \text{Functions on $\OO(\V)(F) \backslash \OO(\V)(\A) \times \Sp(W)(F) \backslash \Sp(W)(\A)$} \}\]
which is an $\OO(\V) \times \Sp(W)$-equivariant map.
\vskip 10pt

In this context, the classical Siegel-Weil formula is an identity of automorphic forms on $\OO(\V)$ which relates the theta integral
\[  I(g, \varphi) = \int_{\Sp(W)(F) \backslash \Sp(W)(\A)} \theta(\varphi)(gh) \, dh \]
to a special value of a Siegel Eisenstein series $E (g,s_0; \varphi)$ for each $\varphi \in \Sw((\V^+ \otimes W)(\A)$. When both the theta integral and the sum defining the Eisenstein series are absolutely convergent, such an identity was proved by Weil ([We]). Subsequent pioneering works of Kudla-Rallis (with significant refinements by Ikeda \cite{Ikeda} and Ichino \cite{Ichino}) extend this identity
to situations where these convergent conditions are not satisfied, culminating in a regularized Siegel-Weil formula ([Kd-R2]) where one requires a regularization of the divergent theta integral $I(-,\varphi)$.
Strictly speaking, the work of Kudla-Rallis deals with the case where one considers a theta integral over $\OO(\V)$, so that the Siegel-Weil formula is an identity of automorphic forms on $\Sp(W)$. In the particular context of this paper, the regularized Siegel-Weil formula was established by Moeglin ([Mo]) and later Jiang-Soudry ([JS]). As we shall explain below, the identity established in all these papers is the so-called ``first term identity".
\vskip 10pt

To explain this terminology, let us briefly recall the setup of the regularized Siegel-Weil formula in [Mo] and [JS]. The regularized theta integral $I(g,s,\varphi)$ (whose definition we recall in \S \ref{S:regularize}) turns out to be equal to
a non-Siegel Eisenstein series $\E^{(m,r)}(g,s;\varphi)$ on $\OO(\V)$.
Suppose that
\[
    m > 2r,
\]
so that we are in what we call ``first term range". Then the regularized theta integral has at most a simple pole at $s=\rho_r'=\frac{r+1}{2}$ and one may consider its Laurent expansion there:
\[
    \E^{(m,r)}(s,g;\varphi)=\frac{B_{-1}^{(m,r)}(\varphi)(g)}{s-\rho_r'}
    +B_0^{(m,r)}(\varphi)(g)+\cdots.
\]
On the other hand, one has the Siegel Eisenstein series $E^{(m,m)}(g, s; \varphi)$ and one may consider its Laurent expansion at $s=\rho_{m,r}=\frac{2r-m+1}{2}$:
\[
     E^{(m,m)}(g,s;\Phi_{\varphi})=\frac{A_{-1}^{(m,r)}(\varphi)(g)}{s-\rho_{m,r}}
    +A_{0}^{(m,r)}(\varphi)+A_1^{(m,r)}(\varphi)(g)(s-\rho_{m,r})+\cdots.
\]
The regularized Siegel-Weil formula proved in the papers [Mo] and [JS] is
an identity between the first (or leading) terms in these Laurent expansions. Hence such a formula is referred to as a ``first term identity".

\vskip 10pt
However, the main concern of this paper is what we call ``second term range":
\[  r+1 \leq m \leq 2r. \]
In this case, the regularized theta integral has a double pole at the point of interest
$s=\rho_r'=\frac{r+1}{2}$, so that its Laurent expansion there looks like:
\[
    \E^{(m,r)}(s,g;\varphi)= \frac{B_{-2}^{(m,r)}(\varphi)(g)}{(s-\rho_r')^2} + \frac{B_{-1}^{(m,r)}(\varphi)(g)}{s-\rho_r'}    +B_0^{(m,r)}(\varphi)(g)+\cdots.
\]
The Laurent expansion of the Siegel Eisenstein series at $s=\rho_{m,r}=\frac{2r-m+1}{2}$ is as before:
\[
     E^{(m,m)}(g,s;\Phi_{\varphi})=\frac{A_{-1}^{(m,r)}(\varphi)(g)}{s-\rho_{m,r}}
    +A_{0}^{(m,r)}(\varphi)+A_1^{(m,r)}(\varphi)(g)(s-\rho_{m,r})+\cdots.
\]
In this case, we shall show that there is still a first term identity relating $B_{-2}$ with $A_{-1}$.
But the chief contribution of this paper is to show that there is also a relationship between the second terms of the two Laurent series above (for a certain class of $\varphi$). Hence we shall establish a (weak) second term identity relating $B_{-1}$ with $A_0$.
More precsiely, we will prove the following. (Again see \S \ref{S:regularize} for the notations.)
\vskip 10pt

\begin{thm}[Weak Second Term Identity]
Let $\dim \V^{+}=m$ with $r+1\leq m\leq 2r$. Then for all $\varphi\in \Sw((\V^{+}\otimes W)(\A))$ that are in the $\OO(\V,\A)$-span of a particular spherical vector $\varphi^0$, the following identity holds:
\begin{align*}
    &B_{-1}^{(m,r)}(\varphi)
    +|D|^{-m\rho_{m,r}}\prod_{i=0}^{2r-m}\frac{\xi(i)}{\xi(2r-2i)}B_0^{(m,m-r-1)}(\Ik(\varphi))\\
    &\qquad\qquad\equiv A_0^{(m,r)}(\varphi)
     \mod \im A^{(m,r)}_{-1},
\end{align*}
where $\im A^{(m,r)}_{-1}$ is the image of the $A^{(m,r)}_{-1}$ map.
\end{thm}
\vskip 10pt

\noindent  Here we call this ``the weak second term identity" because of the assumption we have to impose on $\varphi$. It is expected that this assumption can be removed, but it seems that our method is unable to do this. We note that this weak second term identity was proven by A. Ichino and the first author in their recent preprint \cite{GI} for the case $r=2$ and $m=4$. Strictly speaking, in \cite{GI} they consider the similitude group $\GO(\V)$, but exactly the same computation yields the same formula for the isometry case. We should also mention that the full second term identity was obtained in certain other low rank cases by Kudla-Rallis-Soudry  [KRS] (on the symplectic group $\Sp_4$ for the dual pair $\OO_4\times  \Sp_4$) and V. Tan [T]  (on the unitary group $\U(2,2)$ for the dual pair $\U(3) \times \U(2,2)$). Also the spherical second term identity on the symplectic group has been known to Kudla already since 1989 , whose precise statement appeared in \cite[Thm. 3.13]{Kudla}. This is essentially the symplectic analogue of our Theorem \ref{T:Spherical_Second_Term}. However his proof has never appeared anywhere to our best knowledge.

\vskip 10pt

The main application of the (regularized) Siegel-Weil formula is in the derivation of the Rallis inner product formula for the theta lift from $\OO_m$ to $\Sp_{2r}$. In the first term range,
using the first term identity, we obtain in \S  \ref{S:inner_product} such a Rallis inner product formula and use it to obtain a sufficient condition for the non-vanishing of this theta lift in terms of the location of poles of standard L-functions of $\OO(m)$ (see below). In this way, we recover certain results of Moeglin ([Mo]) and Ginzburg-Jiang-Soudry ([GJS]) in a somewhat more direct manner.
\vskip 10pt

As we explain in \S \ref{S:inner_product}, it is necessary to establish a ``second term identity" in the ``second term range" in order to obtain a Rallis inner product formula for the theta lift from
$\OO_m$ to $\Sp_{2r}$; the first term identity in this range would give one nothing! This Rallis inner product formula relates the inner product of the theta lifts to the values of the standard L-function at certain points, rather than the residues of the standard L-function at these points.
Thus,  as an application of our second term identity, we prove the following non-vanishing result for the global theta lift.
\vskip 5pt

\begin{thm}
Let $\pi$ be a cuspidal automorphic representation of $\OO(V,\A)$
with $\dim V=m$ even and $r+1 \leq m \leq 2r$. Assume that
\vskip 5pt

\begin{enumerate}[(i)]
\item  there is a place $v$ where $\pi_v\cong\pi_v\otimes\det$;
\vskip 5pt

\item  the (incomplete) standard $L$-function $L^S(s,\pi)$ does not vanish at $s=1+\frac{2r-m}{2}$ (a pole is allowed).
\end{enumerate}
Then there exists an automorphic character $\mu$ on
$\mu_2(F)\backslash\mu_2(\A)$ such that the global theta lift
$\theta_{2r}(\pi\otimes(\mu\circ\det))$ of $\pi\otimes(\mu\circ\det)$ to the symplectic
group $\Sp_{2r}(\A)$ of rank $r$ does not vanish.
\end{thm}
\vskip 5pt

Another way of stating the above theorem is that if $\sigma$ is a cuspidal automorphic representation of $\SO(V,\A)$ which satisfies the analogous conditions in the theorem, then $\sigma$
has a nonzero theta lift to $\Sp_{2r}$.
\vskip 10pt

Let us mention that the condition (i) in the theorem is merely technical, though we do not know how to suppress it. As a result, we were not able to obtain the analogous theorem when $\dim V$ is odd.
Indeed, if $\dim V=m$ is odd, then $\pi_v\ncong\pi_v\otimes\det$ for all $v$, and hence the assumption (i) of the theorem is never satisfied.
However, we would like to emphasize that our non-vanishing result applies to an
orthogonal group $\OO(V)$ for any quadratic space $V$, and moreover we do not impose
any further assumption on $\pi$ such as genericity or temperedness.
\vskip 10pt

In the first term range $m > 2r$, the analogous results on non-vanishing of theta lifts were first shown by Moeglin (\cite{Moeglin}) and Ginzburg-Jiang-Soudry (\cite{GJS}), using
Moeglin's idea of ``generalized doubling method" (\cite{Moeglin}).
As we mentioned above, these results can also be shown more directly using the Rallis inner product formula:
 \vskip 5pt

 \begin{thm}
Let $\pi$ be a cuspidal automorphic representation of $\OO(V,\A)$
for $\dim V=m$ where $m\geq 2r+1$.
\begin{enumerate}[(a)]
\item Assume $m\geq 2r+2$. Further assume that the (incomplete) standard $L$-function $L^S(s,\pi)$ has a pole at $s=\frac{m-2r}{2}$. Then there is a character $\mu$ on $\mu_2(F)\backslash\mu_2(\A)$ such
that $\theta_{2r}(\pi\otimes(\mu\circ\det))\neq0$.
\item Assume $m=2r+1$. Further assume that the (incomplete) standard $L$-function $L^S(s,\pi)$ does not vanish at $s=\frac{1}{2}$. Then there is a character $\mu$ on $\mu_2(F)\backslash\mu_2(\A)$ such
that $\theta_{2r}(\pi\otimes(\mu\circ\det))\neq0$.
\end{enumerate}
\end{thm}

For this theorem, we do not need the assumption $\pi_v\cong\pi_v\otimes\det$ for some place $v$, and hence $\dim V$ can be odd. This is because of the crucial Proposition \ref{P:incoherent}.

\vskip 10pt

Now once these theorems have been proven, the natural question to ask is whether the converses are true. We are not able to answer this question. In \cite{GJS}, a conjecture related to this issue was made by using the notion of the ``lowest occurrence" for the range $m\geq 2r+1$. In the very last section, we consider this conjecture in some detail not only for the range $m\geq 2r+1$, but for any range.

\quad\\

\begin{center}
Acknowledgments
\end{center}
We thank Atsushi Ichino for his help on many occasions.
W. T. Gan is partially supported by NSF grant 0801071.
Also part of this paper was written when S. Takeda was at Ben Gurion University of the Negev in Israel. During his stay there, he was supported by the Skirball postdoctoral fellowship of the Center of Advanced Studies in Mathematics at the Mathematics Department of Ben Gurion University.

\quad\\



\section{\bf Notations and Preliminaries}\label{S:notation}


In this paper,  $F$ is a number field with ring of integers $\mathcal{O}$ and ring of adeles $\A$. Fix a non-trivial additive character $\psi=\otimes_v'\psi_v$ on $F\backslash \A$. For each finite $v$, let $c_v$ be the conductor of the additive character $\psi_v$, so that $\psi_v$ is trivial on $\Pv^{-c_v}$. We fix the Haar measure $dx_v$ on $F_v$ (for all $v$) to be self-dual with respect to $\psi_v$, and hence the volume of $\Ov$ is $q_v^{-c_v/2}$.
\vskip 5pt

Let
\[
    \xi(s)=|D|^{\frac{s}{2}}\prod_{v\leq\infty}\zeta_v(s)
\]
be the complete normalized zeta function of $F$, where $D$ is the discriminant of $F$. It satisfies the functional equation
\[
    \xi(1-s)=\xi(s).
\]
Note that $\xi$ has residues at $s=0, 1$. In this paper, we write
\[
    \xi(0):=\underset{s=0}{\Res}\xi(s) \quad\text{ and }\quad \xi(1):=\underset{s=1}{\Res}\xi(s).
\]
Here ${\Res}_{s=s_0}\xi(s)$ of course means the residue of $\xi(s)$ at $s=s_0$. Similarly for any meromorphic function $F(s)$, ${\Res}_{s=s_0}F(s)$ refers to the residue of $F(s)$ at $s=s_0$. We also denote by ${\Val}_{s=s_0}F(s)$ the ``value" of $F(s)$, \ie the constant term of the Laurent series of $F(s)$ at $s=s_0$. Note that ${\Val}_{s=s_0}F(s)$ makes sense even though $F(s)$ is not holomorphic at $s=s_0$.
\vskip 5pt

For an algebraic group $G$ over $F$, we occasionally write
\[
    [G]:=G(F)\backslash G(\A)
\]
for the sake of saving space. Also we denote by $\Aut(G)$ the space of automorphic forms on $G$. We do not impose the $K$-finiteness condition on the elements of $\Aut(G)$ so that the full group $G(\A)$ acts on $\Aut(G)$.
\vskip 5pt

Let $\V=\V^+\oplus \V^-$ be a split quadratic space with $\V^{\pm}$ maximal isotropic subspaces of dimension $m$, and let $W=W^+\oplus W^-$ be the symplectic space with $\dim W^{\pm}=r$. Because the spaces $\V$ and $W$ are split, one can fix self-dual lattices of $\V$ and $W$ which are compatible with the decompositions $\V^+  \oplus \V^-$  and $W^+ \oplus W^-$, thereby endowing the spaces $\V^{\pm}$ and $W^{\pm}$ with $\mathcal{O}$-integral structures. Choosing bases for these
lattices also gives us identifications $\V \cong F^{2m}$ and $W \cong F^{2r}$, which are well-defined up to the natural action of $\GL_{2m}(\mathcal{O})$ and $\GL_{2r}(\mathcal{O})$ respectively. Via these identifications and using our fixed Haar measure $dx_v$ on $F_v$,
we obtain additive Haar measures on the spaces $\V^{\pm}$ and $W^{\pm}$.
\vskip 10pt

 Set
\[  H=\Sp_{2r}=\Sp(W) \quad \text{and} \quad  G=\OO_{m,m}=\OO(\V).\]
The $\mathcal{O}$-integral structures on $\V$ and $W$ also endow the groups $G$ and $H$
with $\mathcal{O}$-integral structures.
This determines maximal compact subgroups $G(\widehat{\mathcal{O}})$ and
$H(\widehat{\mathcal{O}})$ of
$G(\A_f)$ and $H(\A_f)$ respectively. Picking maximal compact subgroups arbitrarily for the archimedean places, we thus obtain maximal compact subgroups
$K_G $ and $K_H$ of $G(\A)$ and $H(\A)$ respectively.

\vskip 10pt

Now let $\omega=\omega_\psi$ be the Weil representation of $G(\A)\times H(\A)$, which can be realized on the space $\Sw((\V^+\otimes W)(\A))$ of Schwartz functions on the adelic vector space $(\V^+\otimes W)(\A)$. Recall that one has the partial Fourier transform
\[
    \Sw((\V^+\otimes W)(\A))\rightarrow\Sw((\V\otimes W^+)(\A))
\]
defined by
\[
    \hat{\varphi}(u\oplus v)=\int_{(\V^+\otimes W^-)(\A)}\varphi(x\oplus u)\cdot\psi(\langle x,v\rangle)\,dx
\]
where $u\in(\V^+\otimes W^+)(\A)$ and $v\in(\V^-\otimes W^+)(\A)$ and $dx$ is the fixed additive Haar measure on $\V^+ \otimes W^-$.
\vskip 10pt

Following \cite[Sec. 7]{GI}, we define \textbf{the spherical Schwartz function}
\[
    \varphi^0=\otimes_v\varphi^0_v\in(\V^+\otimes W)(\A)
\]
to be such that the partial Fourier transform
$\hat{\varphi}^0_v$ is
\begin{enumerate}[$\bullet$]
\item the characteristic function of $(\V\otimes W^+)(\Ov)$ if $v$ is non-archimedean

\item the Gaussian if $v$ is archimedean, namely for $x\in (\V\otimes W^+)(F_v)$
\[
    \hat{\varphi}^0_v(x)=
    \begin{cases}
    \exp(-\pi\langle x, x^\ast\rangle)\quad\text{if $v$ is real}\\
    \exp(-2\pi\langle \bar{x}, x^\ast\rangle)\quad\text{if $v$ is complex},\\
    \end{cases}
\]
where $x^\ast$ is the image of $x$ under the map $\V\otimes W^+\rightarrow \V\otimes W^-$ given by $v\otimes w\mapsto v\otimes w^\ast$ with $w^\ast$ defined as follows: if $\{e_1,\dots,e_r,f_1,\dots,f_r\}$ is the symplectic basis giving the integral $\mathcal{O}$-structure on $W=W^+\oplus W-$ such that
$\langle e_i, f_i\rangle=1$, then $e_i^\ast=f_i$.
\end{enumerate}
 \vskip 5pt

It is not difficult to show that
\[
    \varphi^0_v=
    \begin{cases}
    q_v^{-c_vrm/2}\times\text{the characteristic function of }\\
    \qquad\qquad(\V^+\otimes W^-)(\Pv^{-c_v})\oplus(\V^+\otimes W^+)(\Ov)
    \text{ if $v$ is non-archimedean},\\
    \text{the Gaussian if $v$ is archimedean}.
    \end{cases}
\]
Then because
\[
    \prod_vq_v^{c_v}=|D|,
\]
we see that
\[
    \varphi^0(0)=|D|^{-rm/2}.
\]

Now we define
\[
	\Sw((\V^+\otimes W)(\A))^\circ=\text{ the $\OO(\V)(\A)$-span of the spherical $\varphi^0$}.
\]
Also for each $\varphi\in\Sw((\V^+\otimes W)(\A))$, we define
\[
    \varphi_{K_H}(v\otimes w)=\int_{K_H}\varphi(v\otimes kw)\,dk\quad\text{for $v\otimes w\in (\V^+\otimes W)(\A)$}.
\]



\section{\bf The Regularized Siegel-Weil formula}\label{S:regularize}


The first term identity of the regularized Siegel-Weil formula identifies the leading term of
the Siegel Eisenstein series with the
regularized theta integral (\cite{Kudla-Rallis94, Ikeda, Ichino, Moeglin}) . In this section, we will review this theory to the extent we need it. Essentially everything in this
section is already known.


\subsection{Eisenstein series associated with degenerated principal series}

First, let us fix some of the notations and review the basics of the Eisenstein series on the split
orthogonal group $\OO_{m,m}$ of rank $m$ associated with a family of degenerate principal series.
\vskip 10pt

For each integer $r$ with $1\leq r\leq m$, we let $P_{r}$ be the
standard maximal parabolic subgroup of $\OO_{m,m}$ stabilizing an isotropic subspace of rank $r$, so that the Levi factor
is isomorphic to $\GL_r\times\OO_{m-r,m-r}$. As usual, we denote the degenerate principal series by
\[
    I^{m,r}(s)=\Ind_{P_{r}(\A)}^{\OO_{m,m}(\A)}|\,|^s\quad
    (\text{normalized induction}).
\]
Here $|\,|^s$ is the character $|\det|^s$ with $\det$ is on $\GL_r$ and extended to $P_r$ trivially on the rest. Recall that a section $f(-,s)\in I^{m,r}(s)$ is said to be standard (or flat) if its restriction to the maximal compact $K$ of $\OO_{m,m}(\A)$ is independent of $s$. For each standard section $f(-,s)\in I^{m,r}(s)$, we form
the Eisenstein series $E^{(m,r)}(g, s; f)$ by
\[
    E^{(m,r)}(g, s; f)=\sum_{\gamma\in
    P_{r}(F)\backslash\OO_{m,m}(F)}f(\gamma g, s).
\]
It is well-known that this sum converges absolutely for $\Re(s)\gg 0$
and admits a meromorphic continuation.
\vskip 5pt

When $r=m$, we call $E^{(m,r)}(g, s; f)$ the Siegel Eisenstein series, and otherwise a non-Siegel Eisenstein series. Also when $f$ is a spherical standard section with $f(1)=1$, we call the corresponding Eisenstein series the spherical Eisenstein series.


\subsection{The Siegel Eisenstein series}

Next we will review the theory of the Siegel Eisenstein series of
the split orthogonal group and its relations to the Weil
representation.
\vskip 10pt

Recall that $G(\A)$ has an Iwasawa decomposition
\[
    G(\A)=P(\A)K_G,
\]
where $P$ is the Siegel parabolic subgroup that fixes $\V^+$. Then
for each $g\in G(\A)$, we write $g=nm(a)k$ and
$|a(g)|=|\det(a)|_{\A}$ following the convention of Kudla and
Rallis (\cite[p.9-10]{Kudla-Rallis94}). Also we denote
\[
    \rho_{m}=\frac{m-1}{2}
\]
and
\[
    \rho_{m,r}=r-\rho_m=\frac{2r-m+1}{2}
\]

Now consider the degenerate principal series
$\Ind_{P(\A)}^{G(\A)}|\,|^s$ (normalized induction). For
each standard section $\Phi$, we form the Siegel Eisenstein series
\[
    E^{(m,m)}(g,s;\Phi)=\sum_{\gamma\in P(F)\backslash
    G(F)}\Phi(\gamma g,s).
\]
It is known that this sum is absolutely convergent for
$\Re(s)>\rho_m=\frac{m-1}{2}$, and it has meromorphic continuation
together with the functional equation given by
\[
    E^{(m,m)}(g,s;\Phi)=E^{(m,m)}(g,-s;M_m(s)\Phi),
\]
where $M_m(s)$ is the intertwining operator defined by
\[
    M_m(s)\Phi(g,s)=\int_{N(\A)}\Phi(wng,s)\, dn,
\]
for $\Re(s)$ sufficiently large and by meromorphic
continuation in general.
\vskip 10pt

The following is due to Kudla and Rallis (Theorem 1.0.1 of \cite{Kudla-Rallis90}).
\begin{prop}
Let $\Phi$ be a standard section. Then the Siegel Eisenstein series $E^{(m,m)}(g,s;\Phi)$ has at most simple poles, and for $\Re(s)>0$ those occur only in the set
\[
	\{\hat{0}, \dots,\rho_m-1,\rho_m\}
    =\begin{cases}
    \{\frac{1}{2},\frac{3}{2},\dots,\rho_m\}\quad\text{if $m$ is even};\\
    \{1,\cdots,\rho_m\}\quad\text{if $m$ is odd},
    \end{cases}
\]
where $\hat{0}$ means $0$ is omitted.
\end{prop}

We are interested in the residues and values of $E^{(m,m)}(g,s;\Phi)$ at $s=\rho_{m,r}=\frac{2r-m+1}{2}$ for various $r$. We write the Laurent expansion of $E^{(m,m)}(g,s;\Phi)$ at $\rho_{m,r}=\frac{2r-m+1}{2}$ as
\begin{align*}
     E^{(m,m)}(g,s;\Phi)&=\sum_{d\geq-1}^{\infty}A_d^{(m,r)}(\Phi)(g)(s-\rho_{m,r})^d\\
     &=\frac{A_{-1}^{(m,r)}(\Phi)(g)}{s-\rho_{m,r}}
    +A_{0}^{(m,r)}(\Phi)+A_1^{(m,r)}(\Phi)(g)(s-\rho_{m,r})+\cdots.
\end{align*}
Note that each $A_d^{(m,r)}(\Phi)$ is an automorphic form on $G(\A)$.
\vskip 5pt

The Siegel Eisenstein series is related to the Weil representation in the following way. Let $\Sw((\V^+\otimes W)(\A))$ be the space of
Schwartz functions giving rise to the Weil representation, where $\dim W=2r$ and $\dim \V=\dim (\V^+\oplus \V^-)=2m$. Then for each
$\varphi\in\Sw((\V^+\otimes W)(\A))$, we have a map
\begin{align*}
    \Phi^{(m,r)}: \Sw((\V^+\otimes W)(\A))&\longrightarrow
    \Ind_{P(\A)}^{G(\A)}|\,|^s\\
     \varphi&\mapsto \Phi^{(m,r)}_{\varphi}
\end{align*}
given by
\[
    \Phi^{(m,r)}_{\varphi}(g, s)=\omega(g)\varphi(0)\cdot|a(g)|^{s-\rho_{m,r}}.
\]
Here we should emphasize that this map is \emph{not} $G(\A)$-intertwining unless $s=\rho_{m,r}$. Also notice that
$\Phi^{(m,r)}_{\varphi}$ is clearly a standard section, and so the poles of the Eisenstein series
$E^{(m,m)}(g,s;\Phi^{(m,r)}_{\varphi})$ are at most simple. By abuse of notation, we sometimes write
\[
    A_d^{(m,r)}(\Phi^{(m,r)}_{\varphi})=A_d^{(m,r)}(\varphi),
\]
whenever there is no danger of confusion. Then we have the map
\[
    A_{d}^{(m,r)}:\Sw((\V^+\otimes W)(\A))\longrightarrow
    \Ind_{P(\A)}^{G(\A)}|\,|^{\rho_{m,r}}\longrightarrow\Aut(G).
\]
Now $A_{d}^{(m,r)}$ is $G(\A)$-intertwining if $A_{d}^{(m,r)}$ is the leading term of the Laurent expansion i.e.  $A_{0}^{(m,r)}$ if $m=2r+1$ and $A_{-1}^{(m,r)}$ otherwise. (Let us also note that  $A_d^{(m,r)}$ is $H(\A)$-invariant for all $d$. But we do not use this fact in this paper, even though it plays a pivotal role in the original work on the regularized Siegel-Weil formula by Kudla-Rallis (\cite{Kudla-Rallis94}).)
\vskip 10pt

Now consider the spherical Schwartz function $\varphi^0\in\Sw((\V^+\otimes W)(\A))$ defined in \S \ref{S:notation}. One sees that
\[
    \Phi_{\varphi^0}^{(m,r)}(1,s)=\varphi^0(0)=|D|^{-rm/2}.
\]
Hence at $s=\rho_{m,r}$ we have
\begin{equation}\label{E:E^(m,m)}
    E^{(m,m)}(g, s; \Phi^{(m,r)}_{\varphi^0})=|D|^{-rm/2}E^{(m,m)}(g, s; \Phi^0),
\end{equation}
where $\Phi^0(-, s)$ is the spherical section in
$\Ind_{P(\A)}^{G(\A)}|\,|^s$ normalized as $\Phi^0(1,s)=1$.


\subsection{The regularized theta integral}

Next we will review the theory of the regularized theta integral and how it is related to the non-Siegel Eisenstein series. In this section (and indeed for this entire paper), we assume that we are outside the convergent range of Weil, namely
\[
    m>r.
\]
(See \cite{Kudla-Rallis94} regarding the convergent range.) First for each $\varphi\in\Sw(\V\otimes W^+)$ let us define the theta integral
\[
    I^{(m,r)}(g,s;\varphi)=\int_{H(F)\backslash
    H(\A)}\theta(g,h;\varphi)E(h,s)\, dh,
\]
where
\[
    \theta(g,h;\varphi)=\sum_{v\in (\V^+\otimes W)(F)}\omega(g,h)\varphi(v),
\]
and $E(h,s)$ is the spherical Siegel Eisenstein series on $H(\A)=\Sp_{2r}(\A)$ defined by
\[
    E(h,s)=\sum_{\gamma\in Q_r(F)\backslash H(F)}\Psi(\gamma h,s),
\]
where
\[
    \Psi(h,s)=|a(h)|^{s+\rho_r'},\quad
    \rho_r'=\frac{r+1}{2},
\]
and $Q_r$ is the Siegel parabolic subgroup of $\Sp_{2r}(\A)$. The theta integral might not converge for all $\varphi$. However, there is an element $z$ (resp. $z'$) called a regularizing element in the center of universal enveloping algebra of $\mathfrak{g}_v$ (resp. $\mathfrak{h}_v$) for $v$ real (\cite{Kudla-Rallis94}) or in the spherical Hecke algebra of $G(F_v)$ (resp. $H(F_v)$) for $v$ non-archimedean (\cite{Ichino}) so that
\begin{align*}
    \omega(z)&=\omega(z'),\\
    \omega(g,h)\omega(z)&=\omega(z)\omega(g,h)
\end{align*}
(\ie the action of $z$ (and hence $z'$) commutes with the action of $G(\A)\times H(\A)$),
and such that the function $\theta(g,-,\omega(z)\varphi)$ is rapidly decreasing on any Siegel domain
of $\Sp_{2r}(\A)$. Thus the theta integral $I^{(m,r)}(g,s;\omega(z)\varphi)$ converges for $\Re(s)\gg0$ for all $\varphi$. We call it the regularized theta integral.
\vskip 10pt

By unfolding the Eisenstein series inside the regularized theta integral, we have
\begin{align*}
    &I^{(m,r)}(g,s;\omega(z)\varphi)\\
    &=\int_{\GL_r(F)\backslash \GL_r(\A)}
        \sum_{\gamma\in P_r(F)\backslash G(F)}\sum_{\alpha\in
        \GL_r(F)}\omega(\gamma g,m(\alpha
        a))\omega(z)\hat{\varphi}_{K_H}(w_0)|a|^{s-\rho_r'}\, da\\
        &=\sum_{\gamma\in P_r(F)\backslash G(F)}\int_{\GL_r(\A)}
        \omega(\gamma g,
        m(a))\omega(z)\hat{\varphi}_{K_H}(w_0)|a|^{s-\rho_r'}\,da\\
        &=\sum_{\gamma\in P_r(F)\backslash G(F)}\int_{\GL_r(\A)}
        \omega(\gamma g,
        1)\omega(z)\hat{\varphi}_{K_H}(^taw_0)|a|^{s+m-\rho_r'}\,da.
\end{align*}
Here $w_0$ is the element in $\Sw(\V\otimes W^+)$ corresponding to
the element of the form $v_1\otimes e_1+v_2\otimes e_2+\cdots
+v_r\otimes e_r$, where $v_i$'s and $e_i$'s are the obvious basis
elements of $\V$ and $W$, and $P_r$ is the parabolic subgroup which
stabilizes the isotropic $r$-plane in $\V$ spanned by
$v_1,\dots,v_r$. (See \cite[p.48]{Kudla-Rallis94}.)

Now define
\[
    f^{(m,r)}(g,s;\varphi)=\int_{\GL_r(\A)}\omega(g, 1)
    \hat{\varphi}(^taw_0)|a|^{s+m-\rho_r'}\,da
\]
and so
\[
    I^{(m,r)}(g,s;\omega(z)\varphi)
    =\sum_{\gamma\in P_r(F)\backslash G(F)}f^{(m,r)}(\gamma
    g,s;\omega(z)\varphi_{K_H}).
\]
Here note that
\[
    f^{(m,r)}(-,s;\varphi)\in \Ind_{P_{r}(\A)}^{G(\A)}|\; |^s \quad\quad\text{(normalized
    induction)},
\]
and the map
\begin{align*}
    \Sw((\V^+\otimes W)(\A)) &\ra \Ind_{P_{r}(\A)}^{G(\A)}|\; |^s\\
    \varphi&\mapsto f^{(m,r)}(-,s;\varphi)
\end{align*}
is a $G(\A)$-intertwining operator for $\Re(s)$ sufficiently large.
Indeed, the integral for $f^{(m,r)}(g,s;\varphi)$ converges for
$\Re(s)>\frac{2r-m+1}{2}$. Also
\[
    f^{(m,r)}(g,s;\omega(z)\varphi)=P_z(s)f^{(m,r)}(g,s;\varphi),
\]
where $P_z(s)$ is a holomorphic function in $s$ depending on $z$. In fact $P_z(s)$ is a polynomial in $s$ if we choose our regularizing element $z$ from the center of universal enveloping algebra at an archimedean place, and $P_z(s)\in\C[q_v^{-s},q_v^s]$ if $z$ is from the spherical Hecke algebra at a non-archimedean $v$. For example, if we choose $z$ from the center of universal enveloping algebra at a real place, it is explicitly given by
\[
        P_z(s)=\prod_{i=0}^{r-1}\left((s-\frac{r-1}{2}+i)^2-(m-r)^2\right).
\]
(See \cite[p.51]{Kudla-Rallis94}.) In any case, we have
\[
    I^{(m,r)}(g,s;\omega(z)\varphi)=P_z(s)\E^{(m,r)}(g,s;\varphi),
\]
where
\[
    \E^{(m,r)}(g,s;\varphi)=\sum_{\gamma\in P_{r}(F)\backslash G(F)} f^{(m,r)}(\gamma
    g,s;\varphi_{K_H}),
\]
for the region $\Re(s)>\max\{r-\rho_r', \rho_r'\}$. (See the second paragraph of \cite[p.53]{Kudla-Rallis94}.)
\vskip 5pt

Also, $z'$ as an operator on the space of automorphic forms on $H$ is self-adjoint for the Peterson inner product with the property
\[
    z'\ast E(h,s)=P_z(s)E(h,s).
\]
\vskip 5pt

It is known that at $s=\rho_r'=\frac{r+1}{2}$ the non-Siegel Eisenstein series $\E^{(m,r)}(g, s;\varphi)$ has at most a double pole, and if $m\geq 2r+1$, then it only has a simple pole. (See \cite[bottom of p.53]{Kudla-Rallis94}.) Accordingly, we call
\begin{align*}
 	m\geq 2r+1\quad\text{\ie}\quad r\leq \frac{m}{2}-1&:\text{\bf 1st term range}\\
	r+1\leq m\leq 2r\quad\text{\ie}\quad\frac{m}{2}\leq r\leq m-1&: \text{\bf 2nd term range}.
\end{align*}
Also we sometimes call
\begin{align*}
m=2r+1&:\text{\bf boundary case}.
\end{align*}
The rationale for this terminology is that for the 1st term range, we will need only the first term identity, and for the 2nd term range we will need the 2nd term identity. Also, for the boundary case, the first term identity differs from the non-boundary case. This will be clearer in due course.
\vskip 10pt

We write the Laurent expansion of $\E^{(m, r)}(g, s;\varphi)$ at
$s=\rho_r'=\frac{r+1}{2}$ as
\begin{align*}
    \E^{(m,r)}(g,s;\varphi)&=\sum_{d\geq -2}^{\infty}B_d^{(m,r)}(\varphi)(g){(s-\rho_r')^d}\\
    &=\frac{B_{-2}^{(m,r)}(\varphi)(g)}{(s-\rho_r')^2}
    +\frac{B_{-1}^{(m,r)}(\varphi)(g)}{(s-\rho_r')}+B_0^{(m,r)}(\varphi)(g)+\cdots.
\end{align*}

Note that each $B_d^{(m,r)}$ is a map from $\Sw((V^{+}\otimes W)(\A))$ to $\Aut(G)$. Unlike the case of the Siegel Eisenstein series, however, all the $B_d^{(m,r)}$'s are $G(\A)$-intertwining. This follows from the fact that $\omega(z)$ commutes with the action of $G(\A)$ on $\Sw((V^{+}\otimes W)(\A))$. To be more specific, note that
\begin{align*}
    \E^{(m,r)}(g,s;\varphi)&=\frac{1}{P_z(s)}I^{(m,r)}(g,s;\omega(z)\varphi)\\
    &=\frac{1}{P_z(s)}\int_{H(F)\backslash
    H(\A)}\theta(g,h;\omega(z)\varphi)E(h,s)\, dh\\
    &=\int_{H(F)\backslash
    H(\A)}\theta(g,h;\omega(z)\varphi)\frac{E(h,s)}{P_z(s)}\, dh.
\end{align*}
If we write the Laurent series of $\frac{E(h,s)}{P_z(s)}$ at $s=\rho_r'$ as
\[
    \frac{E(h,s)}{P_z(s)}=\sum_{d=-2}^{\infty}C_d(h)(s-\rho_r')^d
    =\frac{C_{-2}(h)}{(s-\rho'_r)^2}+\frac{C_{-1}(h)}{s-\rho'_r}+C_0(h)+\cdots,
\]
then one sees that
\[
    B_d^{(m,r)}(\varphi)(g)=\int_{H(F)\backslash
    H(\A)}\theta(g,h;\omega(z)\varphi)C_{d}(h)\, dh.
\]
Then the assertion follows from
\[  \theta(gg',h;\omega(z)\varphi)=\theta(g,h;\omega(z)\omega(g',1)\varphi) \]
for all $g'\in G(\A)$.
\vskip 5pt

Let us also note also that the first non-zero $B_d^{(m,r)}$, which turns out to be either $B_{-1}^{(m,r)}$ or $B_{-2}^{(m,r)}$ depending on the range, is $H(\A)$-invariant, but we do not use this fact, even though it is pivotal in the theory of the regularized Siegel-Weil formula.
\vskip 10pt

Finally, let us assume $\varphi^0\in\Sw((V^{+}\otimes W)(\A))$ is our spherical Schwartz function as defined in \S \ref{S:notation}. Then $f^{(m,r)}(-,s;\varphi^0)\in\Ind_{P_{r}(\A)}^{G(\A)}|\; |^s$ is a spherical section if $\Re(s)$ is sufficiently large. Let $f^0\in \Ind_{P_{r}(\A)}^{G(\A)}|\; |^s$ be the spherical section with the property $f^0(1,s)=1$. Then it was shown in \cite{GI} that
\begin{lemma}\label{L:E^(m,r)}
\[
    \E^{(m,r)}(g, s;\varphi^0)
    =|D|^{-r(s+m-\frac{r+1}{2})/2}\prod_{i=1}^{r}\frac{\xi(s+m-\frac{r+1}{2}-(i-1))}{\xi(i)}E^{(m,r)}(g,s;f^0).
\]
(Here recall our convention $\xi(1)=\underset{s=1}{\Res}\xi(s)$.)
\end{lemma}
\begin{proof}
This follows from the fact that
\[
    f^{(m,r)}(1,s;\varphi^0)=|D|^{-r(s+m-\rho_r')/2}\prod_{i=1}^{r}\frac{
    \xi(s+m-\frac{r+1}{2}-(i-1))}{\xi(i)}.
\]
(See \cite[Lemma 7.4]{GI} for the detail computation.)
\end{proof}


\subsection{The regularized Siegel-Weil formula (first term identity)}


Finally, we can state the first term identity of the regularized Siegel-Weil formula in the first term range.
 To be precise, we have the following first term identity, which is essentially due to Moeglin (\cite{Moeglin}) and completed by Jiang and Soudry (\cite{Jiang_and_Soudry}), though their proof heavily depends on Kudla and Rallis (\cite{Kudla-Rallis94}).

\begin{prop}\label{P:first_term}
Assume $r<\frac{m-1}{2}$, \ie $m>2r+1$. Then for all
$\varphi\in\Sw((\V^+\otimes W)(\A))$, there exists a standard section $\Phi(-,s)\in\Ind_{P(\A)}^{G(\A)}|\,|^s$ such that
\[
    B_{-1}^{(m,r)}(\varphi)(g)=A_{-1}^{(m,m-r-1)}(\Phi)(g)
\]
for all $g\in G(\A)$. Moreover if $\varphi$ is spherical, then $\Phi$ can be chosen to be a spherical section.
\end{prop}
\begin{proof}
The first part is \cite[Theorem 2.4]{Jiang_and_Soudry}. The second part follows from the fact that $A_{-1}^{(m,m-r-1)}$ is a $G(\A)$-intertwining map from $\Ind_{P(\A)}^{G(\A)}|\,|^s$ to $\Aut(G)$. Namely, for any $\Phi$ corresponding to a spherical $\varphi$ in this formula, if we take $\Phi'$ to be such that $\Phi'(g)=\int_{K_G}\Phi(gk)\,dk$ then (assuming $dk$ is chosen so that the measure of $K_G$ is $1$) we have
\[
	 A_{-1}^{(m,m-r-1)}(\Phi')(g)=\int_{K_G}A_{-1}^{(m,m-r-1)}(\Phi)(gk)
	=\int_{K_G}B_{-1}^{(m,r)}(\varphi)(gk)=B_{-1}^{(m,r)}(\varphi)(g).
\]
\end{proof}



\section{\bf Spherical Second Term Identity}\label{S:identity}


In this section and the next, we shall derive a certain form of the second term identity
of the Siegel-Weil formula (the weak second term identity) in the second term range; it is an identity between the second terms of the Laurent series expansions of the Siegel Eisenstein
series and the non-Siegel Eisenstein series resulting from the
regularized theta integral on the orthogonal group.
We first derive the spherical second term identity in this section and in the following section, we will extend it to the weak second term identity.


\subsection{Spherical Eisenstein Series}

Recall that at the beginning of the previous section, we defined the Eisenstein series $E^{(m,r)}(g, s; f)$ for $f(-,s)\in I^{m,r}(s)$. In this section, by studying this Eisenstein series when $f$ is spherical, we will derive spherical Siegel-Weil formulas both for the first terms and the second terms. For this purpose, for each $s_0\in\C$ we write the Laurent series expansion as
\[
     E^{(m,r)}(g, s; f)=\sum_{d\gg-\infty}(s-s_0)^dE^{(m,r)}_d(g, s_0; f).
\]
Also for the spherical section $f^0$ with the property that $f^0(1,s)=1$, we simply write
\[
    E^{(m,r)}(g,s):=E^{(m,r)}(g,s;f^0).
\]

Now let $Q=P_{1}$, so that its Levi factor is $\Gm\times\OO_{m-1,m-1}$. We consider the constant term $E_Q^{(m,r)}$ of
the above Eisenstein series $E^{(m,r)}(g,s;f^0)$ along the parabolic $Q$ for the spherical section $f^0$ with the property that $f^0(1,s)=1$. The constant term $E_Q^{(m,r)}$ is an automorphic form on $\Gm(\A)\times\OO_{m-1,m-1}(\A)$ which can be computed as follows. (See also \cite[Appendix. B]{GI}.)

\begin{prop}\label{P:constant_term}
The constant term $E_Q^{(m,r)}$ of the Eisenstein series $E^{(m,r)}$ for the spherical section $f^0$, as an automorphic form on $\Gm(\A)\times\OO_{m-1,m-1}(\A)$, can be expressed as follows.
\begin{enumerate}[$\bullet$]
\item If $1<r<m$, then
\begin{align*}
    E_Q^{(m,r)}((a,g), s)&=|a|^{s+m-\frac{r+1}{2}}E^{(m-1,r-1)}(g,s+\frac{1}{2})\\
    &+|a|^rE^{(m-1,r)}(g,s)\frac{\xi(s+m-\frac{3r+1}{2})}{\xi(s+m-\frac{r+1}{2})}\\
    &+|a|^{-s+m-\frac{r+1}{2}}E^{(m-1,r-1)}(g,s-\frac{1}{2})\frac{\xi(2s)}{\xi(2s+r-1)}\\
    &\qquad\qquad\quad\cdot\frac{\xi(s+\frac{r-1}{2})}{\xi(s+m-\frac{r+1}{2})}
    \frac{\xi(s-m+\frac{3r+1}{2})}{\xi(s+\frac{r+1}{2})}.
\end{align*}

\item If $r=m$, then
\begin{align*}
    E_Q^{(m,m)}((a,g), s)&=|a|^{s+\frac{m-1}{2}}E^{(m-1,m-1)}(g,s+\frac{1}{2})\\
    &+|a|^{-s+\frac{m-1}{2}}E^{(m-1,m-1)}(g,s-\frac{1}{2})\frac{\xi(2s)}{\xi(2s+m-1)}.\\
\end{align*}

\item If $r=1$, then
\begin{align*}
    E_Q^{(m,1)}((a,g), s)&=|a|^{s+m-1}\\
    &+|a|E^{(m-1,1)}(g,s)\frac{\xi(s+m-2)}{\xi(s+m-1)}\\
    &+|a|^{-s+m-1}\frac{\xi(s)}{\xi(s+m-1)}\cdot\frac{\xi(s-m+2)}{\xi(s+1)}.
\end{align*}
\end{enumerate}
Here $\xi(s)$ is the complete normalized zeta function of the number field $F$ with the functional equation $\xi(s)=\xi(1-s)$, and $(a,g)\in\Gm(\A)\times\OO_{m-1,m-1}(\A)$.
\end{prop}

Due to frequent use of those formulas, we simplify the notations by setting
\begin{align*}
    F^{(m,r)}(s)&=\frac{\xi(s+m-\frac{3r+1}{2})}{\xi(s+m-\frac{r+1}{2})}\\
    G^{(m,r)}(s)&=\frac{\xi(2s)}{\xi(2s+r-1)}\frac{\xi(s+\frac{r-1}{2})}{\xi(s+m-\frac{r+1}{2})}
    \frac{\xi(s-m+\frac{3r+1}{2})}{\xi(s+\frac{r+1}{2})}\\
    H^{(m)}(s)&=\frac{\xi(2s)}{\xi(2s+m-1)},
\end{align*}
for the factors containing the zeta functions in the first and second formulas above. Also we write the Laurent series expansion for $F^{(m,r)}(s)$ at $s=s_0$ as
\[
    F^{(m,r)}(s)=\sum_{d\gg-\infty}(s-s_0)^dF_d^{(m,r)}(s_0),
\]
and similarly for $G^{(m,r)}(s)$ and $H^{(m)}(s)$.
Also, throughout this section, we suppress $g$ and $a$ from the notation of the spherical Eisenstein series whenever there is no danger of confusion. Namely for $E^{(m,r)}(g, s)$ and $|a|$ we simply write $E^{(m,r)}(s)$ and $|\,|$ respectively, and for
the Laurent series of $E^{(m,r)}(g, s)$, we write
\[
     E^{(m,r)}(s)=\sum_{d\gg-\infty}(s-s_0)^dE^{(m,r)}_d(s_0).
\]


\subsection{A lemma}

The following elementary lemma will be crucial for our computation and repeatedly used.
\begin{lemma}\label{L:key_lemma}
Let $E_1,\dots, E_k$ and $F_1,\dots, F_l$ be automorphic forms on $\OO_{m,m}(\A)$. Also let $r_1,\cdots r_k$ and $s_1,\dots,s_l$ be real numbers such that all the $r_i$'s are distinct and all the $s_j$'s are distinct but some of the $r_i$'s might be the same as some of the $s_j$'s. Then if
\[
	\sum_{i=1}^{k}|\,|^{r_i}E_i+\sum_{j=1}^{l}|\,|^{s_j}\ln|\,|F_j=0
\]
as an automorphic form on $\Gm(\A)\times\OO_{m,m}(\A)$, then all the $E_i$'s and $F_j$'s are zero.
\end{lemma}
\begin{proof}
Let us fix an embedding $\R^+\subset\A^\times$ by using one of the archimedean places, and view the functions $|\,|^{r_i}$ and $|\,|^{s_j}\ln|\,|$ as functions on $\R^+$. Then those functions are known to be linearly independent over $\C$. Thus the lemma follows.
\end{proof}


\subsection{Spherical first term identity for 1st term range $m\geq 2r+1$}

Using Proposition \ref{P:constant_term}, we will compute the spherical first term identity between the spherical non-Siegel Eisenstein series $E^{(m,r)}(s)$ at $s=\frac{r+1}{2}$ and the spherical Siegel Eisenstein series $E^{(m,m)}(s)$ at $s=\frac{m-2r-1}{2}=-\rho_{m,r}$ in the 1st term range, \ie when $m\geq 2r+1$. It is basically a refinement of  the spherical case of
Proposition \ref{P:first_term}.
\vskip 5pt

\begin{lemma}[Spherical First Term Identity for 1st term range]\label{L:first_B-1}
Assume $r<\frac{m-1}{2}$, \ie $m>2r+1$. Then there exists a non-zero
constant $c_{m,r}$ independent of $g$ and $f$ such that
\[
    E_{-1}^{(m,r)}(\frac{r+1}{2})=c_{m,r}E_{-1}^{(m,m)}(\frac{m-2r-1}{2}),
\]
where $c_{m,r}$ can be explicitly computed as
\[
    c_{m,r}=\prod_{i=0}^{r-1}\frac{\xi(m-2r+i)}{\xi(m-i)}
    \cdot\prod_{i=r+1}^{m-(r+1)}\frac{\xi(2i)}{\xi(i)}.
\]
\end{lemma}
\begin{proof}
The first part of the lemma is immediate from Proposition
\ref{P:first_term}. So let us explicitly compute $c_{m,r}$. First we need to compute $c_{m,1}$. (Let us note that $m\geq 4$ since $m>2r+1$.) Note that the first term identity gives
\[
    E_{-1,Q}^{(m,1)}(1)=c_{m,1}E_{-1,Q}^{(m,m)}(\frac{m-3}{2}),
\]
which together with Proposition \ref{P:constant_term} gives
\begin{align*}
&|\,|E_{-1}^{(m-1,1)}(1)\frac{\xi(m-1)}{\xi(m)}+|\,|^{m-2}\frac{\xi(1)}{\xi(m)}\frac{\xi(3-m)}{\xi(2)}\\
&\quad=c_{m,1}\left(|\,|^{m-2}E_{-1}^{(m-1,m-1)}(\frac{m-2}{2})
+|\,|E_{-1}^{(m-1,m-1)}(\frac{m-4}{2})\frac{\xi(m-3)}{\xi(2m-4)}\right).
\end{align*}
(Here recall our convention that $\xi(1)=\underset{s=1}{\Res}\xi(s)$.)

Now Lemma \ref{L:key_lemma} allows us to ``extract" all the terms containing  $|\,|^{m-2}$. Namely we can equate the second term of the left hand side with the first term of the right hand side. Then we obtain
\[
    \frac{\xi(1)}{\xi(m)}\frac{\xi(3-m)}{\xi(2)}=c_{m,1}E_{-1}^{(m-1,m-1)}(\frac{m-2}{2}).
\]
Notice that this implies that, as an automorphic form, $E_{-1}^{(m-1,m-1)}(\frac{m-2}{2})$ is a constant function. Let
\[
    E_{-1}^{(m,m)}(\frac{m-1}{2})=\lambda_m
\]
so that
\[
    c_{m,1}=\frac{\xi(1)\xi(m-2)}{\xi(2)\xi(m)}\frac{1}{\lambda_{m-1}}.
\]
(Here we used the functional equation $\xi(3-m)=\xi(m-2)$.)
But $\lambda_{m-1}$ can be computed as follows:
\[
    \lambda_m=E_{-1,Q}^{(m,m)}(\frac{m-1}{2})=E_{-1}^{(m-1,m-1)}(\frac{m-2}{2})\frac{\xi(m-1)}{\xi(2m-2)}
    =\lambda_{m-1}\frac{\xi(m-1)}{\xi(2m-2)}.
\]
(To obtain this, we computed $E_{-1,Q}^{(m,m)}(\frac{m-1}{2})$ by Proposition \ref{P:constant_term} and used $E_{-1}^{(m-1,m-1)}(\frac{m}{2})=0$ because $E^{(m-1,m-1)}(s)$ converges at $\frac{m}{2}$.)
This gives
\[
    \lambda_m=\frac{\xi(m-1)\xi(m-2)\cdots\xi(2)}{\xi(2m-2)\xi(2m-4)\cdots\xi(4)}\cdot\lambda_2.
\]
But it is easy to compute $\lambda_2=\frac{\xi(1)}{\xi(2)}$. Hence we get the formula for $c_{m,1}$.

Now let us treat the case $r\geq 2$. First, since $E_{-1}^{(m,r)}(\frac{r+1}{2})=c_{m,r}E_{-1}^{(m,m)}(\frac{m-2r-1}{2})$, we have
$E_{-1, Q}^{(m,r)}(\frac{r+1}{2})=c_{m,r}E_{-1, Q}^{(m,m)}(\frac{m-2r-1}{2})$. Then by Proposition \ref{P:constant_term},
\begin{align*}
    &|\,|^mE_{-1}^{(m-1,r-1)}(\frac{r+2}{2})+|\,|^rE_{-1}^{(m-1,r)}(\frac{r+1}{2})\frac{\xi(m-r)}{\xi(m)}\\
    &\qquad\qquad\qquad+|\,|^{m-r-1}E_{-1}^{(m-1,r-1)}(\frac{r}{2})\frac{\xi(r+1)}{\xi(2r)}\frac{\xi(r)}{\xi(m)}\frac{\xi(2r+1-m)}{\xi(r+1)}\\
    &\quad=c_{m,r}\left(|\,|^{m-r-1}E^{(m-1,m-1)}_{-1}(\frac{m-2r}{2})+|\,|^rE_{-1}^{(m-1,m-1)}(\frac{m-2r-2}{2})\frac{\xi(m-2r-1)}{\xi(2m-2r-2)}\right).
\end{align*}
Note that $m-r-1\neq r$ and $m-r-1\neq m$. Then by lemma \ref{L:key_lemma} we can extract the terms containing $|\,|^{m-r-1}$, and obtain
\[
    E_{-1}^{(m-1,r-1)}(\frac{r}{2})\frac{\xi(m-2r)}{\xi(m)}\frac{\xi(r)}{\xi(2r)}
    =c_{m,r}E^{(m-1,m-1)}_{-1}(\frac{m-2r}{2}).
\]
(Here note that we used the functional equation $\xi(2r+1-m)=\xi(m-2r)$.) Also by definition of $c_{m,r}$ we have
\[
    E_{-1}^{(m-1,r-1)}(\frac{r}{2})=c_{m-1,r-1}E_{-1}^{(m-1,m-1)}(\frac{m-2r}{2}).
\]
Those two together give the recursive relation
\[
    c_{m,r}=c_{m-1,r-1}\cdot\frac{\xi(m-2r)}{\xi(m)}\frac{\xi(r)}{\xi(2r)},
\]
which gives
\[
    c_{m,r}=\prod_{i=0}^{r-2}\frac{\xi(m-2r+i)}{\xi(m-i)}\cdot\prod_{i=2}^{r}\frac{\xi(i)}{\xi(2i)}\cdot c_{m-(r-1),1}.
\]
Hence the formula for $c_{m,1}$ immediately gives the formula for $c_{m,r}$.
\end{proof}

Next we will derive the first term identity on the boundary. First, we need the following lemma, which we will frequently use later.

\begin{lemma}\label{L:no_pole}
\quad
\begin{enumerate}[(i)]
\item $E^{(m,r)}_{-2}(s_0)=0$ for $s_0\geq\frac{r+2}{2}$.
\item $E^{(2r+1,r)}_{-2}(\frac{r+1}{2})=0$.
\item $E^{(m,r)}_{-1}(\frac{r+3}{2})=0$.
\end{enumerate}
\end{lemma}
\begin{proof}
\quad
\begin{enumerate}[(i)]
\item We show this by induction on $r$. So first let us assume that $r=1$.
This can be shown by induction on $m$ by using the third equation of
Proposition \ref{P:constant_term}. So assume the lemma holds for
some $r$. Then by the first equation of Proposition
\ref{P:constant_term} together with the induction hypothesis, we get
\begin{align*}
    E^{(m,r+1)}_{-2,Q}(s_0)&=|\,|^{r+1}E^{(m-1,r+1)}_{-2}(s_0)
                    \frac{\xi(s_0+m-\frac{3r+4}{2})}{\zeta(s_0+m-\frac{r+1}{2})}\\
    E^{(m-1,r+1)}_{-2,Q}(s_0)&=|\,|^{r+1}E^{(m-2,r+1)}_{-2}(s_0)
                    \frac{\xi(s_0+m-1-\frac{3r+4}{2})}{\zeta(s_0+m-1-\frac{r+1}{2})}\\
    E^{(m-2,r+1)}_{-2,Q}(s_0)&=|\,|^{r+1}E^{(m-3,r+1)}_{-2}(s_0)
                    \frac{\xi(s_0+m-2-\frac{3r+4}{2})}{\zeta(s_0+m-2-\frac{r+1}{2})}\\
                            &\cdot\\
                            &\cdot\\
                            &\cdot\\
    E^{(r+2,r+1)}_{-2,Q}(s_0)&=|\,|^{r+1}E^{(r+1,r+1)}_{-2}(s_0)
                    \frac{\xi(s_0+r+2-\frac{3r+4}{2})}{\zeta(s_0+r+2-\frac{r+1}{2})}.
\end{align*}
Notice that in the last equation, $E^{(r+1,r+1)}_{-2}(s_0)=0$
because the spherical Siegel Eisenstein series never has a double
pole. (Also note that $\xi(s+r+2-\frac{3r+4}{2})$ is holomorphic at
$s_0\geq\frac{r+3}{2}$.) So $E^{(r+2,r+1)}_{-2,Q}(s_0)=0$, \ie
$E^{(r+2,r+1)}_{-2}(s_0)=0$. Hence $E^{(r+3,r+1)}_{-2,Q}(s_0)=0$,
\ie $E^{(r+3,r+1)}_{-2}(s_0)=0$ etc, and we get $E^{(m,r+1)}_{-2}(s_0)=0$. Thus the induction is
complete and (i) has been proven.\\

\item By proposition \ref{P:constant_term}, we have
\begin{align*}
    E_{-2, Q}^{(2r+2,r)}(\frac{r+1}{2})&=|\,|^{2r+2}E_{-2}^{(2r+1,r-1)}(\frac{r+2}{2})
    +|\,|^rE_{-2}^{(2r+1,r)}(\frac{r+1}{2})F^{(2r+1,r)}(\frac{r+1}{2})\\
    &\qquad\qquad+|\,|^rE_{-2}^{(2r+1,r-1)}(\frac{r}{2})G^{(2r+2,r)}(\frac{r+1}{2}).
\end{align*}
But $E_{-2, Q}^{(2r+2,r)}(\frac{r+1}{2}), E_{-2}^{(2r+1,r-1)}(\frac{r+2}{2})$ and $E_{-2}^{(2r+1,r-1)}(\frac{r}{2})$ are all zero because, in general, the non-Siegel Eisenstein series $E^{(m,r)}(s)$ does not have a double pole when $m>2r+1$. Hence $E_{-2}^{(2r+1,r)}(\frac{r+1}{2})=0$.\\

\item This can be proven in a similar way as (i) by
induction. But this time, we use the fact that the Siegel Eisenstein
series $E^{(r+1,r+1)}(s)$ does not have a pole at $s=\frac{r+4}{2}$. The detail is left to the reader.
\end{enumerate}
\end{proof}

Then we have

\begin{prop}[Spherical First Term Identity on the boundary]\label{P:first_A0}
Assume $r=\frac{m-1}{2}$, \ie $m=2r+1$. Then
\[
    E_{-1}^{(2r+1,r)}(\frac{r+1}{2})=c_{r}E_{0}^{(2r+1,2r+1)}(0),
\]
where
\[
    c_{r}=c_{2r+2,r}\frac{\xi(1)}{2\xi(r+2)}
    =\frac{1}{2}\cdot\prod_{i=1}^{r+1}\frac{\xi(i)}{\xi(r+i)}.
\]
(Recall $\xi(1)=\underset{s=1}{\Res}\xi(s)$.)
\end{prop}
\begin{proof}
First assume $r>1$. The first equation of Proposition
\ref{P:constant_term} gives us
\begin{align}\label{E:E(2r+2,r)}
    E_{-1,Q}^{(2r+2,r)}(\frac{r+1}{2})&=|\,|^{2r+2}E_{-1}^{(2r+1,r-1)}(\frac{r+2}{2})
    +|\,|^rE_{-1}^{(2r+1,r)}(\frac{r+1}{2})F^{(2r+2,r)}(\frac{r+1}{2})\notag\\
    &\qquad\qquad+|\,|^{r+1}E_{-1}^{(2r+1,r-1)}(\frac{r}{2})G^{(2r+2,r)}(\frac{r+1}{2}).
\end{align}
(To compute this, we used Lemma \ref{L:no_pole} (ii).)
Also Lemma \ref{L:first_B-1} gives us
\[
    E_{-1, Q}^{(2r+2,r)}(\frac{r+1}{2})=c_{2r+2,r}E_{-1,
    Q}^{(2r+2,2r+2)}(\frac{1}{2}),
\]
which, combined with the second equation of Proposition
\ref{P:constant_term}, gives
\begin{align}\label{E:(2r+2,r)2}
    &E_{-1, Q}^{(2r+2,r)}(\frac{r+1}{2})\notag\\
    &\qquad\qquad=c_{2r+2,r}\left(|\,|^{r+1}E_{-1}^{(2r+1,2r+1)}(1)
    +|\,|^rE_0^{(2r+1,2r+1)}(0)H_{-1}^{(2r+1,2r+1)}\right).
\end{align}
Therefore by equating (\ref{E:E(2r+2,r)}) and (\ref{E:(2r+2,r)2}), and (again by Lemma \ref{L:key_lemma}) looking at the terms containing $|\,|^r$,
we obtain
\[
    E_{-1}^{(2r+1,r)}(\frac{r+1}{2})\frac{\xi(r+2)}{\xi(2r+2)}=
    c_{2r+2,r}E_0^{(2r+1,2r+1)}(0)\frac{\xi(1)}{2\xi(2r+2)},
\]
namely
\[
    E_{-1}^{(2r+1,r)}(\frac{r+1}{2})=
    c_{2r+2,r}E_0^{(2r+1,2r+1)}(0)\frac{\xi(1)}{2\xi(r+2)},
\]
where
\[
    c_{r}=c_{2r+2,r}\frac{\xi(1)}{2\xi(r+2)},
\]
and Lemma \ref{L:first_B-1} gives the explicit expression for
$c_{r}$ as in the proposition.

If $r=1$, the same identity can be obtained by more direct
computation, and is done in \cite{GI}.
\end{proof}


\subsection{Identities in the second term range: Idea of the proof}

\vskip 10pt

Next we consider the 2nd term range $r+1\leq m\leq 2r$. In this range, the non-Siegel Eisenstein series
$E^{(m,r)}(g, s; f)$ can have a double pole, and indeed as we will show below, it does have a double pole for a spherical $f$.
In this case, we shall establish not only a spherical first term identity but also a spherical second term identity. Both of these follow from the same idea which we have already exploited in the proof of
the first term identities in the first term range (Lemma \ref{L:first_B-1}) and on the boundary (Proposition \ref{P:first_A0}). We shall presently give a brief description of the idea before plunging into the details of the proof.

\vskip 10pt

Roughly speaking, our proof of the identities for $(\OO_{m,m}, \Sp_{2r})$ is based on induction on the quantity
\[  j= 2r- m. \]
Observe that one has
\begin{align}
 &j < -1 \Longleftrightarrow  \text{first term range;} \notag \\
 &j =-1 \Longleftrightarrow  \text{boundary;} \notag \\
&j > -1 \Longleftrightarrow  \text{second term range.} \notag
\end{align}
We have shown the spherical first term identities when $j \leq -1$ and let us see how we can deduce the spherical first and second term identities for $j = 0$ from the case $j = -1$.
\vskip 10pt

We start with the first term identity on the boundary (\ie when $m = 2r+1$ and $j = -1$), which is given in Proposition \ref{P:first_A0} and Lemma \ref{L:no_pole} (ii):
\begin{align*}
    E_{-2}^{(2r+1,r)}(\frac{r+1}{2})&=0,\\
    E_{-1}^{(2r+1,r)}(\frac{r+1}{2})&=c_rE_{0}^{(2r+1,2r+1)}(0).
\end{align*}
These are identities of automorphic forms on  the group $\OO_{2r+1,2r+1}(\A)$.
 By considering the constant term along $Q$, we have
\begin{align*}
    E_{-2, Q}^{(2r+1,r)}(\frac{r+1}{2})&=0,\\
    E_{-1, Q}^{(2r+1,r)}(\frac{r+1}{2})&=c_rE_{0, Q}^{(2r+1,2r+1)}(0).
\end{align*}
Now using Proposition \ref{P:constant_term}, we can compute both sides of these two equations in terms of Eisenstein series on the lower rank group $\OO_{2r,2r}(\A)$. This results in certain relations between Siegel and non-Siegel Eisenstein series on $\OO_{2r,2r}(\A)$. The relation that follows from the first equation is simply the spherical first term identity for $m=2r$ (i.e. $j = 0$), and the one that follows from the second equation gives the spherical second term identity for $m=2r$ (i.e. $j= 0$).
\vskip 10pt

Using $m=2r$ as the base case, we will then show the spherical second term identity for the remaining cases $r+1\leq m< 2r$ by induction.   For the first term identity, we further prove the case $m=2r-1$ (i.e. $j= 1$) by a similar method, and then show the remaining cases $r+1\leq m< 2r-1$ by induction by using the $m=2r-1$ case as a base step.
\vskip 10pt

While the above idea of deriving the spherical second term identity on $\OO_{m,m}$ from the first term identity of   $\OO_{m+1, m+1}$ is quite simple and elegant, its execution requires some rather lengthy detailed computations. We apologize for not being able to package these computations more elegantly.
\vskip 10pt


\subsection{Spherical first term identities in 2nd term range $r<m\leq 2r$.}


 Let us first consider the first term identity.
\vskip 5pt

\begin{prop}[Spherical First Term Identity for 2nd term range]\label{P:first_B-2}
Assume that we are in the 2nd term range $\frac{m-1}{2}\leq r\leq m-1$, \ie $r+1\leq m\leq 2r$. Then
\[
    E_{-2}^{(m,r)}(\frac{r+1}{2})=d_{m,r}E_{-1}^{(m,m)}(\frac{2r-m+1}{2}),
\]
in which
\begin{align*}
    d_{2r,r}&=c_{2r,r-1}\frac{\xi(1)}{\xi(r)}
    \frac{\xi(2r)}{\xi(r+1)}\\
    d_{m-1,r}&=d_{m,r}\frac{\xi(m)}{\xi(m-r)},
    \text{\quad for\quad } m\leq 2r-1
\end{align*}
(Recall that $\xi(1)=\underset{s=1}{\Res}\xi(s)$.)
\end{prop}
\begin{proof}
Let us note that the case $r=1$ and $m=2$ is essentially done  in \cite[Appendix. B]{GI}.
So we assume that $r>1$.
\vskip 10pt

\noindent{\bf \underline{The case $m = 2r$}}
\vskip 10pt

As we mentioned above, we first prove the case $m=2r$. Namely from the first formula of Proposition \ref{P:constant_term}, we have
\begin{align*}
    E_{-2, Q}^{(2r+1,r)}(\frac{r+1}{2})&=|\,|^{2r+1}E_{-2}^{(2r,r-1)}(\frac{r+2}{2})
    +|\,|^rE_{-2}^{(2r,r)}(\frac{r+1}{2})F^{(2r+1,r)}(\frac{r+1}{2})\\
    &\qquad\qquad+|\,|^{r}E_{-1}^{(2r,r-1)}(\frac{r}{2})G_{-1}^{(2r+1,r)}(\frac{r+1}{2}).
\end{align*}
But as we mentioned at the beginning of the section,
\[
    E_{-2, Q}^{(2r+1,r)}(\frac{r+1}{2})=0
\]
by Lemma \ref{L:no_pole} (ii). Then by extracting terms having $|\,|^r$ (Lemma \ref{L:key_lemma}) we have
\[
    0=E_{-2}^{(2r,r)}(\frac{r+1}{2})F^{(2r+1,r)}(\frac{r+1}{2})
    +E_{-1}^{(2r,r-1)}(\frac{r}{2})G_{-1}^{(2r+1,r)},
\]
from which, by computing $F^{(2r+1,r)}(\frac{r+1}{2})$ and $G_{-1}^{(2r+1,r)}$ explicitly, we obtain
\begin{align*}
    E_{-2}^{(2r,r)}(\frac{r+1}{2})\frac{\xi(r+1)}{\xi(2r+1)}
    &=-E_{-1}^{(2r,r-1)}(\frac{r}{2})\frac{\xi(r+1)}{\xi(2r)}
    \frac{\xi(r)}{\xi(2r+1)}\frac{\xi(0)}{\xi(r+1)}\\
    &=-c_{2r,r-1}E^{(2r,2r)}_{-1}(\frac{1}{2})\frac{\xi(0)}{\xi(r)}
    \frac{\xi(2r)}{\xi(2r+1)}\quad (\text{by Lemma \ref{L:first_B-1}}),
\end{align*}
which gives
\[
    E_{-2}^{(2r,r)}(\frac{r+1}{2})=-c_{2r,r-1}\frac{\xi(0)}{\xi(r)}
    \frac{\xi(2r)}{\xi(r+1)}E^{(2r,2r)}_{-1}(\frac{1}{2}).
\]
(Here recall our convention that $\xi(0)=\underset{s=0}{\Res}\xi(s)$.) We set
\begin{align*}
    d_{2r,r}&=-c_{2r,r-1}\frac{\xi(0)}{\xi(r)}
    \frac{\xi(2r)}{\xi(r+1)}\\
    &=c_{2r,r-1}\frac{\xi(1)}{\xi(r)}
    \frac{\xi(2r)}{\xi(r+1)}\quad\text{by $\xi(0)=-\xi(1)$}.
\end{align*}
Hence the case $m=2r$ is proven.
\vskip 10pt

\noindent{\bf \underline{Base step of the induction: $m = 2r-1$}}
\vskip 10pt

Next we show the case $m\leq 2r-1$. Recall
\[
	m=2r-j\quad\text{ where }\quad 1 \leq j \leq r-1.
\]
Then we show this by induction on $j$. For the base step, consider
\begin{align}\label{E:E_{-2,Q}^{(2r,r)}}
    E_{-2, Q}^{(2r,r)}(\frac{r+1}{2})&=|\,|^{2r}E_{-2}^{(2r-1,r-1)}(\frac{r+2}{2})
    +|\,|^rE_{-2}^{(2r-1,r)}(\frac{r+1}{2})F^{(2r,r)}(\frac{r+1}{2})\notag\\
    &\qquad\qquad+|\,|^{r-1}E_{-1}^{(2r-1,r-1)}(\frac{r}{2})G_{-1}^{(2r,r)}(\frac{r+1}{2}).
\end{align}
We know that
\begin{align*}
    &E_{-2,Q}^{(2r,r)}(\frac{r+1}{2})=d_{2r,r}E_{-1,Q}^{(2r,2r)}(\frac{1}{2})\\
    &E_{-2}^{(2r-1,r-1)}(\frac{r+2}{2})=0 \quad(\text{Lemma \ref{L:no_pole} (i)})\\
    &E_{-1}^{(2r-1,r-1)}(\frac{r}{2})=c_{r-1}E_{0}^{(2r-1,2r-1)}(0).
\end{align*}
Hence (\ref{E:E_{-2,Q}^{(2r,r)}}) becomes
\begin{align}\label{E:2r_r}
    c_{2r,r-1}E_{-1,Q}^{(2r,2r)}(\frac{1}{2})
    &=|\,|^rE_{-2}^{(2r-1,r)}(\frac{r+1}{2})F^{(2r,r)}(\frac{r+1}{2})\notag\\
    &\qquad\qquad
    +|\,|^{r-1}c_{r-1}E_{0}^{(2r-1,2r-1)}(0)G_{-1}^{(2r,r)}(\frac{r+1}{2}).
\end{align}
Now by the second equation of Proposition \ref{P:constant_term}, we
have
\begin{equation}\label{E:2r_2r}
    E_{-1,Q}^{(2r,2r)}(\frac{1}{2})=|\,|^{r}E_{-1}^{(2r-1,2r-1)}(1)
    +|\,|^{r-1}E_{0}^{(2r-1,2r-1)}(0)H_{-1}^{(2r)}(\frac{1}{2}).
\end{equation}
Then by (\ref{E:2r_r}) and (\ref{E:2r_2r}), we have
\begin{align}\label{E:2r-1_r}
 &d_{2r,r}\left(|\,|^{r}E_{-1}^{(2r-1,2r-1)}(1)
    +|\,|^{r-1}E_{0}^{(2r-1,2r-1)}(0)H_{-1}^{(2r-1)}(\frac{1}{2})\right)\notag\\
 &=|\,|^rE_{-2}^{(2r-1,r)}(\frac{r+1}{2})F^{(2r,r)}(\frac{r+1}{2})
 +|\,|^{r-1}c_{r-1}E_{0}^{(2r-1,2r-1)}(0)G_{-1}^{(2r,r)}(\frac{r+1}{2}).
\end{align}
By extracting the terms containing $|\,|^r$ in (\ref{E:2r-1_r}) (Lemma \ref{L:key_lemma}), we obtain
\[
    d_{2r,r}E_{-1}^{(2r-1,2r-1)}(1)
    =E_{-2}^{(2r-1,r)}(\frac{r+1}{2})F^{(2r,r)}(\frac{r+1}{2}),
\]
\ie
\[
    E_{-2}^{(2r-1,r)}(\frac{r+1}{2})=d_{2r,r}F^{(2r,r)}(\frac{r+1}{2})E_{-1}^{(2r-1,2r-1)}(1),
\]
and
\begin{align*}
    d_{2r-1,r}&=d_{2r,r}\frac{1}{F^{(2r,r)}(\frac{r+1}{2})}\\
    &=d_{2r,r}\frac{\xi(2r)}{\xi(r)}.
\end{align*}
This shows the base step.
\vskip 10pt

\noindent{\bf \underline{The induction step}}
\vskip 10pt

Now consider the induction step. First note that for $j\geq1$,
\begin{align}\label{E:2r-j_r}
    E_{-2, Q}^{(2r-j,r)}(\frac{r+1}{2})&=|\,|^{2r-j}E_{-2}^{(2r-j-1,r-1)}(\frac{r+2}{2})
    +|\,|^rE_{-2}^{(2r-j-1,r)}(\frac{r+1}{2})F^{(2r-j,r)}(\frac{r+1}{2})\notag\\
    &\qquad\qquad+|\,|^{r-j-1}E_{-2}^{(2r-j-1,r-1)}(\frac{r}{2})G^{(2r-j,r)}(\frac{r+1}{2})\notag\\
    &=|\,|^rE_{-2}^{(2r-j-1,r)}(\frac{r+1}{2})F^{(2r-j,r)}(\frac{r+1}{2})\notag\\
    &\qquad\qquad+|\,|^{r-j-1}E_{-2}^{(2r-j-1,r-1)}(\frac{r}{2})G^{(2r-j,r)}(\frac{r+1}{2}),
\end{align}
where we used $E_{-2}^{(2r-j-1,r-1)}(\frac{r+2}{2})=0$ (Lemma
\ref{L:no_pole} (i)). Also we have
\begin{align}\label{E:2r-j_2r-j}
    E^{(2r-j,2r-j)}_{-1,Q}(\frac{j+1}{2})
    &=|\,|^rE_{-1}^{(2r-j-1,2r-j-1)}(\frac{j+2}{2})\notag\\
    &\qquad+|\,|^{r-j-1}E_{-1}^{(2r-j-1,2r-j-1)}(\frac{j}{2})H^{(2r-j)}(\frac{j+1}{2}).
\end{align}
Now by the induction hypothesis
\[
    E_{-2,Q}^{(2r-j,r)}(\frac{r+1}{2})=d_{2r-j,r}E^{(2r-j,2r-j)}_{-1,Q}(\frac{j+1}{2}),
\]
together with (\ref{E:2r-j_r}) and (\ref{E:2r-j_2r-j}) above, we obtain
\begin{align*}
    &|\,|^rE_{-2}^{(2r-j-1,r)}(\frac{r+1}{2})F^{(2r-j,r)}(\frac{r+1}{2})
    +|\,|^{r-j}E_{-2}^{(2r-j,r-1)}(\frac{r}{2})G^{(2r-j,r)}(\frac{r+1}{2})\\
    &=d_{2r-j,r}\left(|\,|^rE_{-1}^{(2r-j-1,2r-j-1)}(\frac{j+2}{2})
    +|\,|^{r-j-1}E_{-1}^{(2r-j-1,2r-j-1)}(\frac{j}{2})H^{(2r-j)}(\frac{j+1}{2})\right).
\end{align*}
By extracting the terms containing $|\,|^r$ from both sides (Lemma \ref{L:key_lemma}), we obtain
\[
    E_{-2}^{(2r-j-1,r)}(\frac{r+1}{2})F^{(2r-j,r)}(\frac{r+1}{2})
    =d_{2r-j,r}E_{-1}^{(2r-j-1,2r-j-1)}(\frac{j+2}{2}),
\]
\ie, by computing $F^{(2r-j,r)}(\frac{r+1}{2})$ explicitly,
\[
    E_{-2}^{(2r-j-1,r)}(\frac{r+1}{2})=\frac{\xi(2r-j)}{\xi(r-j)}
    d_{2r-j,r}E_{-1}^{(2r-j-1,2r-j-1)}(\frac{j+2}{2}),
\]
and
\begin{align*}
    d_{2r-j-1,r}&=\frac{\xi(2r-j)}{\xi(r-j)}d_{2r-j,r},
    \text{\quad for\quad } j\geq 1.
\end{align*}
This competes the proof.
\end{proof}
\vskip 10pt

\subsection{Spherical second term identities for 2nd term range $r<m\leq 2r$}

Next we prove the spherical second term identity. As we mentioned at the beginning of the section, we prove the case $m=2r$ as a base step and derive the rest by induction. Namely, first we prove
\vskip 10pt

\begin{prop}[Spherical Second Term Identity for $m=2r$]\label{P:Spherical_Second_Term}
The following identity holds for the spherical Eisenstein series:
\begin{align*}
    &E_{-1}^{(2r,r)}(\frac{r+1}{2})F_0^{(2r+1,r)}(\frac{r+1}{2})+E^{(2r,r-1)}_0(\frac{r}{2})G_{-1}^{(2r+1,r)}(\frac{r+1}{2})\\
    &\qquad\qquad\qquad\qquad\qquad\qquad\qquad
    =2c_rE_{0}^{(2r,2r)}(\frac{1}{2})+\gamma_0E_{-1}^{(2r,2r)}(\frac{1}{2}).
\end{align*}
where $\gamma_0$ is some constant which depends only on $r$.
\end{prop}
\begin{proof}
Recall as we mentioned at the beginning, we start with
\begin{align*}
    E_{-1, Q}^{(2r+1,r)}(\frac{r+1}{2})&=c_rE_{0, Q}^{(2r+1,2r+1)}(0),
\end{align*}
which is obtained by taking the constant term along $Q$ of the spherical first term identity on the boundary (Proposition \ref{P:first_A0}). Also by Proposition \ref{P:constant_term}, we have
\begin{align*}
    E_Q^{(2r+1,r)}(s)&=|\,|^{s+\frac{3r+1}{2}}E^{(2r,r-1)}(s+\frac{1}{2})
    +|\,|^rE^{(2r,r)}(s)F^{(2r+1,r)}(s)\\
    &\qquad\qquad+|\,|^{-s+\frac{3r+1}{2}}E^{(2r,r-1)}(s-\frac{1}{2})G^{(2r+1,r)}(s).
\end{align*}
Now we will compute the residue of both sides at $\frac{r+1}{2}$.
Note that
\begin{align*}
    &E^{(2r,r-1)}_{-1}(\frac{r+1}{2}+\frac{1}{2})=0 \quad (\text{Lemma \ref{L:no_pole} (iii)}),\\
    &E^{(2r,r)}_{-2}(\frac{r+1}{2})\neq 0,\\
    &E^{(2r,r-1)}_{-1}((\frac{r+1}{2}-\frac{1}{2}))\neq 0,\\
    &G^{(2r+1,r)}_{-1}(s)\neq 0.
\end{align*}
By taking all these into account, we obtain
\begin{align}\label{E:2r+1_r_second}
    &E_{-1, Q}^{(2r+1,r)}(\frac{r+1}{2})\notag\\
    &=|\,|^rE_{-2}^{(2r,r)}(\frac{r+1}{2})F_1^{(2r+1,r)}(\frac{r+1}{2})
    +|\,|^rE_{-1}^{(2r,r)}(\frac{r+1}{2})F_0^{(2r+1,r)}(\frac{r+1}{2})\notag\\
    &+|\,|^{r}E^{(2r,r-1)}_{-1}(\frac{r}{2})G_{0}^{(2r+1,r)}(\frac{r+1}{2})
    +|\,|^{r}E^{(2r,r-1)}_0(\frac{r}{2})G_{-1}^{(2r+1,r)}(\frac{r+1}{2})\notag\\
    &-|\,|^{r}\ln|\,|
    E^{(2r,r-1)}_{-1}(\frac{r}{2})G_{-1}^{(2r+1,r)}(\frac{r+1}{2}).
\end{align}

Next we will compute $E_{0, Q}^{(2r+1,2r+1)}(0)$ in terms of the Siegel Eisenstein series for the lower rank group. Namely by Proposition
\ref{P:constant_term}, we have
\begin{align*}
    E_{Q}^{(2r+1,2r+1)}(s)&=|\,|^{s+r}E^{(2r,2r)}(s+\frac{1}{2})
    +|\,|^{-s+r}E^{(2r,2r)}(s-\frac{1}{2})H^{(2r+1)}(s).
\end{align*}
Then we have
\begin{align*}
    E_{0,
    Q}^{(2r+1,2r+1)}(0)&=|\,|^rE_{0}^{(2r,2r)}(\frac{1}{2})+|\,|^r\ln|\,|E_{-1}^{(2r,2r)}(\frac{1}{2})\\
    &\qquad+|\,|^rE_0^{(2r,2r)}(-\frac{1}{2})H_{0}^{(2r+1)}(0)+|\,|^rE_{1}^{(2r,2r)}(-\frac{1}{2})H_{-1}^{(2r+1)}(0)\\
    &\qquad-|\,|^r\ln|\,|E_{0}^{(2r,2r)}(-\frac{1}{2})H_{-1}^{(2r+1)}(0).
\end{align*}
Now recall the functional equation of the Siegel-Eisenstein series:
\[
    E^{(2r,2r)}(s)=\beta_{2r}(s)E^{(2r,2r)}(-s),
\]
where, in general, $\beta_m(s)$ is given by
\[
    \beta_m(s)=\prod_{i=0}^{m-2}\frac{\xi(2s-i)}{\xi(2s+m-1-2i)}.
\]

Then at $s=-\frac{1}{2}$, $\beta_{2r}(s)$ has a zero of order one
and $E^{(2r,2r)}(s)$ has a simple pole. Accordingly,
$E^{(2r,2r)}(-s)$ does not have a pole or zero at $s=\frac{1}{2}$.
Hence
\begin{align*}
    E_0^{(2r,2r)}(-\frac{1}{2})&=-\beta_{2r, 1}(-\frac{1}{2})E_{-1}^{(2r,2r)}(\frac{1}{2})\\
    E_1^{(2r,2r)}(-\frac{1}{2})&=\beta_{2r, 1}(-\frac{1}{2})E_{0}^{(2r,2r)}(\frac{1}{2})
            -\beta_{2r, 2}(-\frac{1}{2})E_{-1}^{(2r,2r)}(\frac{1}{2}),
\end{align*}
where, as we have been doing for other meromorphic functions, we write the Laurent expansion of $\beta_{2r}(s)$ at
$s_0$ by
\[
    \beta_{2r}(s)=\sum_{d\gg-\infty}(s-s_0)^d\beta_{2r, d}(s_0).
\]
So we have
\begin{align}\label{E:2r+1_2r+1_second}
    E_{0,Q}^{(2r+1,2r+1)}(0)\notag
    &=|\,|^rE_{0}^{(2r,2r)}(\frac{1}{2})+|\,|^r\ln|\,|E_{-1}^{(2r,2r)}(\frac{1}{2})
    -|\,|^r\beta_{2r,1}(-\frac{1}{2})E_{-1}^{(2r,2r)}(\frac{1}{2})H_{0}^{(2r+1)}(0)\notag\\
    &+|\,|^r\left(\beta_{2r,1}(-\frac{1}{2})E_{0}^{(2r,2r)}(\frac{1}{2})-\beta_{2r,2}(-\frac{1}{2})
    E_{-1}^{(2r,2r)}(\frac{1}{2})\right)H_{-1}^{(2r+1)}(0)\notag\\
    &-|\,|^r\ln|\,|E_{0}^{(2r,2r)}(-\frac{1}{2})H_{-1}^{(2r+1)}(0).
\end{align}

Now by
\begin{align*}
    E_{-1, Q}^{(2r+1,r)}(\frac{r+1}{2})&=c_rE_{0, Q}^{(2r+1,2r+1)}(0)
    \quad(\text{Proposition \ref{P:first_A0}}),
\end{align*}
together with (\ref{E:2r+1_r_second}) and (\ref{E:2r+1_2r+1_second}), we obtain the following by extracting the terms having $|\,|^r$ (Lemma \ref{L:key_lemma}):
\begin{align*}
    &E_{-2}^{(2r,r)}(\frac{r+1}{2})F_1^{(2+1,r)}(\frac{r+1}{2})
    +E_{-1}^{(2r,r)}(\frac{r+1}{2})F_0^{(2r+r,r)}(\frac{r+1}{2})\\
    &\qquad\qquad+E^{(2r,r-1)}_{-1}(\frac{r}{2})G_{0}^{(2r+1,r)}(\frac{r+1}{2})
    +E^{(2r,r-1)}_0(\frac{r}{2})G_{-1}^{(2r+1,r)}(\frac{r+1}{2})\\
    &=c_r\Big[E_{0}^{(2r,2r)}(\frac{1}{2})-\beta_{2r,1}(-\frac{1}{2})E_{-1}^{(2r,2r)}(\frac{1}{2})H_{0}^{(2r+1)}(0)\\
    &\qquad\qquad+\left(\beta_{2r,1}(-\frac{1}{2})E_{0}^{(2r,2r)}(\frac{1}{2})-\beta_{2r,2}(-\frac{1}{2})
    E_{-1}^{(2r,2r)}(\frac{1}{2})\right)H_{-1}^{(2r+1)}(0)\Big]\\
    &=c_r\left(1+\beta_{2r,1}(-\frac{1}{2})H_{-1}^{(2r+1)}(0)\right)E_{0}^{(2r,2r)}(\frac{1}{2})\\
    &\qquad\qquad+c_r\left(-\beta_{2r,1}(-\frac{1}{2})-\beta_{2r,2}(-\frac{1}{2})H_{-1}^{(2r+1)}(0)\right)E_{-1}^{(2r,2r)}(\frac{1}{2}),
\end{align*}
which is the first form of our spherical second term identity.

But the first term identities (Lemma \ref{L:first_B-1} and
Proposition \ref{P:first_B-2}) imply
\begin{align*}
    E^{(2r,r-1)}_{-1}(\frac{r}{2})&=c_{2r,r-1}E^{(2r,2r)}_{-1}(\frac{1}{2})\\
    E_{-2}^{(2r,r)}(\frac{r+1}{2})&=d_{2r,r}E^{(2r,2r)}_{-1}(\frac{1}{2}).
\end{align*}
Thus our second term identity is now written as
\begin{align*}
    &E_{-1}^{(2r,r)}(\frac{r+1}{2})F_0^{(2r+1,r)}(\frac{r+1}{2})+E^{(2r,r-1)}_0(\frac{r}{2})G_{-1}^{(2r+1,r)}(\frac{r+1}{2})\\
    &\quad=c_r\left(1+\beta_{2r,1}(-\frac{1}{2})H_{-1}^{(2r+1,2r+1)}(0)\right)E_{0}^{(2r)}(\frac{1}{2})\\
    &\quad\quad+c_r\bigg[-\beta_{2r,1}(-\frac{1}{2})-\beta_{2r,2}(-\frac{1}{2})H_{-1}^{(2r+1)}(0)\\
    &\quad\quad-d_{2r,r}F_1^{(2r+1,r)}(\frac{r+1}{2})-c_{2r,r-1}G_{0}^{(2r+1,r)}(\frac{r+1}{2})\bigg]E_{-1}^{(2r,2r)}(\frac{1}{2}),
\end{align*}
which we write as
\begin{align*}
    &E_{-1}^{(2r,r)}(\frac{r+1}{2})F_0^{(2r+1,r)}(\frac{r+1}{2})+E^{(2r,r-1)}_0(\frac{r}{2})G_{-1}^{(2r+1,r)}(\frac{r+1}{2})\\
    &\quad\quad\quad=c_r\left(1+\beta_{2r,1}(-\frac{1}{2})H_{-1}^{(2r+1)}(0)\right)E_{0}^{(2r,2r)}(\frac{1}{2})+
    \gamma_0E_{-1}^{(2r,2r)}(\frac{1}{2}),
\end{align*}
where $\gamma_0$ is a constant independent of $f$ and $g$.

Finally, we explicitly compute the coefficient of the first term of the right hand side as
\[
    1+\beta_{2r,1}(-\frac{1}{2})H_{-1}^{(2r+1)}(0)=2.
\]
To see this, first of all, we have
\[
    H_{-1}^{(2r+1)}(0)=\frac{\xi(0)}{2\xi(2r)}.
\]
(Recall our convention $\xi(0)=\underset{s=0}{\Res}\xi(s)$.) Secondly,
notice that
\[
    \beta_{2r}(s)=\frac{1}{\xi(2s+1)}\cdot{\prod_{i=1}^{2r-2}\xi(2s-i)}
    {\prod_{i=0}^{r-2}\xi(2s+2r-1-2i)\prod_{i=r}^{2r-2}\xi(2s+2r-1-2i)},
\]
where the numerator of this fraction does not have a pole or zero at $s=-\frac{1}{2}$. So
\[
    \beta_{2r,1}(-\frac{1}{2})=\frac{2}{\xi(0)}\cdot{\prod_{i=1}^{2r-2}\xi(-1-i)}
    {\prod_{i=0}^{r-2}\xi(-1+2r-1-2i)\prod_{i=r}^{2r-2}\xi(-1+2r-1-2i)}.
\]
Recall that our $\xi(s)$ is normalized so that $\xi(s)=\xi(1-s)$.
Then we see most of the $\xi$'s get canceled out and we obtain
\[
    \beta_{2r,1}(\frac{1}{2})=\frac{2}{\xi(0)}\xi(2r).
\]
So
\[
    1+\beta_{2r, 1}(-\frac{1}{2})H_{-1}^{(2r+1)}(0)
    =1+\frac{2}{\xi(0)}\xi(2r)\cdot\frac{\xi(0)}{2\xi(2r)}=2.
\]

Hence we obtain the spherical second term identity for the case $m=2r$.
\end{proof}
\vskip 10pt

Now using this second term identity as our base step, we will show the following more general second term identity by induction. Recall
\[
		m=2r-j\quad\text{ for }\quad 0\leq j\leq r-1.
\]
Then by induction on $j$, we prove

\begin{thm}[Spherical Second Term Identity]\label{T:Spherical_Second_Term}
For $j=0,\dots, r-1$, the following spherical second term identity holds:
\begin{align*}
    &E_{-1}^{(2r-j,r)}(\frac{r+1}{2})\prod_{i=0}^{j}F^{(2r-i+1,r)}(\frac{r+1}{2})
    +E_0^{(2r-j,r-j-1)}(\frac{r-j}{2})\prod_{i=0}^{j} G^{(2r-i+1,r-i)}(\frac{r-i+1}{2})\\
    &\qquad\qquad\qquad\qquad\qquad\qquad\qquad
    =2c_rE_{0}^{(2r-j,2r-j)}(\frac{1+j}{2})+\gamma_jE_{-1}^{(2r-j,2r-j)}(\frac{1+j}{2}),
\end{align*}
where $\gamma_j$ is some constant which depends only on $j$ (and $r$). Note that explicitly
\[
    F^{(2r-i+1,r)}(\frac{r+1}{2})=\frac{\xi(r-i+1)}{\xi(2r-i+1)}
\]
and
\[
  G^{(2r-i+1,r-i)}(\frac{r-i+1}{2})=\frac{\xi(r-i)\xi(i)}{\xi(2r-2i)\xi(2r-i+1)},
\]
where for $i=0$ and $1$ we agree that by $G^{(2r-i+1,r-i)}(\frac{r-i+1}{2})$ we mean $G_{-1}^{(2r-i+1,r-i)}(\frac{r-i+1}{2})$.
\end{thm}
\begin{proof}
We will show this by induction on $j$. Clearly $j=0$ is the spherical second term identity for $m=2r$, which we have just proved. Assume the identity holds for a given $j$. Then by looking at the constant terms, we have
\begin{align*}
    &E_{-1, Q}^{(2r-j,r)}(\frac{r+1}{2})\prod_{i=0}^{j}F^{(2r-i+1,r)}(\frac{r+1}{2})
    +E_{0, Q}^{(2r-j,r-j-1)}(\frac{r-j}{2})\prod_{i=0}^{j} G^{(2r-i+1,r-i)}(\frac{r-i+1}{2})\\
    &\qquad\qquad\qquad\qquad\qquad\qquad\qquad
    =2c_rE_{0, Q}^{(2r-j,2r-j)}(\frac{1+j}{2})+\gamma_jE_{-1, Q}^{(2r-j,2r-j)}(\frac{1+j}{2}),
\end{align*}
where each of the terms $E_{-1, Q}^{(2r-j,r)}(\frac{r+1}{2})$, $E^{(2r-j,r-j-1)}_{0, Q}(\frac{r-j}{2})$, $E_{0, Q}^{(2r-j,2r-j)}(\frac{1+j}{2})$, and $E_{-1,Q}^{(2r-j,2r-j)}(\frac{1+j}{2})$ can be computed by Proposition \ref{P:constant_term} as
{\allowdisplaybreaks
\begin{align*}
    &E_{-1,Q}^{(2r-j,r)}(\frac{r+1}{2})\\
    &=|\,|^{2r-j}E_{-1}^{(2r-j-1,r-1)}(\frac{r+2}{2})\\
    &+|\,|^rE_{-1}^{(2r-j-1,r)}F^{(2r-j, r)}(\frac{r+1}{2})
    +|\,|^rE_{-2}^{(2r-j-1,r)}(\frac{r+1}{2})F_1^{(2r-j,r)}(\frac{r+1}{2})\\
    &+|\,|^{r-j-1}E_{-1}^{(2r-j-1,r-1)}(\frac{r}{2})G^{(2r-j,r)}(\frac{r+1}{2})
    +|\,|^{r-j-1}E_{-2}^{(2r-j-1,r-1)}(\frac{r}{2})G_{1}^{(2r-j,r)}(\frac{r+1}{2})\\
    &+|\,|^{r-j-1}\ln |\,| E_{-2}^{(2r-j-1,r-1)}(\frac{r}{2})G^{(2r-j,r)}(\frac{r+1}{2}),
\end{align*}
\begin{align*}
    &E^{(2r-j,r-j-1)}_{0, Q}(\frac{r-j}{2})\\
    &=|\,|^{2r}E_0^{(2r-j-1,r-j-2)}(\frac{r-j+1}{2})
    +|\,|^{r-j-1}E_0^{(2r-j-1,r-j-1)}(\frac{r-j}{2})F^{(2r-j,r-j-1)}(\frac{r-j}{2})\\
    &+|\,|^{r-j-1}E_{-1}^{(2r-j-1,r-j-1)}(\frac{r-j}{2})F_1^{(2r-j,r-j-1)}(\frac{r-j}{2})\\
    &+|\,|^rE_0^{(2r-j-1,r-j-2)}(\frac{r-j-2}{2})G^{(2r-j,r-j-1)}(\frac{r-j}{2})\\
    &-|\,|^r\ln|\,| E_{-1}^{(2r-j-1,r-j-2)}(\frac{r-j-2}{2})G^{(2r-j,r-j-1)}(\frac{r-j}{2})\\
    &+|\,|^rE_{-1}^{(2r-j-1,r-j-2)}(\frac{r-j-2}{2})G_1^{(2r-j,r-j-1)}(\frac{r-j}{2}),
\end{align*}
\begin{align*}
    &E_{0, Q}^{(2r-j,2r-j)}(\frac{1+j}{2})\\
    &=|\,|^rE_0^{(2r-j-1,2r-j-1)}(\frac{2+j}{2})+|\,|^r\ln|\,| E_{-1}^{(2r-j-1,2r-j-1)}(\frac{2+j}{2})\\
    &+|\,|^{r-1}E_0^{(2r-j-1,2r-j-1)}(\frac{j}{2})H^{(2r-j)}(\frac{j+1}{2})
    +|\,|^{r-1}E_{-1}^{(2r-j-1,2r-j-1)}(\frac{j}{2})H_{1}^{(2r-j)}(\frac{j+1}{2})\\
    &-|\,|^{r-1}\ln |\,| E_{-1}^{(2r-j,2r-j)}(\frac{j}{2})H^{(2r-j)}(\frac{j+1}{2}),
\end{align*}
and
\begin{align*}
    &E_{-1,Q}^{(2r-j,2r-j)}(\frac{1+j}{2})\\
    &=|\,|^rE_{-1}^{(2r-j-1,2r-j-1)}(\frac{2+j}{2})+|\,|^{r-1}E_{-1}^{(2r-j-1,2r-j-1)}(\frac{j}{2})H^{(2r-j)}(\frac{1+j}{2}).
\end{align*}
}
Then by extracting the terms containing $|\,|^r$ (Lemma \ref{L:key_lemma}), we obtain the identity
\begin{align*}
    &\left( E_{-1}^{(2r-j-1,r)}F^{(2r-j, r)}(\frac{r+1}{2})
    +E_{-2}^{(2r-j-1,r)}(\frac{r+1}{2})F_1^{(2r-j,r)}(\frac{r+1}{2})\right)\prod_{i=0}^{j}F^{(2r-i+1,r)}(\frac{r+1}{2})\\
    &+\bigg(E_0^{(2r-j-1,r-j-2)}(\frac{r-j-2}{2})G^{(2r-j,r-j-1)}(\frac{r-j}{2})\\
    &\qquad+E_{-1}^{(2r-j-1,r-j-2)}(\frac{r-j-2}{2})G_1^{(2r-j,r-j-1)}(\frac{r-j}{2})\bigg)\prod_{i=0}^{j} G^{(2r-i+1,r-i)}(\frac{r-i+1}{2})\\
    &\qquad\qquad\qquad\qquad\qquad
    =2c_rE_0^{(2r-j-1,2r-j-1)}(\frac{2+j}{2})+\gamma_j E_{-1}^{(2r-j-1,2r-j-1)}(\frac{2+j}{2}).
\end{align*}
Now by Lemma \ref{L:first_B-1} and Proposition \ref{P:first_B-2}, we have
\begin{align*}
    E_{-1}^{(2r-j-1,r-j-2)}(\frac{r-j-2}{2})&=c_{2r-j-1,r-j-2}E^{(2r-j-1,2r-j-1)}_{-1}(\frac{2+j}{2})\\
    E_{-2}^{(2r-j-1,r)}(\frac{r+1}{2})&=d_{2r-j-1,r}E^{(2r-j-1,2r-j-1)}_{-1}(\frac{2+j}{2}).
\end{align*}
Using those two, one can see that the above identity simplifies to the desired second term identity for $m=2r-(j+i)$, which completes the induction step and thus the proof of the theorem.
\end{proof}

\begin{rmk}
In the above proof, to derive the formula for $m=2r-1$ from the one for $m=2r$, the computation for the induction is slightly different due to the locations of poles of Eisenstein series and the zeta function. However the essential idea is the same. The verification for this case is left to the reader.
\end{rmk}
\quad\\



\section{The Weak Second Term Identity}\label{S:weak_identity}


In this section, we will show that the spherical second term
identity on the group $G=\OO(\V)$ which we derived in the previous section can be naturally
extended to those sections which are in the $G(\A)$-span of the
spherical section.


\subsection{The complementary space and Ikeda's map}


For the pair $(\V=\V^+\oplus \V^-, W=W^+\oplus W^-)$ of split symmetric and symplectic spaces with $\dim \V=2m$ and $\dim W=2r$, the complementary space of $W$ with respect to $\V$ is defined to be the symplectic space $W_c$ with
\[
	\dim W_c+\dim W=\dim\V-2,
\]
\ie $\dim W_c=2m-2r-2$. (Of course we have been assuming $m>r$.) Now assume $\dim W_c >\dim W$, namely
\[
	m>2r+1.
\]
Fix an embedding $W\subset W_c$ so that the polarization $W_c=W_c^+\oplus W_c^-$ is of the form
\[
	W_c^+=U^+\oplus W^+\quad\text{and}\quad W_c^-=U^-\oplus W^+.
\]
Let us denote each elements in $\V^+\otimes W$ and $\V^+\otimes W_c$ by the matrices
\[
	\begin{pmatrix} v_1\otimes w^+\\v_2\otimes w^-\end{pmatrix}
	\quad\text{and}\quad
	\begin{pmatrix} x_1\otimes u^+\\ v_1\otimes w^+\\x_2\otimes u^-\\
	v_2\otimes w^-\end{pmatrix},
\]
respectively, where $v_1,v_2,x_1,x_2\in \V^+, w^+\in W^+, w^-\in W^-, u^+\in U^+$ and $u^-\in U^-$. Then there is a nice $G(\A)=\OO(\V,\A)$ intertwining map
\[
	\Ik^{(m,r)}:\Sw((\V^+\otimes W_c)(\A))\rightarrow \Sw((\V^+\otimes W)(\A))
\]
defined by
\[
	\Ik^{(m,r)}(\varphi)\begin{pmatrix} v_1\otimes w^+\\v_2\otimes w^-\end{pmatrix}
	=\int_{(\V^+\otimes U^-)(\A)}\varphi_{K_H}
	\begin{pmatrix} x_1\otimes u^+\\ v_1\otimes w^+\\0\\
	v_2\otimes w^-\end{pmatrix}\, d(x_1\otimes u^+).
\]
It can be checked that if $\varphi^0_c$ is the spherical function of our choice for $\Sw((\V^+\otimes W_c)(\A))$, then $\Ik^{(m,r)}(\varphi^0_c)$ is the spherical function $\varphi^0$ of our choice for $\Sw((\V^+\otimes W)(\A))$ as long as the measure $d(x_1\otimes u^+)$ is chosen to be the Tamagawa measure on $(\V^+\otimes U^-)(\A)$.

It is our understanding that this map is due to Ikeda and hence our notation. (See \cite{Ikeda} for the analogous map for the Siegel-Weil formula for the symplectic group.)


\subsection{\bf The (weak) first term identity for 1st term range $m\geq2r+1$}


Before taking care of the second term identity, let us take care of
the first term identities for the 1st term range $m\geq2r+1$. For the case
$m>2r+1$, the first term identity is, of course, Proposition
\ref{P:first_term}. But for the sake of our applications, we need to refine it by incorporating Ikeda's map $\Ik^{(m,r)}$ defined above.

First recall in Section \ref{S:notation} we have defined
\[
	\Sw((\V^+\otimes W)(\A))^\circ=\text{ the $\OO(\V)(\A)$-span of $\varphi^0$},
\]
where $\varphi^0\in \Sw((\V^+\otimes W)(\A))$ is the spherical function of our choice. Similarly we define $\Sw((\V^+\otimes W_c)(\A))^\circ$. Then we have

\begin{prop}\label{P:first_term2}
Assume $r<\frac{m-1}{2}$, \ie $m>2r+1$. Then for all
$\varphi\in\Sw((\V^+\otimes W_c)(\A))^\circ$, there exists a constant $a_{m,r}$ independent of $\varphi$ such that
\[
    A_{-1}^{(m,m-r-1)}(\varphi)=a_{m,r}B_{-1}^{(m,r)}(\Ik^{(m,r)}(\varphi)),
\]
namely $A_{-1}^{(m,m-r-1)}$ and $B_{-1}^{(m,r)}\circ\Ik^{(m,r)}$ are proportional as maps from $\Sw((\V^+\otimes W_c)(\A))^\circ$ to $\Aut(G)$.  Moreover the constant $a_{m,r}$ is explicitly given by
\[
	a_{m,r}=|D|^{m\rho_{m,r}}\prod_{i=1}^{m-2r-1}\frac{\xi(2(i+r))}{\xi(i)}.
\]
\end{prop}
\vskip 5pt

 \begin{proof}
This essentially follows from the spherical first term identity (Lemma \ref{L:first_B-1}) and the property of the map $\Ik^{(m,r)}$. Namely, first we have
\begin{align*}
	A_{-1}^{(m,m-r-1)}(\varphi_c^0)
	&=|D|^{-\frac{(m-r-1)m}{2}}E_{-1}^{(m,m)}(\frac{m-r-1}{2})\quad\text{by (\ref{E:E^(m,m)})}\\
	&=|D|^{-\frac{(m-r-1)m}{2}}c_{m,r}^{-1}E_{-1}^{(m,r)}(\frac{r+1}{2})\quad\text{by Lemma \ref{L:first_B-1}}.
\end{align*}
Also by Lemma \ref{L:E^(m,r)}, we have
\begin{align*}
	B_{-1}^{(m,r)}(\Ik^{(m,r)}(\varphi_c^0))
	&=B_{-1}^{(m,r)}(\varphi^0)\\
	&=|D|^{-\frac{rm}{2}}\prod_{i=1}^{r}\frac{\xi(s+m-\frac{r+1}{2}-(i-1))}{\xi(i)}E^{(m,r)}(\frac{r+1}{2}).
\end{align*}
So by combining those, we have
\begin{align*}
	&A_{-1}^{(m,m-r-1)}(\varphi_c^0)\\
	&=|D|^{-\frac{(m-r-1)m}{2}}c_{m,r}^{-1}\cdot
	|D|^{\frac{rm}{2}}\prod_{i=1}^{r}\frac{\xi(i)}{\xi(s+m-\frac{r+1}{2}-(i-1))}
	B_{-1}^{(m,r)}(\Ik^{(m,r)}(\varphi_c^0)).
\end{align*}
Then one sees that the constant in front of $B_{-1}^{(m,r)}(\Ik^{(m,r)}(\varphi_c^0))$ can be simplified to $a_{m,r}$ as in the proposition.

Now notice that if $m>2r+1$, both $A_{-1}^{(m,m-r-1)}$ and $B_{-1}^{(m,r)}$ define $G(\A)$-intertwining maps
\[
	A_{-1}^{(m,m-r-1)}:\Sw((\V^+\otimes W_c)(\A))
	\xrightarrow{\Phi^{(m,m-r-1)}}\Ind_{P(\A)}^{G(\A)}|\,|^{\rho_{m,m-r-1}}
	\rightarrow\Aut(G),
\]
and
\[
    B_{-1}^{(m,r)}:\Sw((\V^+\otimes
    W)(\A))\xrightarrow{f^{(m,r)}}\Ind_{P_r(\A)}^{G(\A)}|\,|^{\frac{r+1}{2}}\rightarrow
    \Aut(G).
\]
Since the map $\Ik^{(m,r)}$ is a $G(\A)$-intertwining map such that $\Ik^{(m,r)}(\varphi_c^0)=\varphi^0$, we see that the two maps $A_{-1}^{(m,m-r-1)}$ and $B_{-1}^{(m,r)}\circ\Ik^{(m,r)}$ from $\Sw((\V^+\otimes W_c)(\A))$ to $\Aut(G)$ are $G(\A)$-intertwining and
\[
	A_{-1}^{(m,m-r-1)}(\varphi)=a_{m,r}B_{-1}^{(m,r)}\circ\Ik^{(m,r)}(\varphi)
\]
for all $\varphi\in\Sw((\V^+\otimes W_c)(\A))^\circ$.
\end{proof}

\begin{rmk}
Since  the first term identity is shown to hold only for those $\varphi$'s in the $G(\A)$-span of the spherical function, we call it ``the weak first term identity".
For the Siegel-Weil formula of the symplectic group, Ikeda \cite{Ikeda} and Ichino \cite{Ichino} have shown this form of the first term identity, \ie the first term identity with Ikeda's map, for all Schwartz functions for the symplectic group. There is no doubt that it can be shown that the above first term identity holds for all $\varphi\in\Sw((\V^+\otimes W_c)(\A))$ simply by modifying their method. However we show the first term identity only in this weak form, since it is sufficient for our main purposes.
\end{rmk}

\vskip 5pt

Next consider the boundary case $m=2r+1$. For this, $W_c=W$ and so there is no need to introduce Ikeda's map. Indeed, we have
\begin{prop}\label{P:weak_first_term}
Let $\dim \V^+=2r+1$ and $\dim W=2r$, and $\varphi\in\Sw(\V^+\otimes
W)^\circ$. Then we have
\[
    A_0^{(2r+1,r)}(\varphi)=2 B_{-1}^{(2r+1,r)}(\varphi).
\]
\end{prop}
\begin{proof}
The proof is essentially the same as Proposition \ref{P:first_term2}. This time, notice that the first term of the Siegel Eisenstein series is $A_0^{(2r+1,r)}$ and we use Proposition \ref{P:first_A0} for the spherical first term identity. Also one sees that the constant of proportionality is simplified to $2$. The detail is left to the reader.
\end{proof}
\vskip 5pt

\begin{rmk}
Just like the case $m>2r+1$, it is highly likely that one can derive this identity for all $\varphi\in\Sw(\V^+\otimes
W)$ by using the method of Ikeda and Ichino.
\end{rmk}
\vskip 5pt


\subsection{The (weak) first term identity for 2nd term range $r<m\leq 2r$}


For the 2nd term range $r<m\leq 2r$, the non-Siegel Eisenstein series has a double pole, \ie $B_{-2}^{(m,r)}\neq0$. Then the analogous first term identity holds between  $B_{-2}^{(m,r)}$ and $A_{-1}^{(m,r)}$. Namely we have

\begin{prop}\label{P:weak_first_term2}
Let $r<\dim \V^+=m\leq 2r$ and $\dim W=2r$, and $\varphi\in\Sw(\V^+\otimes
W)^\circ$. Then we have
\[
    A_{-1}^{(m,r)}(\varphi)= b_{m,r}B_{-2}^{(m,r)}(\varphi),
\]
where the constant $b_{m,r}$ is given by
\[
    b_{m,r}=d_{m,r}^{-1}\prod_{i=1}^{r}\frac{\xi(i)}{\xi(m-i+1)}.
\]
\end{prop}
\begin{proof}
Again the proof is essentially the same as Proposition \ref{P:first_term2}. This time, we use Proposition \ref{P:first_B-2} for the spherical first term identity.
\end{proof}


\subsection{The (weak) second term identity for 2nd term range $r<m\leq 2r$}


Now let us consider the second term identity for the 2nd term range $r<m\leq 2r$.
\vskip 10pt

First notice that $A_{-1}^{(m,r)}$ defines the map
\begin{align*}
    A_{-1}^{(m,r)}: \Sw((\V^+\otimes W)(\A))&\xrightarrow{\Phi^{(m,r)}}
    \Ind_{P(\A)}^{G(\A)}|\,|^{\rho_{m,r}}\rightarrow\Aut(G)\\
    \varphi&\mapsto A_{-1}^{(m,r)}(\varphi).
\end{align*}
This map is $G(\A)$-intertwining. Let
\[
    \sR=\textnormal{Im}(A_{-1}^{(m,r)})
\]
be the image of $A_{-1}^{(m,r)}$. Now the map
\begin{align*}
    \Sw((\V^+\otimes W)(\A))&\longrightarrow\Aut(G)\\
    \varphi&\mapsto A_{0}^{(m,r)}(\varphi)
\end{align*}
is not $G(\A)$-intertwining. However it is $G(\A)$-intertwining
modulo $\sR$. So we consider the composite
\begin{align*}
    A_0^{(m,r)}: \Sw((\V^+\otimes W)(\A))&\longrightarrow\Aut(G)\longrightarrow
    \Aut(G)\slash\sR\\
    \varphi&\mapsto A_{0}^{(m,r)}(\varphi) \mod \sR.
\end{align*}
This is a $G(\A)$-intertwining map. Note that by (\ref{E:E^(m,m)}) we have
\[
    A_0^{(m,r)}(\varphi^0)(g)=|D|^{-rm/2}E_0^{(m,m)}(g, \rho_{m,r}; \Phi^0).
\]

Next let us consider the non-Siegel Eisenstein series $\E^{(m,r)}(g,s;\varphi)$.
For the non-Siegel Eisenstein series, the point of our interest is
$s=\frac{r+1}{2}$. Then at this point, $\E^{(m,r)}(s,g;\varphi)$ has
at most a double pole. Recall that we write the Laurent expansion of
$\E^{(m, r)}(s,g;\varphi)$ as
\[
    \E^{(m,r)}(s,g;\varphi)=\sum_{d=-2}^{\infty}(s-\frac{r+1}{2})^dB_d^{(m,r)}(\varphi)(g).
\]
We are interested in $B_{-1}^{(m,r)}$. As we mentioned in \S \ref{S:regularize}, for the non-Siegel Eisenstein series not only the first terms but all the terms (and so in particular $B_{-1}^{(m,r)}$) define $G(\A)$-intertwining maps. Then we consider the $G(\A)$-intertwining map
\begin{align*}
    B_{-1}^{(m,r)}:\Sw((\V^+\otimes W)(\A))&\longrightarrow
    \Aut(G)\longrightarrow\Aut(G)\slash\sR\\
    \varphi&\mapsto B_{-1}^{(m,r)}(\varphi) \mod \sR.
\end{align*}
Also note that by Lemma \ref{L:E^(m,r)} we have
\[
    B_{-1}^{(m,r)}(\varphi^0)(g)=|D|^{-rm/2}\prod_{i=1}^{r}\frac{\xi(m-i+1)}
    {\xi(i)}E_{-1}^{(m,r)}(g, \frac{r+1}{2}; f_0).
\]

Finally, we need to take care of the ``complementary term"
$E_0^{(m,m-r-1)}(\frac{m-r}{2})$. If $W$ has a symplectic basis $\{e_1,\dots,
e_r,  f_1,\dots,f_r\}$, let
\[
    W_0=\Span\{e_1,\dots,e_{m-r-1}, f_1,\dots,f_{m-r-1}\},
\]
which has the obvious symplectic
structure. We also denote $W_0^+$ and $W_0^-$ for the spans of
$\{e_1,\dots, e_{m-r-1}\}$ and $\{f_1,\dots,f_{m-r-1}\}$, respectively.
Then notice the complementary space of $W_0$ with respect to $\V$ is actually $W$. Hence we have Ikeda's map
\[
	\Ik^{(m,m-r-1)}:\Sw((\V^+\otimes W)(\A))\rightarrow\Sw((\V^+\otimes W_0)(\A)).
\]
Whenever there is no danger of confusion, we simply write
\[
    \Ik=\Ik^{(m,m-r-1)}.
\]
Now for each $\varphi\in\Sw((\V^+\otimes W)(\A))$ as in \S \ref{S:regularize} we can define the
non-Siegel Eisenstein series
\[
    \E^{(m,m-r-1)}(g,s;\Ik(\varphi))=\sum_{\gamma\in P_{m-r-1}(F)\backslash G(F)} f^{(m,m-r-1)}(\gamma
    g,s;{\Ik(\varphi)_{K_{H'}}}),
\]
where $H'=\Sp(2(m-r-1))$. This has at most a simple pole at
$s=\frac{m-r}{2}$ and we write its Laurent expansion at
$s=\frac{m-r}{2}$ as
\[
    \E^{(m,m-r-1)}(g,s;\Ik(\varphi))
     =\sum_{d=-1}^{\infty}B_d^{(m,m-r-1)}(\Ik(\varphi))(g)(s-\frac{m-r}{2})^d
\]
as usual. Then $B_0^{(m,m-r-1)}\circ\Ik$ defines a $G(\A)$-intertwining map
\[
    B_0^{(m,m-r-1)}:\Sw((\V^+\otimes W)(\A))
    \longrightarrow\Aut(G)\longrightarrow\Aut(G)\slash\sR.
\]
Also note that for the spherical Schwartz function $\varphi^0\in\Sw((\V^+\otimes W)(\A))$, by lemma \ref{L:E^(m,r)} we
have
\begin{align*}
    &B_0^{(m,m-r-1)}(\Ik(\varphi^0))
    =|D|^{-m(m-r-1)/2}\prod_{i=1}^{m-r-1}\frac{\xi(m-i+1)}
    {\xi(i)}E^{(m,m-r-1)}(\frac{m-r}{2}).
\end{align*}

Then the spherical second term identity (Theorem
\ref{T:Spherical_Second_Term}) can be expressed in terms of
$A_0^{(m,r)}, B_{-1}^{(m,r)},$ and $B_0^{(m,m-r-1)}$ as follows.
\begin{align*}
    &\prod_{i=0}^{2r-m}F^{(2r-i+1,r)}(\frac{r+1}{2})
    \cdot|D|^{rm/2}\prod_{i=1}^{r}\frac{\xi(i)}{\xi(m-i+1)}\cdot B_{-1}^{(m,r)}(\varphi^0)\\
    &+\prod_{i=0}^{2r-m} G^{(2r-i+1,r-i)}(\frac{r-i+1}{2})\cdot|D|^{m(m-r-1)/2}\prod_{i=1}^{m-r-1}\frac{\xi(i)}
    {\xi(m-i+1)}B_0^{(m,m-r-1)}(\Ik(\varphi^0))\\
    &\qquad\qquad\qquad
    \equiv |D|^{rm/2}\prod_{i=1}^{r+1}\frac{\xi(i)}{\xi(r+i)}A_0^{(m,r)}(\varphi^0)(g) \mod \sR.
\end{align*}
Here, when one applies Theorem \ref{T:Spherical_Second_Term}, it is convenient to keep in mind that in Theorem \ref{T:Spherical_Second_Term}
\[
		m=2r-j\quad\text{and}\quad\rho_{m,r}=\frac{1+j}{2}.
\]
One sees that many of the $\xi$'s get canceled out, and we can rewrite the spherical second term identity as
\begin{align*}
    &B_{-1}^{(m,r)}(\varphi^0)
    +|D|^{-m\rho_{m,r}}\prod_{i=0}^{2r-m}\frac{\xi(i)}{\xi(2r-2i)}B_0^{(m,m-r-1)}(\Ik(\varphi^0))\\
    &\qquad\qquad\qquad\equiv
    A_0^{(m,r)}(\varphi^0) \mod \sR,
\end{align*}
Hence we have
\begin{thm}[Weak Second Term Identity]\label{T:Second_Term_Identity}
For all $\varphi\in \Sw((\V^{+}\otimes W)(\A))$ that are in the $\OO(\V,\A)$-span of $\varphi^0$, the following identity holds:
\begin{align*}
    &B_{-1}^{(m,r)}(\varphi)
    +|D|^{-m\rho_{m,r}}\prod_{i=0}^{2r-m}\frac{\xi(i)}{\xi(2r-2i)}B_0^{(m,m-r-1)}(\Ik(\varphi))\\
    &\qquad\qquad\qquad\equiv
    A_0^{(m,r)}(\varphi) \mod \sR,
\end{align*}
\end{thm}
\begin{proof}
This is immediate from the above form of the spherical second term identity together with the fact that $B_{-1}^{(m,r)}, A_0^{(2r,r)}$ and $B_0^{(m,m-r-1)}\circ\Ik$ are $\OO(\V,\A)$-intertwining maps from $\Sw((V^{+}\otimes W)(\A))$ to $\Aut(G)\slash\sR$.
\end{proof}

\begin{rmk}
As we mentioned in the introduction, the above weak second term identity has been shown by Ichino and the first author \cite[Proposition 7.5]{GI} for the case $r=2$ and $m=4$.
\end{rmk}

\begin{rmk}
Just as our first term identity, we are able to derive the second term identity only for those $\varphi$'s in the span of the spherical function. This is why we call it ``the weak second term identity". However, unlike the (weak) first term identity (Proposition \ref{P:first_term2} and \ref{P:weak_first_term2}), there seems to be no known method that possibly allows one to extend the weak second term identity to full generality so that the above second identity holds for all $\varphi\in\Sw((V^{+}\otimes W)(\A))$.
\end{rmk}

\quad



\section{\bf Inner product formulas}\label{S:inner_product}


By using our Siegel-Weil formula (both the first term and the second term identities), we will derive the Rallis inner
product formula for the theta lift from $\OO(V,\A)$ to $\Sp_{2r}$ if $\dim V$ is even and to $\Spt_{2r}$ if it is odd. For the rest of the paper, we let
\[
    \Mp_{2r}=\begin{cases}\Sp_{2r}\quad\text{if $\dim V$ is even}\\
    \Spt_{2r}\quad\text{if $\dim V$ is odd}.
    \end{cases}
\]
The reader will be able to tell which one is meant from the context.

Let us recall our setting. Namely, $W$ is the symplectic space of $\dim W=2r$ with a
fixed polarization $W^+\oplus W^-$, and $V$ is a (not necessarily split) quadratic space of $\dim V=m$ defined over $F$ and $\V=V\oplus
(-V)$, where $\V$ is the split quadratic space with the
underlining space $V\oplus V$ and the quadratic form defined by
\[
    \langle(v_1,v_2),(v'_1,v_2')\rangle_{\V}=\langle
    v_1,v_1'\rangle_V-\langle v_2,v_2'\rangle_V.
\]
Let us again note that $V$ is not necessarily split but $\V$ is always split.
Indeed, a maximal isotropic subspace is
\[  \V^+=\Delta V\subset V\oplus(-V). \]
There is a natural embedding
\[
    i:\OO(V)\times\OO(-V)\hookrightarrow\OO(\V),
\]
which should be called the embedding of the doubling method. Also we
have a $\OO(V,\A)\times\OO(-V,\A)$-intertwining map
\[
    \sigma:\Sw((V\otimes W^+)(\A))\hat{\otimes}\Sw((V\otimes
    W^+)(\A))\rightarrow\Sw((\V\otimes W^+)(\A))\simeq\Sw((\V^+\otimes
    W)(\A)),
\]
where we view $\Sw((\V\otimes W^+)(\A))$ and $\Sw((\V^+\otimes W)(\A))$ as representations of $\OO(V,\A)\times\OO(-V,\A)$ via the
embedding $i$.
\vskip 10pt

Now let $\pi$ be a cuspidal automorphic representation of $\OO(V,\A)$. For a cusp form $f\in\pi$ and Schwartz function $\phi\in\Sw((V\otimes W^+)(\A))$, we define the theta lift $\theta_{\psi,2r}(f,\phi)\in\Aut(\Mp_{2r})$ of $f$ to $\Mp_{2r}$ (with respect to $\phi$) by
\[
    \theta_{\psi,2r}(f,\phi)(h)=\int_{\OO(V,F)\backslash\OO(V,\A)}\sum_{v\in(V\otimes W^+)(F)}\omega_{\psi}(h,g)\phi(v)\;dg.
\]
Then we define the theta lift $\theta_{\psi, 2r}(\pi)$ of $\pi$ to $\Mp_{2r}$ by
\[
    \theta _{\psi,2r}(\pi)=\{\theta_{\psi,2r}(f,\phi):f\in\pi,\phi\in\Sw((V\otimes W^+)(\A))\}.
\]
We often omit $\psi$ from the notations and simply write $\theta_{2r}(f,\phi)$ and $\theta_{2r}(\pi)$. It is well-known that if $\theta_{2r_0}(\pi)\neq0$ for some $r_0$, then $\theta_{2r}(\pi)\neq0$ for all $r\geq r_0$ (tower property), and if $r$ is the smallest integer with $\theta_{2r}(\pi)\neq0$, then $\theta_{2r}(\pi)$ is in the space of cusp forms.
\vskip 10pt

Now we have
\vskip 10pt

\begin{prop}\label{P:inner_product}
Let $\pi$ be a cuspidal automorphic representation of $\OO(V,\A)$ with $\dim V=m$. For $f_1, f_2\in\pi$ and $\phi_1, \phi_2\in\Sw((V\otimes W^+)(\A))$, we let $\theta_{2r}(f_i,\phi_i)$ be the theta lift of $f_i$ with $\phi_i$ to $\Mp_{2r}(\A)$. Further assume $\theta_{2r-2}(\pi)=0$ so that both of the $\theta_{2r}(f_i,\phi_i)$'s are cusp forms (possibly zero). Then the Peterson inner product $\langle\theta_{2r}(f_1,\phi_1), \theta_{2r}(f_2,\phi_2)\rangle$ of the lifts is given by
{\allowdisplaybreaks
\begin{align*}
    &\langle\theta_{2r}(f_1,\phi_1),\theta_{2r}(f_2,\phi_2)\rangle\\
    &= \int_{\Sp_{2r}(F)\backslash \Sp_{2r}(\A)}\theta_{2r}(f_1,\phi_1)(h)
    \overline{\theta_{2r}(f_2,\phi_2)(h)}\,dh\\
    &=\frac{1}{\kappa} \cdot \underset{s=\frac{r+1}{2}}{\Res}
    \int_{(G\times G)(F)\backslash(G\times G)(\A)}
    f_1(g_1)\overline{f_2(g_2)}\E^{(m,r)}(i(g_1,g_2),s,\sigma(\phi_1\otimes\overline{\phi_2}))\,dg_1dg_2\\
    &=\frac{1}{\kappa} \cdot \int_{(G\times G)(F)\backslash(G\times
    G)(\A)}f_1(g_1)\overline{f_2(g_2)}B_{-1}^{(m,r)}(\sigma(\phi_1\otimes\overline{\phi_2}))(i(g_1,g_2))\,dg_1dg_2,
\end{align*}
where $G=\OO(V)$ and $\kappa$ is the residue of the auxiliary Eisenstein series $E(h,s)$ on $\Sp_{2r}(\A)$. Note that even if $\dim V$ is odd, the integral for the inner product is over $\Sp_{2r}(F)\backslash \Sp_{2r}(\A)$.}
\end{prop}
\begin{proof}
First recall that, for an algebraic group $G$, we write $[G]=G(F)\backslash G(\A)$.
Then we have
{\allowdisplaybreaks
\begin{align*}
    &\frac{1}{\kappa}\int_{[G\times G]}
    f_1(g_1)\overline{f_2}(g_2)\E^{(m,r)}(i(g_1,g_2),s;\sigma(\phi_1\otimes\overline{\phi_2}))\,dg_1dg_2\\
    &=\frac{1}{\kappa P_z(s)}\int_{[G\times G]}
    f_1(g_1)\overline{f_2(g_2)}I^{(m,r)}(i(g_1,g_2),s;\omega(z)\sigma(\phi_1\otimes\overline{\phi_2}))\,dg_1dg_2\\
    &=\frac{1}{\kappa P_z(s)}\int_{[G\times G]}
    f_1(g_1)\overline{f_2(g_2)}\int_{[\Sp]}
	\theta(i(g_1,g_2),h;\omega(z)\sigma(\phi_1\otimes\overline{\phi_2}))E(h,s)\,dh\,dg_1dg_2\\
    &=\frac{1}{\kappa P_z(s)}\int_{[\Sp]}\left(\int_{[G\times G]}
    f_1(g_1)\overline{f_2(g_2)}\theta(i(g_1,g_2),h;\omega(z)\sigma(\phi_1\otimes\overline{\phi_2}))\,dg_1dg_2\right)E(h,s) dh\\
    &=\frac{1}{\kappa P_z(s)}\int_{[\Sp]}\left(\int_{[G\times G]}
    f_1(g_1)\overline{f_2(g_2)}\theta(i(g_1,g_2),h;\sigma(\phi_1\otimes\overline{\phi_2}))\,dg_1dg_2\right)z'\ast E(h,s) dh\\
    &=\frac{1}{\kappa}\int_{[\Sp]}\left(\int_{[G\times G]}
    f_1(g_1)\overline{f_2(g_2)}\theta(i(g_1,g_2),h;\sigma(\phi_1\otimes\overline{\phi_2}))\,dg_1dg_2\right)E(h,s) dh\\
    &=\frac{1}{\kappa}\int_{[\Sp]}
    \theta_{2r}(f_1,\phi_1)(h)\overline{\theta_{2r}(f_2,\phi_2)}(h)E(h,s)\,dh.
\end{align*}}
Here we used the Poisson summation formula to show that for $\phi_1, \phi_2\in\Sw((V\otimes W^+)(\A))$
\[
    \theta(i(g_1,g_2),h;\sigma(\phi_1\otimes\overline{\phi_2}))
    =\theta(g_1,h;\phi_1)\theta(g_2,h;\overline{\phi_2}).
\]
Also we used the adjointness of the operator $z'$. Now the auxiliary Eisenstein series $E(h,s)$ has a constant residue
$\kappa$ at $s=\frac{r+1}{2}$. So we have
\begin{align*}
    &\underset{s=\frac{r+1}{2}}{\Res}\frac{1}{\kappa}\int_{\Sp(F)\backslash \Sp(\A)}
    \theta_{2r}(f_1,\phi_1)(h)\overline{\theta_{2r}(f_2,\phi_2)}(h)E(h,s)\,dh\\
    &=\int_{\Sp(F)\backslash \Sp(\A)}
    \theta_{2r}(f_1,\phi_1)(h)\overline{\theta_{2r}(f_2,\phi_2)}(h)\,dh.
\end{align*}
Thus the proposition follows.
\end{proof}

As a corollary, which we will use later, we have
\begin{cor}\label{C:inner_product}
Let
$\underset{i}{\sum}\sigma(\phi_{1,i}\otimes\overline{\phi_{2,i}})\in\Sw((\V^+\otimes
W)(\A))$. Then
\begin{align*}
    &\sum_i\langle\theta_{2r}(f_1,\phi_{1,i}),\theta_{2r}(f_2,\phi_{2,i})\rangle\\
    &=\frac{1}{\kappa}\int_{[G\times G]}f_1(g_1)\overline{f_2(g_2)}B_{-1}^{(m,r)}(\sum_i\sigma(\phi_{1,i}\otimes\overline{\phi_{2,i}}))(i(g_1,g_2))\,dg_1dg_2.\\
\end{align*}
\end{cor}
\begin{proof}
This immediately follows from the above proposition because the map
$B_{-1}^{(m,r)}$ is linear on $\Sw((\V^+\otimes W)(\A))$.
\end{proof}


\subsection{Inner product formula for the theta lift from $\OO(V)$ to $\Mp_{2r}$}

We derive the inner product formula for the theta lift from the orthogonal group
$\OO(V)$ to $\Mp_{2r}$, provided $\sigma(\phi_1\otimes\phi_2)$ is in $\Sw((V^+\otimes W)(\A))^\circ$. For this, we need to consider the following see-saw diagram:
\[
    \xymatrix{
    \OO(V\oplus(-V))\ar@{<-}[d]_{i}\ar@{-}[dr]&\Mp_{2r}\times\Mp_{2r}\ar@{<-}[d]\\
    \OO(V)\times\OO(-V)\ar@{-}[ur]&\Sp_{2r},
    }
\]
where $\Mp_{2r}$ is either $\Sp_{2r}$ or $\Spt_{2r}$ depending on the parity of $\dim V$. Note that even if $\dim V$ is odd, at the lower right corner of the diagram we have $\Sp_{2r}$.

Recall that the 1st term range is $m\geq 2r+1$ and the 2nd term range is $r+1\leq m\leq 2r$. It will be helpful to keep the following diagram in mind regarding the theta lift from $\OO(V)$ to the symplectic tower. Namely, when $\dim V=m$ is even, we consider
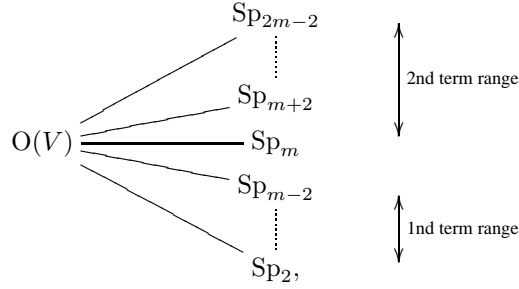
\[
    \xymatrix@R=1pt{
	&&\Sp_{2m-2}\ar@{.}[ddd]&\ar@{<->}[dddd]^{\text{2nd term range}}\\
	&&&\\
	&&&\\
	&&\Sp_{m+2}&\\
\OO(V)\ar@{-}[uuuurr]\ar@{-}[urr]\ar@{-}[rr]\ar@{-}[drr]\ar@{-}[ddddrr]
	&&\Sp_{m}&\\
	&&\Sp_{m-2}\ar@{.}[ddd]&\ar@{<->}[ddd]^{\text{1nd term range}}\\
	&&&\\
	&&&\\
	&&\Sp_2,&
    }
\]
and when $\dim V=m$ is odd, we consider
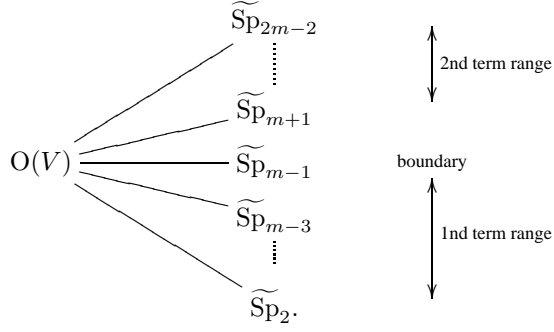
\[
    \xymatrix@R=1pt{
	&&\Spt_{2m-2}\ar@{.}[ddd]&\ar@{<->}[ddd]^{\text{2nd term range}}\\
	&&&\\
	&&&\\
	&&\Spt_{m+1}&\\
\OO(V)\ar@{-}[uuuurr]\ar@{-}[urr]\ar@{-}[rr]\ar@{-}[drr]\ar@{-}[ddddrr]
	&&\Spt_{m-1}&\text{\scriptsize boundary}\ar@{<->}[dddd]^{\text{1nd term range}}\\
	&&\Spt_{m-3}\ar@{.}[dd]&\\
	&&&\\
	&&&\\
	&&\Spt_2.&
    }
\]

First we have
\begin{prop}\label{P:inner_product_2nd_term1}
Keeping all the assumptions and the notations of Proposition \ref{P:inner_product}, let us further assume that $\dim V=m$ where $r+1\leq m\leq 2r$. Then if $\sum_i\sigma(\phi_{1,i}\otimes\phi_{2,i})\in\Sw((\V^+\otimes W)(\A))^\circ$, we have
\begin{align*}
    &\sum_i\langle\theta_{2r}(f_1,\phi_{1,i}),\theta_{2r}(f_2,\phi_{2,i})\rangle\\
    &=\frac{1}{\kappa}\int_{(G\times G)(F)\backslash(G\times
    G)(\A)}f_1(g_1)\overline{f_2(g_2)}A_{0}^{(m,r)}(\sum_i\sigma(\phi_{1,i}\otimes\overline{\phi_{2,i}}))(i(g_1,g_2))\,dg_1dg_2.
\end{align*}
\end{prop}
\begin{proof}
Corollary \ref{C:inner_product} together with our weak second term identity immediately implies
\begin{align*}
    &\kappa\sum_i\langle\theta_{2r}(f_1,\phi_{1,i}),\theta_{2r}(f_2,\phi_{2,i})\rangle\\
    &=\int_{[G\times G]}f_1(g_1)\overline{f_2(g_2)}B_{-1}^{(m,r)}(\sum_i\sigma(\phi_{1,i}\otimes\overline{\phi_{2,i}}))(i(g_1,g_2))\,dg_1dg_2\\
    &=\int_{[G\times G]}f_1(g_1)\overline{f_2(g_2)}A_{0}^{(m,r)}(\sum_i\sigma(\phi_{1,i}\otimes\overline{\phi_{2,i}}))(i(g_1,g_2))\,dg_1dg_2\\
    &-c\int_{[G\times G]}f_1(g_1)\overline{f_2(g_2)}
    B_{0}^{(m,m-r-1)}(\Ik(\sum_i\sigma(\phi_{1,i}\otimes\overline{\phi_{2,i}})))(i(g_1,g_2))\,dg_1dg_2\\
    &+C\int_{[G\times G]}f_1(g_1)\overline{f_2(g_2)}A_{-1}^{(m,r)}(\varphi')(i(g_1,g_2))\,dg_1dg_2,
\end{align*}
where $c$ is the constant in the second term identity, $C$ is some constant depending on $\sum_i\sigma(\phi_{1,i}\otimes\phi_{2,i})$ and $r$, and $\varphi'$ is some Schwartz function in $\Sw((\V^+\oplus W)(\A))$. We have only to show
\begin{equation}\label{E:B_0=0}
    \int_{[G\times G]}f_1(g_1)\overline{f_2(g_2)}
    B_{0}^{(m,m-r-1)}(\Ik(\sum_i\sigma(\phi_{1,i}\otimes\overline{\phi_{2,i}})))(i(g_1,g_2))\,dg_1dg_2=0
\end{equation}
and
\begin{equation}\label{E:A_-1=0}
    \int_{[G\times G]}f_1(g_1)\overline{f_2(g_2)}
    A_{-1}^{(m,r)}(\varphi')(i(g_1,g_2))\,dg_1dg_2=0.
\end{equation}
\vskip 10pt

Both of these are implied by our assumption $\theta_{2r-2}(\pi)=0$. First let us define the natural isomorphism
\begin{align*}
    &\sigma': \Sw((V\otimes W_0^+)(\A))\hat{\otimes}\Sw((V\otimes
    W_0^+)(\A))\\
    &\qquad\qquad\qquad\rightarrow\Sw((\V\otimes W_0^+)(\A))\simeq\Sw((\V^+\otimes
    W_0)(\A)).
\end{align*}

To show (\ref{E:B_0=0}), let us write
\[
    \Ik(\sum_i\sigma(\phi_{1,i}\otimes\overline{\phi_{2,i}}))
	=\sum_{j}\sigma'(\phi'_{1,j}\otimes\overline{\phi'_{2,j}}),
\]
where $\phi'_{1,j}\otimes\overline{\phi'_{2,j}}\in\Sw((V\otimes
W_0^+)(\A))\hat{\otimes}\Sw((V\otimes W_0^+)(\A))$.
\vskip 10pt

Then by the analogous computation as the proof of Proposition \ref{P:inner_product}, we can see
\begin{align*}
    &\int_{[G\times G]}f_1(g_1)\overline{f_2(g_2)}
    B_{0}^{(m,m-r-1)}(\Ik(\sum_i\sigma(\phi_{1,i}\otimes\overline{\phi_{2,i}})))(i(g_1,g_2))\,dg_1dg_2\\
    &=\underset{s=\frac{m-r}{2}}{\Val}\sum_j\int_{[\Sp_{2(m-r-1)}]}
    \theta_{2(m-r-1)}(f_1,\phi'_{1,j})(h)\theta_{2(m-r-1)}(f_2,\overline{\phi'_{2,j}})(h)E'(h,s)\,
    dh,
\end{align*}
where $E'(h,s)$ is the auxiliary Eisenstein series on the group
$\Sp_{2(m-r-1)}(\A)$. But
\[
    \theta_{2(m-r-1)}(f_1,\phi'_{1,j})=\theta_{2(m-r-1)}(f_2,\overline{\phi'_{2,j}})=0
\]
because by our assumption $\theta_{2r-2}(\pi)=0$, and hence (\ref{E:B_0=0}) follows.
\vskip 10pt

To show (\ref{E:A_-1=0}), recall that the complementary space of $W_0$ is actually $W$. Hence by Proposition \ref{P:first_term2},
\begin{align*}
    &\int_{[G\times G]}f_1(g_1)\overline{f_2(g_2)}
    A_{-1}^{(m,r)}(\varphi')(i(g_1,g_2))\,dg_1dg_2\\
    &=a_{m,m-r-1}\int_{[G\times G]}f_1(g_1)\overline{f_2(g_2)}
    B_{-1}^{(m,m-r-1)}(\Ik(\varphi'))(i(g_1,g_2))\,dg_1dg_2.
\end{align*}
Again let us write $\Ik(\varphi')=\sum_{j}\sigma'(\phi''_{1,j}\otimes\overline{\phi''_{2,j}})$, by Corollary \ref{C:inner_product} we have
\begin{align*}
    &\kappa\sum_j\langle\theta_{2(m-r-1)}(f_1,\phi_{1,j}),\theta_{2(m-r-1)}(f_2,\phi_{2,j})\rangle\\
    &=\int_{[G\times G]}f_1(g_1)\overline{f_2(g_2)}
    B_{-1}^{(m,m-r-1)}(\Ik(\varphi'))(i(g_1,g_2))\,dg_1dg_2.
\end{align*}
But this is again zero because
$\theta_{2(m-r-1)}(f_1,\phi''_{1,j})=\theta_{2(m-r-1)}(f_2,\overline{\phi''_{2,j}})=0$ by our assumption that $\theta_{2r-2}(\pi)=0$, and hence (\ref{E:A_-1=0}) follows.
\end{proof}
\vskip 10pt

Now this proposition implies the following inner product formula.
\vskip 10pt

\begin{thm}[Inner Product Formula for 2nd term range]\label{T:inner_product_2nd_term2}
Let $\pi$ be a cuspidal automorphic representation of $\OO(V,\A)$ with $\dim V=m$ and $r+1\leq m\leq 2r$ such that $\theta_{2r-2}(\pi)=0$. Let $\varphi\in\Sw((\V^+\otimes W)(\A))^\circ$ be such that $\Phi^{(m,r)}_{\varphi}$ is factorizable as $\otimes'\Phi_v$. Let us write $\varphi=\sum_i\sigma(\phi_{1,i}\otimes\overline{\phi_{2,i}})$. Then for $f_1,f_2\in\pi$ we have
\begin{align*}
    &\sum_i\langle\theta_{2r}(f_1,\phi_{1,i}),\theta_{2r}(f_2,\phi_{2,i})\rangle\\
    &\qquad\qquad=\frac{1}{d^S_{m}(\rho_{m,r})}\underset{s=\rho_{m,r}}{\Val}\left(L^S(s+\frac{1}{2}, \pi)\cdot \prod_{v\in S}Z_v(s,f_{1,v}, f_{2,v}, \Phi_v, \pi_v)\right),
\end{align*}
where $d_{m}^S(s)$ is the normalizing factor for the doubling method, which is explicitly given by
\[
    d_{m}^S(s)=\begin{cases}\prod_{i=1}^{\frac{m}{2}}\zeta^S(2s+2i-1)\quad\text{if $m$ is even}\\
            \prod_{i=1}^{\frac{m-1}{2}}\zeta^S(2s+2i)\quad\text{if $m$ is odd}.
            \end{cases}
\]
\end{thm}
\begin{proof}
By Proposition \ref{P:inner_product_2nd_term1} we have
\begin{align*}
    &\sum_i\langle\theta_{2r}(f_1,\phi_{1,i}),\theta_{2r}(f_2,\phi_{2,i})\rangle\\
    &=\frac{1}{\kappa}\int_{[G\times G]}f_1(g_1)\overline{f_2(g_2)}
    A_{0}^{(m,r)}(\sum_i\sigma(\phi_{1,i}\otimes\overline{\phi_{2,i}}))(i(g_1,g_2))\,dg_1dg_2\\
    &=\frac{1}{\kappa}\int_{[G\times G]}f_1(g_1)\overline{f_2(g_2)}
    A_{0}^{(m,r)}(\Phi_\varphi^{(m,r)})(i(g_1,g_2))\,dg_1dg_2,
\end{align*}
which is, by the doubling method, written as
\[
    \frac{1}{\kappa}\underset{s=\rho_{m,r}}{\Val}\frac{1}{d^S_{m}(s)}
    \left(L^S(s+\frac{1}{2}, \pi)\prod_{v\in S}Z_v(s,f_{1,v}, f_{2,v},
    \Phi_v,\pi_v)\right).
\]
\end{proof}

\begin{rmk}
The assumption that $\sum_i\sigma(\phi_{1,i}\otimes\overline{\phi_{2,i}})$ is in the $\OO(\V,\A)$-span of the spherical Schwartz function $\varphi^0$ is necessary because our second term identity works only those $\varphi$'s in $\Sw((\V^+\otimes W)(\A))^\circ$. Of course, if the second term identity can be extended to all the Schwartz functions in $\Sw((\V^+\otimes W)(\A))$, the inner product formula can also be completed in full generality.
\end{rmk}
\vskip 10pt

Next we also consider the inner product formula for the 1st term range, \ie $m\geq 2r+1$. First of all, for the boundary case, \ie $m=2r+1$, it is exactly the same as the inner product formula for the 2nd term range. Namely we have

\begin{thm}[Inner Product Formula on the boundary]
Let $\pi$ be a cuspidal automorphic representation of $\OO(V,\A)$ with $\dim V=m=2r+1$ such that $\theta_{2r-2}(\pi)=0$. Also let $\varphi\in\Sw((\V^+\otimes W)(\A))^\circ$ be such that $\Phi^{(m,r)}_{\varphi}$ is factorizable as $\otimes'\Phi_v$. Let us write $\varphi=\sum_i\sigma(\phi_{1,i}\otimes\overline{\phi_{2,i}})$. Then for $f_1,f_2\in\pi$ we have
\begin{align*}
    &\sum_i\langle\theta_{2r}(f_1,\phi_{1,i}),\theta_{2r}(f_2,\phi_{2,i})\rangle\\
    &=\frac{2}{\kappa}\int_{(G\times G)(F)\backslash(G\times G)(\A)}f_1(g_1)\overline{f_2(g_2)}
    A_{0}^{(2r+1,r)}(\sum_i\sigma(\phi_{1,i}\otimes\overline{\phi_{2,i}}))(i(g_1,g_2))\,dg_1dg_2.
\end{align*}
Further if $\varphi=\sum_i\sigma(\phi_{1,i}\otimes\overline{\phi_{2,i}})$ is such that $\Phi^{(m,r)}_\varphi=\otimes'\Phi_v$, then the latter is equal to
\[
    \frac{2}{\kappa d^S_{2r+1}(0)}\underset{s=0}{\Val}
    \left(L^S(s+\frac{1}{2}, \pi)\prod_{v\in S}Z_v(s,f_{1,v}, f_{2,v},
    \Phi_v,\pi_v)\right).
\]
\end{thm}
\begin{proof}
This follows from a similar computation as the inner product formula for the 2nd term range by using Proposition \ref{P:weak_first_term}. Actually this time is much easier and left to the reader.
\end{proof}

Now for the 1st term range other than the boundary case, \ie $m>2r+1$, by using the first term identity Proposition \ref{P:first_term2}, we can derive the inner product formula as follows.

\begin{thm}[Inner Product Formula for 1st term range]\label{T:inner_product3}
Let $\pi$ be a cuspidal automorphic representation of $\OO(V,\A)$ with $\dim V=m>2r+1$ such that $\theta_{2r-2}(\pi)=0$. Also let $\varphi'\in\Sw((\V^+\otimes W_c)(\A))^\circ$ be such that $\Phi^{(m,m-r-1)}_{\varphi'}$ is factorizable as $\otimes'\Phi_v$. Let us write $\Ik(\varphi')=\sum_i\sigma(\phi_{1,i}\otimes\overline{\phi_{2,i}})$. Then for $f_1,f_2\in\pi$ we have
\begin{align*}
    &\sum_i\langle\theta_{2r}(f_1,\phi_{1,i}),\theta_{2r}(f_2,\phi_{2,i})\rangle\\
    &\qquad\qquad=\frac{a_{m,r}}{d^S_{m}(-\rho_{m,r})}\underset{s=-\rho_{m,r}}{\Res}\left(L^S(s+\frac{1}{2}, \pi)\prod_{v\in S}Z_v(s,f_{1,v}, f_{2,v}, \Phi_v, \pi_v)\right).
\end{align*}
\end{thm}
\begin{proof}
Left to the reader.
\end{proof}

As the last thing in this section, let us mention the following. By looking at the equation (\ref{E:A_-1=0}) in the proof of Proposition \ref{P:inner_product}, one notices that this immediately implies the following fact, which we will use later.

\begin{lemma}\label{L:no_residue}
If $\pi$ is a cuspidal representation of $\OO(V,\A)$ with $\dim V=m$ such that $\theta_{2r-2}(\pi)=0$ and $r+1\leq m\leq 2r$, \ie the 2nd term range, then for $\varphi\in\Sw((\V^+\otimes W)(\A))$ we have
\[
    \underset{s=\rho_{m,r}}{\Res}\int_{[G\times G]}f_1(g_1)\overline{f_2(g_2)}
    E^{(m,m)}(i(g_1,g_2),s;\Phi^{(m,r)}_{\varphi})dg_1\,dg_2=0.
\]
Hence if $f_1,f_2$ and $\Phi^{(m,r)}_{\varphi}$ are factorizable as $\otimes'f_{1,v}, \otimes'f_{2,v}$ and $\otimes'\Phi_v$, respectively, then by the doubling method, we have
\[
    \underset{s=\rho_{m,r}}{\Res} \frac{1}{d_{m}^S(s)}L^S(s+\frac{1}{2}, \pi)\prod_{v\in S}Z_v(s,f_{1,v},f_{2,v},\Phi_v, \pi_v)=0,
\]
and so
\[
    L^S(s+\frac{1}{2}, \pi)\prod_{v\in S}Z_v(s,f_{1,v},f_{2,v},\Phi_v, \pi_v)
\]
is holomorphic at $s=\rho_{m,r}$. (Note that in this lemma, $\varphi$ is not just in $\Sw((\V^+\otimes W)(\A))^\circ$ but in $\Sw((\V^+\otimes W)(\A))$.)
\end{lemma}
\begin{proof}
This is immediate from (\ref{E:A_-1=0}). The last statement follows from the fact that the Siegel Eisenstein series has at most a simple pole. Also note that the normalizing factor $d_{m}^S(s)$ is non-zero holomorphic at $s=\rho_{m,r}$.
\end{proof}

\quad\\



\section{\bf Non-vanishing of theta lifts}\label{S:non_vanishing}


In this section, as an application of our inner product formula, we show a certain non-vanishing result for the global theta lift from $\OO(V)$ to $\Mp_{2r}$. In this section, $G=\OO_{m,m}$ and $P$ is the Siegel parabolic of $G$. We often write $G_v=G(F_v)$ and $P_v=P(F_v)$.


\subsection{The structure of degenerate principal series}


First, we need to recall the structure of the local degenerate
principal series $\Ind_{P_v}^{G_v}|\,|^s$ for any $v$.  The detailed structure of this degenerate principal
series is well-known. (See \cite{Ban_Jantzen} for the non-archimedean
case, \cite{Loke} for the real case, and \cite{Lee-Zhu} for the
complex case). For our purposes, we only need the following.

\begin{prop}
Let $s_0=\rho_{m,r}=\frac{2r-m+1}{2}$, where $0\leq s_0\leq\frac{m-1}{2}$ \ie $\frac{m-1}{2}\leq r\leq m-1$. Then the degenerate principal series
$\Ind_{P_v}^{G_v}|\,|^{s_0}$ has two maximum subrepresentations
$\sigma_v^+$ and $\sigma_v^-$ generating $\Ind_{P_v}^{G_v}|\,|^{s_0}$, one of which, say $\sigma_v^+$, is
generated by a spherical vector. Moreover, for each $\epsilon=+$ or $-$,
\[
	\tau_v^\epsilon:=\sigma_v^{\epsilon}/(\sigma_v^+\cap\sigma_v^-)
\]
is irreducible. Also when $s_0=0$, $\sigma_v^+\cap\sigma_v^-=0$ and $\Ind_{P_v}^{G_v}|\,|^{s_0}=\sigma_v^+\oplus\sigma_v^-$.
\end{prop}
\begin{proof}
See \cite[Proposition 3.3]{Ban_Jantzen} for the non-archimedean case, \cite[Theorem A.2.1]{Loke} for the real case and \cite[Theorem 1.3.2]{Lee-Zhu} for the complex case.
\end{proof}

The important fact we use about the structure of the degenerate
principal series, which easily follows from the above proposition, is
the following. First of all, for $s_0=\rho_{m,r}$,
\[
    \sigma_v^++\sigma_v^-=\Ind_{P_v}^{G_v}|\,|^{s_0}.
\]
Here the sum is not necessarily the direct sum. Now for each
$f\in\Ind_{P_v}^{G_v}|\,|^s$, let $\det\cdot f\in\Ind_{P_v}^{G_v}|\,|^s$ be
the function defined by
 \[ (\det\cdot f)(g)=\det(g)f(g). \]
  Then
\[
    \sigma_v^-=\{\det\cdot f: f\in\sigma_v^+\}.
\]
\noindent  Also notice that if $W_v$ is the symplectic space with $\dim W_v=2r$, then the image of the map
\[
    \Sw((\V_v^+\otimes W_v)(F_v))^\circ\xrightarrow{\Phi^{(m,r)}}\Ind_{P_v}^{G_v}|\,|^{s_0}
\]
is $\sigma_v^+$, because $\sigma_v^+$ is generated by the spherical vector. In fact, not only the image of $\Sw((\V_v^+\otimes W_v)(F_v))^\circ$ but of $\Sw((\V_v^+\otimes W_v)(F_v))$ is $\sigma_v^+$. This is because of the Howe duality for the trivial representation, which is known to be true for any residual characteristic. But we only need this weaker version for our purposes.
\vskip 10pt

Next, we consider the global degenerate principal series $\Ind_{P(\A)}^{G(\A)}|\,|^s$. For this, let
\[
	\epsilon:\{\text{places $v$ of $F$}\}\rightarrow \{+, -\}
\]
be a map such that $\epsilon(v)=+$ for almost all $v$. Then we say that
\[
    \text{$\epsilon$ is }\begin{cases}\text{{\bf coherent} if $\prod_v\epsilon(v)=+$, \ie $\epsilon(v)=-$ for an even number of $v$'s}\\
    \text{{\bf incoherent} if $\prod_v\epsilon(v)=-$, \ie $\epsilon(v)=-$ for an odd number of $v$'s}.
    \end{cases}
\]
Notice that by the above proposition, we know that for each $\epsilon$ the induced space $\Ind_{P(\A)}^{G(\A)}|\,|^{s_0}$ at $s_0=\rho_{m,r}$ has a submodule and an irreducible quotient which are respectively isomorphic to
\[
	\sigma(\epsilon):={\bigotimes_v}'\sigma_v^{\epsilon(v)}\quad\text{and}\quad \tau(\epsilon):={\bigotimes_v}'\tau_v^{\epsilon(v)}.
\]
Thus one has an equivariant projection
\[  \Ind_{P(\A)}^{G(\A)}|\,|^{s_0} \longrightarrow \bigoplus_{\epsilon} \tau(\epsilon).  \]
We say that
\[
    \text{$\Phi\in\Ind_{P(\A)}^{G(\A)}|\,|^s$ is }
    \begin{cases}
    \text{{\bf coherent} at $s=\rho_{m,r}$ if the image of $\Phi(-,\rho_{m,r})$ lies in
    $\bigoplus_{\text{$\epsilon$ coherent}} \tau(\epsilon)$}\\
    \text{{\bf incoherent} at $s=\rho_{m,r}$ if the image of $\Phi(-,\rho_{m,r})$  lies in
    $\bigoplus_{\text{$\epsilon$ incoherent}} \tau(\epsilon)$}.
    \end{cases}
\]
\vskip 10pt

\subsection{The residues of Siegel Eisenstein series.}

Recall that the Siegel Eisenstein series $E^{(m,m)}(g,s;\Phi)$ has a pole of order at most 1 at $s\in\{\hat{0},\dots,\rho_m-1,\rho_m\}$ where $\rho_m=\frac{m-1}{2}$, or equivalently it has a pole at $s=\rho_{m,r}=\frac{2r-m+1}{2}=\rho_{m,r}$ for $\frac{m}{2}\leq r\leq m-1$. Also recall that at each $s=\rho_{m,r}$ we write Laurent expansion of  $E^{(m,m)}(g,s;\Phi)$ as
\[  E^{(m,m)}(g,s;\Phi)=\sum_{d=-1}^\infty A^{(m,r)}_d(\Phi)(g)(s-\rho_{m,r})^d. \]
Then we have the following important proposition, which determines the image of the leading term
$A^{(m,r)}_{-1}$.
\vskip 10pt

\begin{prop}\label{P:incoherent}
Let $\Phi\in\Ind_{P(\A)}^{G(\A)}|\,|^s$. Then
\begin{enumerate}[(a)]
\item $A^{(m,r)}_{-1}(\Phi)=0$ if $\Phi$ is incoherent at $s=\rho_{m,r}$ where $\frac{m-1}{2}<r\leq m-1$.
In particular,
\[  \text{Image of $A^{(m,r)}_{-1}$} \cong \bigoplus_{\text{$\epsilon$ coherent}} \tau(\epsilon). \]
\vskip 5pt

\item $A^{(2r+1,r)}_{0}(\Phi)=0$ if $\Phi$ is incoherent at $s=\rho_{2r+1,r}=0$. In particular,
\[   \text{Image of $A^{(2r+1,r)}_{0}$} \cong  \bigoplus_{\text{$\epsilon$ coherent}} \tau(\epsilon). \]
\end{enumerate}
\end{prop}
\vskip 5pt

The rest of this subsection is devoted to the proof of this proposition. The proof is by induction on $m$. For this proof, since we work with $\OO(m,m)$ with various $m$, let us write
\[
    G_m = \OO(m,m)  \quad \text{and} \quad P_m =M_ m \cdot N_m =  \text{the Siegel parabolic subgroup of $G_m$}.
\]
When $m \geq 2$, recall that $Q= Q_m$ is the parabolic subgroup of $G_m$ whose Levi factor is
$L  = \GL_1 \times G_{m-1}$. We  shall also set
\[
    I_m(s) := \Ind_{P_m(\A)}^{G_m(\A)} |\det|^s  \quad \text{(normalized induction)}.
\]
We are interested in the behavior of the Siegel Eisenstein series $E^{(m,m)}(-,s; \Phi)$ associated to standard sections $\Phi_s$ of $I_m(s)$  for $s$ in the set
\[  \Sigma_m =\begin{cases}
 \{ \frac{1}{2}, \frac{3}{2},\dots, \rho_m \}  \text{  if $m$ is even;} \\
 \{  0,1,\dots, \rho_m \}  \text{  if $m$ is odd,} \end{cases}
\]
where $\rho_m = \frac{m-1}{2}$. A point in $\Sigma_m$ is of the form
\[
    \rho_{m,r} = \frac{2r - m+1}{2} \quad \text{for} \quad \frac{m}{2} \leq r \leq m-1.
\]
Also in this proof,  the only Eisenstein series we shall work with is the Siegel Eisenstein series, and so we simply write
\[
    E^m(s;\Phi):=E^{(m,m)}(-,s; \Phi),
\]
when there is no danger of confusion. The leading term of the Laurent expansion of $E^m(s; \Phi)$ at $s = \rho_{m,r}$ is denoted by
\[
    A_*^{(m,r)}(\Phi)= \begin{cases}
        A_{-1}^{(m,r)}(\Phi)  \text{  if $\rho_{m,r} > 0$;} \\
        A_0^{(m,r)}(\Phi) \text{  if $\rho_{m,r} = 0$.} \end{cases}
\]

Consider the base step of the induction $m =1$. In this case, $G_1 = \OO(1,1)$, $P_1 = \SO(1,1)$ and
$\Sigma_1 = \{ 0 \}$, so that we are in situation (b) of the proposition with $r = 0$ and $s= 0$. In this case,
\[
    I_1(0) =   \bigoplus_S  {\det}_S
\]
where $S$ ranges over finite subsets of the set of places of $F$ and
\[
    {\det}_S  = \left( \bigotimes_{v \in S} {\det}_v \right)  \otimes \left( \bigotimes_{v \notin S} 1_v\right).
\]
The incoherent submodules of $I_1(0)$ are spanned by those $\det_S$'s with $\# S$ odd. The Eisenstein series attached to $\Phi \in I_1(0)$ is given at $s = 0$ by
\[
    E^{(1,1)}(g, 0;\Phi) = \Phi(g) + \Phi(\epsilon g),
\]
where $\epsilon \in \OO(1,1)(F) \smallsetminus \SO(1,1)(F)$. It is thus clear that if $\Phi$ is incoherent, then $E(g,0,\Phi) =0$ because $\Phi(\epsilon g)=-\Phi(g)$. This proves the proposition when $m=1$, which is the base step of the induction.\\

Now consider the case $m > 1$.  We want to show that
$A_*^{(m,r)}(\Phi)$ is zero when $\Phi$ is incoherent. For this, it suffices to show that
\[
    A_*^{(m,r)}(\Phi)_Q = 0,
\]
where the LHS denotes the constant term of $A_*^{(m,r)}(\Phi)$ along $Q$. To compute $A_*^{(m,r)}(\Phi)$, we calculate the constant term of $E^m(s;\Phi)$ along $Q$ and take the relevant term in its Laurent expansion at $s = \rho_{m,r}$.
As an automorphic form of $G_{m-1}$, we have (cf. \cite[Appendix B]{GI}):
\[
    E^{m}_Q(s; \Phi)
    = E^{m-1}(s+\frac{1}{2}; \Phi|_{G_{m-1}}) + E^{m-1}(s-\frac{1}{2}; (M_{w_1}(s)\Phi )|_{G_{m-1}}),
\]
where $w_1$ is the Weyl group element given by
\[
    w_1 = \left( \begin{array}{ccc}
        0 & 0 & 1 \\
        0 & I_{2m-2} & 0 \\
        1 & 0 & 0 \end{array} \right).
\]
Moreover, $M_{w_1}(s)$ is the standard intertwining operator defined by
\[
     M_{w_1}(s)(\Phi)(g)  =  \int_{(U  \cap w_1 P w_1)\backslash U} \Phi_s(w_1 u g) \, du
\]
when $\Re(s)$ is sufficiently large. Thus, when $s =\rho_{m,r}$, we will be interested in the Laurent expansion of
\[
    E^{m-1}(s; \Phi|_{G_{m-1}})  \quad \text{at  $s = \rho_{m-1, r}$}
\]
and
\[
    E^{m-1}(s; (M_{w_1}(s +1/2)\Phi )|_{G_{m-1}}) \quad \text{at $s = \rho_{m-1, r-1}$}.
\]

We first note the following lemma:

\begin{lemma}
\quad
\begin{enumerate}[(i)]
\item Suppose that $\rho_{m-1,r} \in \Sigma_{m-1}$, i.e. $\rho_{m-1,r} \leq \rho_{m-1}$, then
\[   \text{$\Phi \in I_m(\rho_{m,r})$ incoherent $\Longrightarrow$ $\Phi|_{G_{m-1}} \in I_m(\rho_{m-1,r})$  incoherent.} \]

\item The intertwining operator $M_{w_1}(s)$ is holomorphic at all $s \in \Sigma_m$, except when $s = 0$ or $1/2$, in which case it has a pole of order $1$.

\item Denoting the leading term of the Laurent expansion of $M_{w_1}(s)$ at $s = \rho_{m,r} \in \Sigma_m$ by $M^{(m,r)}_{w_1}$, we have:
\[  \text{$\Phi \in I_m(\rho_{m,r})$ incoherent $\Longrightarrow$ $M^{(m,r)}_{w_1}\Phi|_{G_{m-1}} \in I_m(\rho_{m-1,r-1})$  incoherent,}\]
when $\rho_{m-1, r-1}\in \Sigma_m$, i.e. when $\rho_{m-1,r-1} \geq 0$.
\end{enumerate}
\end{lemma}
\begin{proof}
\quad
\begin{enumerate}[(i)]
\item This is obvious.
\vskip 5pt

\item
Suppose that $s_0 \in \Sigma_m$. For each place $v$ of $F$, the local representation $I_{m,v}(s_0)$ is generated as a $G_{m,v}$-module by the spherical vector $f_{0,v}$ and  its twist $f_{0,v} \cdot \det_v$. Hence, the global representation is generated as a $G_m(\A)$-module by
\[
    f_S =  \left( \otimes_{v \in S} f_{0,v} \cdot {\det}_v \right)  \otimes \left( \otimes_{v \notin S} f_{0,v} \right)
\]
as $S$ ranges over all finite sets of places of $F$. Hence the analytic behavior of $M_{w_1}(s)$ at $s = s_0$ is determined by the analytic behavior of the collection of $M_{w_1}(s)(f_S)$ at $s = s_0$. But for $\Re(s) \gg 0$,
\begin{align*}
    M_{w_1,v}(s)(f_{0,v} \cdot {\det}_v)(g)
    &=  \int_{(U \cap w_1Pw_1) \backslash U} f_{0,v}(w_1ug) \cdot{\det}_v (w_1ug) \, du\\
    &= -{\det}_v(g) \cdot M_{w_1,v}(s)(f_{0,v})(g).
\end{align*}
By meromorphic continuation, we thus have
\[
    M_{w_1,v}(s)(f_{0,v} \cdot {\det}_v) = -{\det}_v \cdot  M_{w_1,v}(s)(f_{0,v}).
\]
Globally, we deduce that
\[
    M_{w_1}(s)(f_S) = (-1)^{\#S} \cdot {\det}_S \cdot M_{w_1}(s)(f_0).
\]
Hence, for any $S$, the analytic behavior of $M_{w_1}(s)(f_S) $ at $s = s_0$ is the same as that of $M_{w_1}(s)(f_0)$ at $s = s_0$. But by Proposition \ref{P:constant_term}, we see that
\[
    M_{w_1}(s)(f_0)(1) = \frac{\zeta(2s)}{\zeta(2s+m-1)},
\]
which is holomorphic at all $s \in \Sigma_m$ except for $s = 0$ or $1/2$ where it has a pole of order
$1$.
\vskip 5pt

\item
For each place $v$ of $F$,  $M^{(m,r)}_{w_1,v}(f_{0,v})|_{G_{m-1}}$ is clearly, up to scaling, the spherical vector in $I_{m-1}(\rho_{m-1,r-1})$. Moreover, as we showed in (ii), $M^{(m,r)}_{w_1,v}(f_{0,v} \cdot {\det}_v)|_{G_{m-1}}$ is the twist by $\det_v$ of the spherical vector in $I_{m-1}(\rho_{m-1, r-1})$ (again up to scaling). Since $\sigma_{m,r,v}^+$ is generated by $f_{0,v}$ and $\sigma_{m,r,v}^-$ by $f_{0,v} \cdot \det_v$ as $G_{m,v}$-modules, we see that $M^{(m,r)}_{w_1,v}$ carries $\sigma_{m,r,v}^{+}$ (resp. $\sigma_{m,r,v}^{-}$) to $\sigma_{m-1,r-1,v}^{+}$ (resp. $\sigma_{m,r,v}^{-}$), and thus carries $\sigma_{m,r,v}^+ \cap \sigma_{m,r,v}^-$ to $\sigma_{m-1,r-1,v}^+ \cap \sigma_{m-1,r-1,v}^-$. This implies that, globally, one has a commutative diagram
\[  \begin{CD}
I_m(\rho_{m,r})  @>M^{(m,r)}_{w_1}>>  I_{m-1}(\rho_{m-1,r-1}) \\
@VVV   @VVV \\
\bigoplus_{\text{$\epsilon$ coherent}} \tau_m(\epsilon) @>>>
\bigoplus_{\text{$\epsilon$ coherent}} \tau_{m-1}(\epsilon),
\end{CD}
\]
where to be more precise the upper vertical arrow is given by $f_v\mapsto M_{w_1}^{(m,r)}(f_v)|_{G_{m-1}}$. This shows the desired result.
\end{enumerate}
\end{proof}

Going back to the proof of the proposition for $m > 1$, we first consider situation (a) so that $\rho_{m,r} > 0$. In this case, we have
\[  A_{-1}^{(m,r)}(\Phi)_Q = A_{-1}^{(m-1,r)}(\Phi|_{G_{m-1}}) +
A_*^{(m-1, r-1)}((M^{(m,r)}_{w_1}\Phi)|_{G_{m-1}}). \]
By the lemma and induction hypothesis, we know that the RHS of the above equation is zero. This establishes the induction step when $\rho_{m,r} >0$.\\

It remains to treat the case when $\rho_{m,r} = 0$, so that we are in situation (b) of the proposition.
In this case, $m = 2r+1$ is odd and we have
\[
    A^{(m,r)} _0(\Phi)_Q = A^{(m-1,r)}_{0}(\Phi|_{G_{m-1}})+  [E^{m-1}(s-\frac{1}{2}; (M_{w_1}(s)\Phi)|_{G_{m-1}})]_{s=0}.
\]
Examining the second term on the RHS above, we deduce by the functional equation for Eisenstein series that
\[
    E^{m-1}(s-\frac{1}{2}; (M_{w_1}(s)\Phi)|_{G_{m-1}}) = E^{m-1}(\frac{1}{2}-s; M_{m-1}(s-\frac{1}{2}) ((M_{w_1}(s)\Phi)|_{G_{m-1}})),
\]
where
\[
    M_{m-1}(s) : I_{P_{m-1}}(s) \longrightarrow I_{P_{m-1}}(-s)
\]
is the standard intertwining operator defined for $\Re(s) \gg 0$ by
\[  M_{m-1}(s)\Phi(g)= \int_{N_{m-1}} \Phi(w_2 ng) \, dn, \]
with
\[   w_2 = \begin{pmatrix}
 &&&&1\\
 &\mbox{\LARGE 0}&&1&\\
 &&\ddots&&\\
 &1&&\mbox{\LARGE 0}&\\
 1&&&&
 \end{pmatrix}.
\]
Now it is not difficult to see that
\[  M_{m-1}(s-\frac{1}{2}) ((M_{w_1}(s)\Phi)|_{G_{m-1}}) = (M_m(s) \Phi)|_{G_{m-1}}. \]
Hence,
\[  E^{m-1}(s-\frac{1}{2}; (M_{w_1}(s)\Phi)|_{G_{m-1}}) = E^{m-1}(\frac{1}{2} -s;   (M_m(s) \Phi)|_{G_{m-1}}).\]
We are interested in
\begin{equation} \label{E:bound}
  E^{m-1}(\frac{1}{2}+s; \Phi|_{G_{m-1}}) + E^{m-1}(\frac{1}{2} -s;   (M_m(s) \Phi)|_{G_{m-1}})
\end{equation}
at $s = 0$.

\begin{lemma}
\quad
\begin{enumerate}[(i)]
\item The intertwining operator $M_m(s): I_{P_m}(s) \longrightarrow I_{P_m}(-s)$ is holomorphic at $s=0$.
It acts as $+1$, \ie the identity, on the coherent submodules of $I_{P_m}(0)$ and as $-1$ on the incoherent submodules.

\item If $M_m'(0)$ denotes the derivative of $M_m(s)$ at $s=0$, then $M_m'(0)$ commutes with $M_m(0)$. In particular, $M_m'(0)$ preserves the incoherent submodule of $I_{P_m}(0)$.
\end{enumerate}
\end{lemma}
\begin{proof}
\quad
\begin{enumerate}[(i)]
\item We have the functional equation
\[  M_m(-s) \circ M_m(s)  = 1, \]
which implies that $M_m(s)$ is holomorphic at $s = 0$ and is nonzero as an operator.
Moreover, $M_m(0)^2 = 1$, so that $M_m(0)$ acts as $\pm 1$ on any irreducible submodule of
$I_{P_m}(0)$.

Now it is not difficult to see that if $f_0$ is the spherical vector in $I_m(0)$, then
$M_m(0) f_0 = f_0$. One way of seeing this is to observe that the functional equation for Eisenstein series
\[  E^m(-s; M_m(s)f_0)   = E^m(s; f_0) \]
implies that
\[  A^{(m,r)}_0(M_m(0)f_0)  = A^{(m,r)}_0(f_0). \]
As we observed in the proof of the spherical first term identity in \S \ref{S:identity} (see Proposition \ref{P:first_A0}), $A^{(m,r)}_0(f_0) \ne 0$, so that we must have  $M_m(0) f_0 = f_0$.
\vskip 10pt

On the other hand, for a place $v$ of $F$
\begin{align*}
  M_{m,v}(s)(f_{0,v}\cdot {\det}_v)(g) &= \int_{N_m} f_{0,v}(w_2 ng) \cdot {\det}_v(w_2ng) \, dn\\
    &= - M_{m,v}(s)(f_{0,v})(g) \cdot {\det}_v(g)\\
    &=-f_{0,v}(g) \cdot {\det}_v(g),
\end{align*}
since $\det(w_2) = -1$ as $m$ is odd. Thus, globally, we have
\[  M_m(0) (f_0 \cdot {\det}_S)  = (-1)^{\# S} \cdot f_0 \cdot {\det}_S \]
 for any finite subset of places of $F$. This shows that $M_m(0)$ acts as $+1$ on the coherent submodule of $I_m(0)$ which is generated by the $f_0 \cdot \det_S$'s for $\# S$ even and it acts as $-1$
 on the incoherent submodule which is generated by the $f_0 \cdot \det_S$'s for $\# S$ odd.
\vskip 5pt

\item Differentiating
\[  M_m(-s) \circ M_m(s)  = 1\]
with respect to $s$ and setting $s=0$ gives
\[  M_m(0) \circ M'_m(0) - M_m'(0) \circ M_m(0) = 0, \]
as desired.
\end{enumerate}
\end{proof}
\vskip 5pt

From the lemma, observe that both terms in (\ref{E:bound}) could have a pole of order $1$ at $s = 0$ but their residues there cancel (as they should). More relevantly, when $\Phi$ is incoherent, the constant term in the Laurent expansion at $s = 0$ of (\ref{E:bound}) is:
\[  A_0^{(m-1,r)}(\Phi|_{G_{m-1}})  - A_0^{(m-1,r)}(\Phi|_{G_{m-1}}) + A_{-1}^{(m-1,r)}((M_m'(0) \Phi)|_{G_{m-1}}). \]
The first two terms cancel whereas the third term vanishes by induction hypothesis since
$(M_m'(0) \Phi)|_{G_{m-1}}$ is incoherent.
Thus we conclude that $A^{(m,r)} _0(\Phi)_Q = 0$ when $\rho_{m,r} = 0$ as well.\\

We have thus shown that the leading term map $A_*^{m,r}$ vanishes on all incoherent $\Phi$'s. Hence
\[  \text{Image of $A_*^{m,r}$} \subset  \bigoplus_{\text{$\epsilon$ coherent}} \tau(\epsilon). \]
On the other hand, we know that  $A_*^{m,r}$ is nonzero on the spherical section, and thus by twisting by automorphic determinant characters, we see that equality must hold above.
This completes the proof of Proposition \ref{P:incoherent}\\


\subsection{Non-vanishing results}

In what follows, we will show our non-vanishing results for the theta lift. First let us prove a couple of lemmas which will be quite crucial for our proof.
\vskip 5pt

\begin{lemma}\label{L:twist}
Let $\pi_v$ be an irreducible admissible representation of
$\OO(V,F_v)$. (Here $v$ can be either archimedean or
non-archimedean.) Also let $\Phi_v(-,s)\in\Ind_{P_v}^{G_v}|\,|^s$ be
a ($K$-finite) standard section, and $\langle\pi_v(g_v)f_1,f_2\rangle$
a ($K$-finite) matrix coefficient of $\pi_v$.
\begin{enumerate}[(a)]
\item Assume $\pi_v\cong\pi_v\otimes\det$. Then there exist a ($K$-finite) matrix coefficient
$\langle\pi_v(g_v)f'_1,f'_2\rangle$ of $\pi_v$ so that
\[
    Z_v(s, f_1, f_2, \Phi_v, \pi_v)=Z_v(s, f'_1, f'_2, \det\cdot\Phi_v, \pi_v).
\]

\item Assume $\pi_v\ncong\pi_v\otimes\det$. Then
\[
    Z_v(s, f_1, f_2, \Phi_v, \pi_v)=Z_v(s, f_1, f_2,\det\cdot\Phi_v, \pi_v\otimes\det).
\]
\end{enumerate}
\end{lemma}
\begin{proof}
Note that
\begin{align*}
    Z_v(s, f_1, f_2,\Phi_v, \pi_v)
    &=\int_{\OO(V,F_v)}\Phi_v(i(g_v,1))\langle\pi_v(g_v)f_1,f_2\rangle\,dg_v\\
    &=\int_{\OO(V,F_v)}\det(i(g_v,1))\Phi_v(i(g_v,1))\langle\det(g_v)\pi_v(g_v)f_1,f_2\rangle\,dg_v.
\end{align*}
So this proves the lemma if $\pi_v\ncong\pi_v\otimes\det$. Now if
$\pi_v\cong\pi_v\otimes\det$, then
$\langle\det(g_v)\pi_v(g_v)f_1,f_2\rangle$ is a matrix coefficient of
$\pi_v$ and so written as
$\langle\det(g_v)\pi_v(g_v)f_1,f_2\rangle=\langle\pi_v(g_v)f'_1,f'_2\rangle$
for some $f'_1$ and $f'_2$. Hence the lemma follows.
\end{proof}

\begin{lemma}\label{L:spherical_function}
Let $\Phi(-,s)=\otimes'\Phi_v(-,s)\in\Ind_{P}^{G}|\,|^s$ be a
standard section such that for every place $v$,
$\Phi_v(-,\rho_{m,r})\in\sigma_v^+$. Then there is a
Schwartz function $\varphi\in\Sw((\V^+\otimes W)(\A))^\circ$ of the
form $\varphi=\sum_i\sigma(\phi_{1,i}\otimes\overline{\phi_{2,i}})$
such that $\Phi=\Phi^{(m,r)}_{\varphi}$.
\end{lemma}
\begin{proof}
Since $\Phi_v (-,\rho_{m,r})\in\sigma_v^+$, for each $v$ we have
$\Phi_v(-,\rho_{m,r})=\Phi^{(2r+j,r)}_{\varphi_v}(-,\rho_{m,r})$ for
some $\varphi_v$ which is in the $\OO(\V,F_v)$-span of the spherical
Schwartz function $\varphi_{0,v}$. Clearly, we can
choose $\varphi_v=\varphi_{0,v}$ for almost all $v$. Thus if we let
$\varphi=\otimes'\varphi_v$, we have $\Phi=\Phi^{(m,r)}_{\varphi}$,
which is of the form
$\varphi=\sum_i\sigma(\phi_{1,i}\otimes\overline{\phi_{2,i}})$.
\end{proof}

Now we are ready to prove our non-vanishing theorems. For this purpose, we need to introduce the following notation. Let $\mu_2$ be the group of order two and so $\mu_2(\A)=\prod_v\{\pm 1\}$. For a finite set of places $T$, we define the character $\sign(T)$ on $\mu_2(\A)$ by
\[
	\sign(T):=\otimes'\sign(T)_v
\]
where
\[
	\sign(T)_v(-1)=\begin{cases}1\quad\text{if $v\notin T$}\\
				-1\quad\text{if $v\in T$}.
			\end{cases}
\]
Notice that $\sign(T)$ is an automorphic character on $\mu_2(F)\backslash\mu_2(\A)$ if and only if the cardinality $|T|$ of $T$ is even. And any automorphic character $\mu$ on $\mu_2(F)\backslash\mu_2(\A)$ is of this form. From now on, by a character $\mu$ we always mean an automorphic character on $\mu_2(F)\backslash\mu_2(\A)$. For each character $\mu$ and an automorphic representation $\pi$ of $\OO(V,\A)$, we write
\[
    \pi\otimes\mu:=\pi\otimes(\mu\circ\det),
\]
which is again automorphic.
\vskip 10pt

We first prove the following non-vanishing result for the 1st term range. We should mention that the same theorem has been proven by Ginzburg-Jiang-Soudry in their recent paper ([GJS, Thm. 1.1]) by following Moeglin (\cite{Moeglin}) and using her method of ``generalized doubling method". However, we will show that the same theorem quite simply follows from our inner product formula together with Proposition \ref{P:incoherent}.
\vskip 5pt

\begin{thm}[Non-vanishing theorem for 1st term range]\label{T:non_vanishing_1st_term}
Let $\pi$ be a cuspidal automorphic representation of $\OO(V,\A)$
for $\dim V=m$ where $m\geq 2r+1$.
\begin{enumerate}[(a)]
\item Assume $m> 2r+1$. Further assume that the (incomplete) standard $L$-function $L^S(s,\pi)$ has a pole at $s=\frac{m-2r}{2}$. Then there is a character $\mu$ on $\mu_2(F)\backslash\mu_2(\A)$ such
that $\theta_{2r}(\pi\otimes\mu)\neq0$.
\item Assume $m=2r+1$. Further assume that the (incomplete) standard $L$-function $L^S(s,\pi)$ does not vanish at $s=\frac{1}{2}$. Then there is a character $\mu$ on $\mu_2(F)\backslash\mu_2(\A)$ such
that $\theta_{2r}(\pi\otimes\mu)\neq0$.
\end{enumerate}
\end{thm}
\begin{proof}
First let us assume $m> 2r+1$ \ie situation (a). Suppose that there is a character $\mu$ such that $\theta_{2r-2}(\pi\otimes\mu)\neq0$. Then by the tower property of theta lifting, the theorem follows. Hence we assume that $\theta_{2r-2}(\pi\otimes\mu)=0$ for all $\mu$, and hence $\theta_{2r}(\pi\otimes\mu)$ is in the space of cusp forms (possibly zero) for all $\mu$.
\vskip 5pt

Now it is known that if $v$ is non-archimedean, one can always find a standard section $\Phi_v$ and vectors $f_{1,v}$ and $f_{2,v}$ so that
\[
    Z_v(-\rho_{m,r}, f_{1,v}, f_{2,v}, \Phi_v, \pi_v)=1,
\]
and if $v$ is archimedean, one can always find a $K$-finite section $\Phi_v$ and  vectors $f_{1,v}$ and $f_{2,v}$ so that
\[
    Z_v(s, f_{1,v}, f_{2,v}, \Phi_v,\pi_v)
\]
either has a pole or is non-zero holomorphic at $s=-\rho_{m,r}$. (See \cite[Theorem 2.0.2 and 2.0.3]{Kudla-Rallis90}.) Moreover by Lemma \ref{L:twist}, if $\pi_v\cong\pi_v\otimes\det$, then we can find a suitable $f'_{i,v}$ so that
\[
    Z_v(s, f_{1,v}, f_{2,v}, \Phi_v,\pi_v)=Z_v(s, f'_{1,v}, f'_{2,v}, \det\cdot\Phi_v,\pi_v).
\]
Since if $\Phi_v(-,-\rho_{m,r})\in\sigma_v^-$ then $\det\cdot\Phi_v(-,-\rho_{m,r})\in\sigma_v^+$, we may assume that at those $v$'s where $\pi_v\cong\pi_v\otimes\det$ the above choice of $\Phi_v$ is such that $\Phi_v(-,-\rho_{m,r})\in\sigma_v^+$. Therefore one can find a global standard factorizable section $\Phi=\otimes'\Phi_v$, which can be chosen to be either coherent or incoherent at $s=-\rho_{m,r}$, and vectors $f_1=\otimes'f_{1,v}$ and $f_2=\otimes'f_{2,v}$
so that the product
\[
    \prod_{v\in S}Z_v(s, f_{1,v}, f_{2,v},\Phi_v,\pi_v)
\]
either has a pole or is non-zero holomorphic at $s=-\rho_{m,r}$ for sufficiently large $S$ containing all the archimedean $v$'s, and all the non-archimedean $v$'s where $\pi_v$ is ramified, and moreover $\Phi_v(-,-\rho_{m,r})\in\sigma_v^+$ if $\pi_v\cong\pi_v\otimes\det$.
\vskip 5pt

Now by the doubling method, for this choice of $\Phi$, $f_1$ and $f_2$, we have
\begin{align*}
    &\int_{(G\times G)(F)\backslash(G\times
    G)(\A)}f_1(g_1)\overline{f_2(g_2)}E^{(m,m)}(i(g_1,g_2),s;\Phi)\,dg_1dg_2\\
    &=\frac{1}{d^S_{m}(s)}L^S(s+\frac{1}{2}, \pi)\prod_{v\in S}Z_v(s,f_{1,v}, f_{2,v},
    \Phi_v,\pi_v).
\end{align*}
Now at $s=\rho_{m,m-r-1}=-\rho_{m,r}$, the left hand side has at most a simple pole. Hence by our assumption that $L^S(s,\pi)$ has a pole at $s=\frac{m-2r}{2}$, we know that the product of local zeta integrals on the right hand side cannot have a pole and so must be non-zero holomorphic. Then we have
\begin{align*}
    &\int_{(G\times G)(F)\backslash(G\times
    G)(\A)}f_1(g_1)\overline{f_2(g_2)}A_{-1}^{(m,m-r-1)}(\Phi)(i(g_1,g_2))\,dg_1dg_2\\
    &=\frac{1}{d^S_{m}(-\rho_{m,r})}\underset{s=-\rho_{m,r}}{\Res}\left(L^S(s+\frac{1}{2}, \pi)\prod_{v\in S}Z_v(s,f_{1,v}, f_{2,v},\Phi_v,\pi_v)\right),
\end{align*}
where the right hand side and hence the left hand side are non-zero, and in particular
\[
    A_{-1}^{(m,m-r-1)}(\Phi)\neq 0.
\]
Therefore by Proposition \ref{P:incoherent}, we know that $\Phi$ is a coherent section. Hence if we let
\[
    T=\{v\in S: \Phi_v(-,-\rho_{m,r})\notin\sigma^+\},
\]
then the size $|T|$ of $T$ is even by the coherence of $\Phi$, and so if we let $\mu=\sign(T)$, then $\mu$ is automorphic. Since by our choice of $\Phi$, we know that $\pi_v\ncong\pi_v\otimes\det$ for all $v\in T$, and so by Lemma \ref{L:twist} we have
\[
	\prod_{v\in T}Z_v(s,f_{1,v}, f_{2,v}, \Phi_v,\pi_v)=\prod_{v\in T}Z_v(s,f_{1,v}, f_{2,v}, \det\cdot\Phi_v,\pi_v\otimes\det),
\]
which, together with $\mu=\sign(T)$, gives
\[
	\prod_{v\in S}Z_v(s,f_{1,v}, f_{2,v}, \Phi_v,\pi_v)=\prod_{v\in S}Z_v(s,f_{1,v}, f_{2,v}, \mu_v\cdot\Phi_v,\pi_v\otimes\mu_v).
\]
Therefore since $L^S(s+\frac{1}{2}, \pi)=L^S(s+\frac{1}{2}, \pi\otimes\mu)$, we see that
\[
	L^S(s+\frac{1}{2}, \pi\otimes\mu)\prod_{v\in S}Z_v(s,f_{1,v}, f_{2,v},\mu_v\cdot\Phi_v,\pi_v\otimes\mu_v)
\]
has a simple pole at $s=-\rho_{m,r}$.
\vskip 5pt

Now by the very definition of $\mu$, we have
\[
	\mu_v\cdot\Phi_v\in\sigma_v^+\quad\text{for all $v$},
\]
and so there is a $\phi'\in\Sw((\V^+\otimes W_c)(\A))^\circ$ such that $\Phi^{(m,m-r-1)}_{\varphi'}=\mu\cdot\Phi$ by Lemma \ref{L:spherical_function}. Hence by invoking Theorem \ref{T:inner_product3} (Inner Product Formula for 1st term range), we obtain
\begin{align*}
    &\sum_i\langle\theta_{2r}(f'_1,\phi_{1,i}),\theta_{2r}(f'_2,\phi_{2,i})\rangle\\
    &\qquad\qquad=\frac{a_{m,r}}{d^S_{m}(-\rho_{m,r})}\underset{s=-\rho_{m,r}}{\Res}\left(L^S(s+\frac{1}{2}, \pi\otimes\mu)\prod_{v\in S}Z_v(s,f_{1,v}, f_{2,v}, \mu_v\cdot\Phi_v, \pi_v\otimes\mu_v)\right).
\end{align*}
for some $\sum_i\sigma(\phi_{1,i}\otimes\overline{\phi_{2,i}})\in\Sw((\V^+\otimes W)(\A))^\circ$ and for some $f'_{i}\in\pi\otimes\mu$. Since  the right hand side is non-zero, we see that
\[
    \langle\theta_{2r}(f'_1,\phi_{1,i}),\theta_{2r}(f'_2,\phi_{2,i})\rangle\neq 0
\]
for some $i$. Thus part (a) of the theorem has been proven.
\vskip 10pt

The part (b) of the theorem can be similarly shown by using the inner product formula on the boundary, and the detail of the proof is left to the reader.
\end{proof}
\vskip 10pt

Next we prove our non-vanishing result for the 2nd term range. Unfortunately, however, for this range we need to assume that $\dim V=m$ is even. This is because we need to impose on the cuspidal representation $\pi$ the assumption that $\pi_v\cong\pi_v\otimes\det$ for at least one place $v$, and if $\dim V$ is odd, then $\pi_v\ncong\pi_v\otimes\det$ for all $v$. Also the reason we can get away with this assumption for the 1st term range is Proposition \ref{P:incoherent}. For the 2nd term range, we have not been able to prove the analogous proposition and we do not even know if such a statement even exists at all. In any case, the following is our main non-vanishing theorem for the 2nd term range.
\vskip 10pt

\begin{thm}[Non-vanishing theorem for 2nd term range]\label{T:non_vanishing_2nd_term}
Let $\pi$ be a cuspidal automorphic representation of $\OO(V,\A)$
for $\dim V=m$ where $m$ is even with $r+1\leq m\leq 2r$. Assume that

\begin{enumerate}
\item[(i)] there is a place $v$ such that
$\pi_v\cong\pi_v\otimes\det$;
\vskip 5pt

\item[(ii)] $L^S(1+\frac{2r-m}{2},\pi)\neq 0$. (A pole is allowed.)
\end{enumerate}
Then there is a character $\mu$ on $\mu_2(F)\backslash\mu_2(\A)$ such
that
\[
    \theta_{2r}(\pi\otimes\mu)\neq0.
\]
\end{thm}
\vskip 10pt

\begin{proof}
One can prove this theorem analogously to Theorem \ref{T:non_vanishing_1st_term} by using the inner product formula of the 2nd term range. However, this time the set $T$ as defined in the proof of Theorem \ref{T:non_vanishing_1st_term} cannot be shown to have an even cardinality due to the absence of the analogue of Proposition \ref{P:incoherent}. To get around this, we use a place $v$ where $\pi_v\cong\pi_v\otimes\det$. The detailed proof is as follows.
\vskip 5pt

First if $\theta_{2r-2}(\pi\otimes\mu)\neq 0$ for some character $\mu$, then the theorem holds by the tower property of theta lifting. Also if the (incomplete) $L$-function has a pole at $s=\rho_{m,r}$, then by Theorem \ref{T:non_vanishing_1st_term} we already know that $\theta_{2r-2\rho_{m,r}}(\pi\otimes\mu)\neq0$, and so again by the tower property $\theta_{2r-2k}(\pi\otimes\mu)\neq0$ for $k\leq 2\rho_{m,r}$ and so the theorem follows. Hence for the rest of the proof, we assume that $\theta_{2r-2}(\pi\otimes\mu)=0$ for all $\mu$ and the (incomplete) $L$-function is non-zero holomorphic at $s=\rho_{m,r}$.
\vskip 10pt

Now by the same reasoning as the proof of Theorem \ref{T:non_vanishing_1st_term}, there is
a global $K$-finite factorizable standard section $\Phi=\otimes'\Phi_v$ and vectors
$f_{1}=\otimes'f_{1,v}$ and $f_2=\otimes'f_{1,v}$ so that
\[
    \prod_{v\in S}Z_v(s, f_{1,v}, f_{2,v},\Phi_v,\pi_v)
\]
is either non-holomorphic or non-zero holomorphic at
$s=\rho_{m,r}$ for $S$ as in the proof of Theorem \ref{T:non_vanishing_1st_term}. Moreover we may assume that $\Phi_v(-,\rho_{m,r})\in\sigma_v^+$ for all $v$'s at which $\pi_v\cong\pi_v\otimes\det$. Then let
\[
    T=\{v\in S: \Phi_v(-,\rho_{m,r})\notin\sigma^+\},
\]
and define
\[
	\mu=\begin{cases}\sign(T)\quad\text{if $|T|$ is even}\\
			\sign(T\cup\{v_0\})\quad\text{if $|T|$ is odd},
		\end{cases}
\]
where $v_0$ is a place at which $\pi_{v_0}\cong\pi_{v_0}\otimes\det$. Then $\mu$ is an automorphic character.

Then by invoking the inner product formula for the 2nd term range, for suitable choice of $f'_i\in\pi\otimes\mu$ we have
\begin{align*}
    &\sum_i\langle\theta_{2r}(f'_1,\phi_{1,i}),\theta_{2r}(f'_2,\phi_{2,i})\rangle\\
    &=\frac{1}{d^S_{m}(\rho_{m,r})}\underset{s=\rho_{m,r}}{\Val}\left(L^S(s+\frac{1}{2}, \pi\otimes\mu)
       \prod_{v\in S}Z_v(s, f_{1,v}, f_{2,v}, \mu_v\cdot\Phi_v,\pi_v\otimes\mu_v)\right)
\end{align*}
where $\varphi=\sum_i\sigma(\phi_{1,i}\otimes\overline{\phi_{2,i}})\in\Sw((\V^+\otimes W)(\A))^\circ$ is such that $\Phi^{(m,r)}_{\varphi}=\Phi$ by Lemma \ref{L:spherical_function}. Notice that in the right hand side, the product
\[
	L^S(s+\frac{1}{2}, \pi\otimes\mu)
       \prod_{v\in S}Z_v(s, f_{1,v}, f_{2,v}, \mu_v\cdot\Phi_v,\pi_v\otimes\mu_v)
\]
is holomorphic by Lemma \ref{L:no_residue} at $s=\rho_{m,r}$, and thus non-zero holomorphic at $s=\rho_{m,r}$ because the product of the local factors is non-zero and $L^S(s+\frac{1}{2}, \pi\otimes\mu)=L^S(s+\frac{1}{2}, \pi)$. Therefore the theorem has been proven.
\end{proof}

The above two theorems are viewed as non-vanishing results ``up to disconnectedness". Indeed, if we use the ``dual pair" $(\SO(V), \Sp(2r))$ instead of the usual $(\OO(V), \Sp(2r))$, the above theorems are restated as follows:

\begin{cor}
Let $\tau$ be a cuspidal automorphic representation of $\SO(V,\A)$
with $\dim V=m$.
\begin{enumerate}[(a)]
\item Assume $m\geq 2r+2$. Further assume that the (incomplete) standard $L$-function $L^S(s,\tau)$ has a pole at $s=\frac{m-2r}{2}$. Then $\theta_{2r}(\tau)\neq0$.
\item Assume $m=2r+1$. Further assume that the (incomplete) standard $L$-function $L^S(s,\pi)$ does not vanish at $s=\frac{1}{2}$. Then $\theta_{2r}(\tau)\neq0$.
\item Assume $m$ is even with $r+1\leq m\leq 2r$ such that $\tau^c \ncong\tau$. Assume that $L^S(1+\frac{2r-m}{2},\tau)\neq 0$. (A pole is allowed.) Then $\theta_{2r}(\tau)\neq0$. (Here $\tau^c$ is the representation obtained by conjugating $\tau$ via the outer automorphism on $\SO(V)$.)
\end{enumerate}
\end{cor}

\begin{rmk}
If we assume $\tau$ is generic, then automatically $L^S(1+\frac{2r-m}{2},\tau)\neq 0$ for $m=2r$, and so part (c) of this corollary implies $\theta_{m}(\tau)\neq0$ (if  $\tau^c \ncong\tau$). This is a part of the main theorem of \cite{GRS}, though their theorem applies even when $\tau^c\cong\tau$.
\end{rmk}

\quad\\



\section{\bf On the lowest occurrence conjecture}


Once the non-vanishing theorems in the previous section have been proven, the natural question to ask is whether the $L$-function conditions are necessary. In this last section, we will examine this issue. First we show that there is at least an example in which the $L$-function condition is necessary. Namely, we construct a cuspidal automorphic representation $\tau$ on $\SO(4)$ so that $L^S(s,\tau)$ vanishes at $s=1$ and $\theta_4(\tau)=0$. After that, following \cite{GJS}, we will consider the non-vanishing of theta lifts by using the notation of the lowest occurrence as in \cite{GJS}.


\subsection{An example}


We will construct an example as mentioned above. Let $V=D$ be a non-split quaternion division algebra over $F$. Then first consider the similitude group $\GSO(D)$. It is well-known that $\GSO(D)\cong (D^\times\times D^\times)/{\Gm}$ where $\Gm$ is embedded as $\{ (t,t^{-1}); t\in \Gm \}$. Hence, a cuspidal automorphic representation $\taut$ of $\GSO(D)(\A)$ is identified with a cuspidal automorphic representation $\tau_1\boxtimes\tau_2$ of $D^\times(\A)\times D^\times(\A)$ so that $\tau_1$ and $\tau_2$ have the same central character. Now assume
\[
    \tau_1\boxtimes\tau_2=\tau_1\boxtimes\mathbf{1},
\]
\ie $\tau_1$ has the trivial central character and $\tau_2$ is the trivial character $\mathbf{1}$ on $D^\times(\A)$. This is certainly a cuspidal automorphic representation since $\GSO(D)$ is anisotropic. Then let $\tau$ be an irreducible constituent of $\tau_1\boxtimes\tau_2|_{\SO(D,\A)}$. (Here the restriction refers to the restriction of functions.) Then one sees
\[
    L^S(s, \tau)=L^S(s+\frac{1}{2},\tau_1)\cdot L^S(s-\frac{1}{2},\tau_1).
\]
(See the second paragraph of the proof of Theorem 5.4 in \cite{Schmidt}.) Hence if we choose $\tau_1$ to be such that $L^S(\frac{1}{2},\tau_1)=0$, then we have $L^S(1,\tau)=0$.
\vskip 5pt

Now we have

\begin{prop}
Let $\tau$ be the cuspidal automorphic representation on $\SO(D,\A)$ constructed above. Then $\theta_4(\tau)=0$. Equivalently, if $\pi$ is a cuspidal automorphic representation on $\OO(D)$ such that $\pi|_{\SO(D,\A)}$ (restriction of functions) contains $\tau$, then for any automorphic character $\mu$ on $\mu_2(F)\backslash\mu_2(\A)$, we have $\theta_4(\pi\otimes\mu)=0$.
\end{prop}
\begin{proof}
To prove the proposition, it suffices to show that the global theta lift of  $\tau \boxtimes \tau_2= \tau_1 \boxtimes {\bf 1}$ from $\GSO(D)$ to $\GSp(4)$ is zero. For this, we note that if $\theta(\tau_1 \boxtimes \tau_2)$ is nonzero, it must have a nonzero Bessel model associated to a pair $(E, \chi)$ where $E$ is a quadratic field extension of $F$ and $\chi$ is an automorphic character of $E^{\times}$ (cf. \cite{PT} for the definition of Bessel models). By \cite[Theorem 3]{PT}, the Bessel model of $\theta(\tau_1 \boxtimes \tau_2)$ with respect to the pair $(E, \chi)$ is nonzero if and only if both $\tau_1$ and $\tau_2$ have nonzero period integrals against the character $\chi^{-1}$ of the maximal torus $E^{\times} \hookrightarrow D^{\times}$. Since $\tau_2$ is the trivial representation of $D^{\times}$, the $(E,\chi)$-Bessel model of $\theta(\tau_1 \boxtimes \tau_2)$ is nonzero if and only if $\chi$ is trivial and $\tau_1$ has nonzero period integral over $E^{\times}$. By a well-known result of Waldspurger \cite{W}, this holds if and only if
\[
    L(\frac{1}{2}, \tau_1) \cdot L(\frac{1}{2}, \tau_1 \otimes \chi_E) \ne 0
\]
where $\chi_E$ is the quadratic character associated to $E/F$ by class field theory.
To summarize, we have shown that a necessary condition for the nonvanishing of $\theta(\tau_1 \boxtimes {\bf 1})$ is the nonvanishing of $L(\frac{1}{2}, \tau_1)$. Since we are assuming that
$L(\frac{1}{2}, \tau_1) =0$, we conclude that $\theta(\tau_1 \boxtimes {\bf 1}) = 0$.
\end{proof}


\subsection{On the lowest occurrence conjecture}


Now we will consider the non-vanishing problem following \cite{GJS}. First as in \cite{GJS}, we define the notion of lowest occurrence. Let $\pi$ be a cuspidal automorphic representation of $\OO(V,\A)$, where $\dim V=m$. Then define
\[
	\FO(\pi):=\min\{r:\theta_{\psi,2r}(\pi)\neq 0\},
\]
namely $\FO(\pi)$ is the rank of first occurrence of the theta lift to the symplectic tower. Then we also define
\[
	\LO(\pi):=\underset{\mu}{\min}\{\FO(\pi\otimes\mu)\},
\]
where $\mu$ runs through all the automorphic characters on $\mu_2(F)\backslash\mu_2(\A)$, namely $\LO(\pi)$ is the rank of first occurrence ``up to disconnectedness", and following \cite{GJS} we call it the rank of ``lowest occurrence". Note that until this point, we have been fixing the additive character $\psi$, but in this section we consider $\FO(\pi)$ and $\LO(\pi)$ varying as $\psi$ varies.

Next recall that for each irreducible admissible representation $\pi_v$ of $\OO(V,F_v)$, Lapid and Rallis (\cite{Lapid_Rallis}) have defined the local $L$-factor $L_v(s,\pi_v)$ in such a way that it satisfies a number of expected properties which determine it uniquely. Hence it makes sense to consider the complete (standard) $L$-function $L(s,\pi)$ of $\pi$. Then define
\[
	\Pole(\pi):=
	\begin{cases}\max\{s_0\in\{1,2,\dots,\frac{m}{2}-1\}: L(s,\pi) \text{ has a pole at $s=s_0$}\}
				\quad\text{if $m$ is even}\\
		      \max\{s_0\in\{\frac{3}{2},\frac{5}{2},\dots,\frac{m}{2}-1\}: L(s,\pi)
				\text{ has a pole at $s=s_0$}\}\quad\text{if $m$ is odd},
	\end{cases}
\]
whenever it exists. Now if $\Pole(\pi)$ does not exist, \ie the complete $L$-function $L(s,\pi)$ does not have a pole in the above set, then define
\[
	\Zero(\pi):=
	\begin{cases}\min\{s_0\in\{1,2,\dots,\frac{m}{2}\}: L(s,\pi) \text{ does not vanish at $s=s_0$}\}
				\quad\text{if $m$ is even}\\
		      \min\{s_0\in\{\frac{1}{2}, \frac{3}{2},\dots,\frac{m}{2}-1\}: L(s,\pi)
				\text{ does not vanish at $s=s_0$}\}\quad\text{if $m$ is odd}.
	\end{cases}
\]
If $\Zero(\pi)$ is not defined in this way, namely the complete $L$-function $L(s,\pi)$ vanishes at all the $s_0$ in the above set, then we define
\[
	\Zero(\pi):=
	\begin{cases}\frac{m}{2}+1 \quad\text{if $m$ is even}\\
		      \frac{m}{2} \quad\text{if $m$ is odd},
	\end{cases}
\]
so that whenever $\Pole(\pi)$ does not exist, $\Zero(\pi)$ exists.
\vskip 10pt

However, as we have see in the previous section, our non-vanishing theorems are stated in terms of the incomplete $L$-function $L^S(s,\pi)$. So similarly to $\Pole(\pi)$ and $\Zero(\pi)$, let us define its incomplete analogue. Namely for a finite set $S$ of places,
\[
	\Pole^S(\pi):=
	\begin{cases}\max\{s_0\in\{1,2,\dots,\frac{m}{2}-1\}: L^S(s,\pi) \text{ has a pole at $s=s_0$}\}
				\quad\text{if $m$ is even}\\
		      \max\{s_0\in\{\frac{3}{2},\frac{5}{2},\dots,\frac{m}{2}-1\}: L^S(s,\pi)
				\text{ has a pole at $s=s_0$}\}\quad\text{if $m$ is odd},
	\end{cases}
\]
whenever it exists, and when it does not exist we define $\Zero^S(\pi)$ in the analogous way as $\Zero(\pi)$. Those definitions make sense for any finite $S$, and so $\Pole^S(\pi)=\Pole(\pi)$ and $\Zero^S(\pi)=\Zero(\pi)$ if $S$ is empty, though in practice we only consider the case where $S$ contains all the ``bad" places. Clearly,
\[
    \Pole(\pi)\geq\Pole^S(\pi)\quad\text{and}\quad\Zero(\pi)\leq\Zero^S(\pi),
\]
whenever both sides exist, because each local factor $L(s,\pi_v)$ never has a zero.
\vskip 5pt

In any case, using this language, let us state the lowest occurrence conjecture by Ginzburg-Jiang-Soudry [GJS, Conjecture 1.2] as follows.
\begin{conj}[Ginzburg-Jiang-Soudry]\label{C:main}
Let $\pi$ be a cuspidal automorphic representation of $\OO(V,\A)$ with $\dim V=m$ and
assume that $\Pole^S(\pi)$ exists. Then
\[
	\LO(\pi)=\frac{m}{2}-\Pole^S(\pi).
\]
\end{conj}
Note that in particular, this conjecture implies that for $\pi$ satisfying the assumption of the conjecture, $\LO(\pi)$ is independent of the choice of $\psi$.
\vskip 10pt

Now the non-vanishing theorems we proved in the previous section can be stated as follows, recalling that part (a) is a theorem of Ginzburg-Jiang-Soudry [GJS, Thm. 1.1].
\vskip 5pt

\begin{thm}
Let $\pi$ be a cuspidal automorphic representation of $\OO(V,\A)$, $S$ a finite set of places containing all the archimedean places and the ramified places.
\begin{enumerate}[(a)]
\item Assume $\Pole^S(\pi)$ exists. Then
\[
	\LO(\pi)\leq\frac{m}{2}-\Pole^S(\pi).
\]
\item Assume $\Pole^S(\pi)$ does not exist. Further assume $\dim V$ is even, and there is a place $v$ such that $\pi_v\cong\pi_v\otimes\det$. Then
\[
	\LO(\pi)\leq\frac{m}{2}-1+\Zero^S(\pi).
\]
\end{enumerate}
\end{thm}
\vskip 5pt

If one could show that the $\leq$ is actually $=$ for the first case, then the conjecture by \cite{GJS} would be proven. Also analogously to their conjecture, one might wonder if the same is true for the second case. In what follows, we will investigate this issue.
\vskip 10pt

First of all, to make $\leq$ into $=$ in the above theorem is almost synonymous to the converse of our non-vanishing theorem. However one difficulty is in how to control the analytic behavior of the ``bad factors" of the local zeta integrals present in the inner product formula. One possible way to get around this difficulty is to consider the complete $L$-function and assume the following expected property of the local $L$-factor.

\begin{conj}
Let $\pi_v$ be an irreducible admissible representation of $\OO(V,F_v)$, where $v$ can be archimedean or non-archimedean. Then the normalized zeta integral
\[
	Z_v^*(s,f_{1,v},f_{2,v},\Phi_v,\pi_v):=\frac{Z_v(s,f_{1,v},f_{2,v},\Phi_v,\pi_v)}{L_v(s+\frac{1}{2},\pi_v)}
\]
is holomorphic for all $s\in\C$ and for any choice of $f_{i,v}$ and standard $\Phi_v$.
\end{conj}

We will take up this issue in our later work. However for the purpose of this paper, we do not need the full strength of this conjecture but the following weaker version suffices.

\begin{conj}\label{C:GCD}
Let $\pi_v$ be an irreducible admissible representation of $\OO(V,F_v)$, where $v$ can be archimedean or non-archimedean. The we have the following.
\begin{enumerate}[(a)]
\item The normalized zeta integral $Z_v^*(s,f_{1,v},f_{2,v},\Phi_v,\pi_v)$ is holomorphic at all $s\in\C$ such that
\[
    s+\frac{1}{2}\in \begin{cases}\{1,2,\dots,\frac{m}{2}-1\}&\text{if $m$ is even}\\
                                  \{\frac{3}{2},\frac{5}{2},\dots,\frac{m}{2}-1\}&\text{if $m$ is odd}.
                     \end{cases}
\]
\item The normalized zeta integral $Z_v^*(s,f_{1,v},f_{2,v},\Phi_v,\pi_v)$ is holomorphic at all $s\in\C$ such that
\[
    s+\frac{1}{2}\in \begin{cases}\{1,2,\dots,\frac{m}{2}+1\}&\text{if $m$ is even}\\
                                  \{\frac{1}{2},\frac{3}{2},\dots,\frac{m}{2}\}&\text{if $m$ is odd}.
                     \end{cases}
\]
\end{enumerate}
\end{conj}
\vskip 10pt

The reason we have two slightly different versions is that, as we will see below, when both $\Pole^S(\pi)$ and $\Pole(\pi)$ exist, the first version suffices, and when both $\Zero^S(\pi)$ and $\Zero(\pi)$ exist, the second one does.
\vskip 10pt

Another difficulty lies in the fact that our second term identity is ``weak" in the sense that it works only for those $\varphi$'s in $\Sw((\V^{+}\otimes W)(\A))^\circ$. Of course, it is also expected that the second term identity is true in the strong form. Namely, we have

\begin{conj}\label{C:SW}
The weak second term identity (Theorem \ref{T:Second_Term_Identity}) as well as the weak first term identity on the boundary (Proposition \ref{P:weak_first_term}) can be extended to all $\varphi\in\Sw((\V^{+}\otimes W)(\A))$.
\end{conj}

Then we can prove the following
\begin{prop}\label{P:lowest_occurrence}
Let $\pi$ be a cuspidal automorphic representation of $\OO(V,\A)$ with $\dim V=m$.
\begin{enumerate}[(a)]
\item Assume Conjecture \ref{C:GCD} (a), and both $\Pole(\pi)$ and $\Pole^S(\pi)$ exist. Then
\[
	\frac{m}{2}-\Pole(\pi)\leq\LO(\pi)\leq\frac{m}{2}-\Pole^S(\pi).
\]
\item Assume Conjecture \ref{C:GCD} (b) and \ref{C:SW} and both $\Zero(\pi)$ and $\Zero^S(\pi)$ exist. Further assume that $m$ is even, and there is a place $v$ such that $\pi_v\cong\pi_v\otimes\det$.Then
\[
	\frac{m}{2}-1+\Zero(\pi)\leq\LO(\pi)\leq\frac{m}{2}-1+\Zero^S(\pi).
\]
\end{enumerate}
\end{prop}
\begin{proof}
First let us consider (b). We already know that $\LO(\pi)\leq\frac{m}{2}-1+\Zero^S(\pi)$. Let $r=\LO(\pi)$ and so there exists $\mu$ such that $\theta_{\psi,2r}(\pi\otimes\mu)\neq 0$, and moreover since $r=\LO(\pi)$, it is in the space of cusp forms. So there exists $f_1, f_2\in\pi\otimes\mu$ and $\phi_1,\phi_2\in\Sw((V\otimes W^+)(\A))$ such that
\[
	\langle\theta_{2r}(f_1,\phi_1),\theta_{2r}(f_2,\phi_2)\rangle\neq 0,
\]
where $\dim W=2r$. Notice that we may assume that $f_i$'s are factorizable. But we do not know if $\sigma(\phi_1\otimes\phi_2)$ is in $\Sw((\V^+\otimes W^+)(\A))^\circ$. (Here $\sigma$ is as in Proposition \ref{P:inner_product_2nd_term1}.) However if we assume Conjecture \ref{C:SW}, by using essentially the same proof as the one for Theorem \ref{T:inner_product_2nd_term2}, we can obtain
\begin{align*}
    &\langle\theta_{2r}(f_1,\phi_1),\theta_{2r}(f_2,\phi_2)\rangle\\
    &=\frac{1}{\kappa d^S_{m}(\rho_{m,r})}\sum_i\underset{s=\rho_{m,r}}{\Val}
    \left(L^S(s+\frac{1}{2}, \pi\otimes\mu)\prod_{v\in S}Z_v(s,f_{1,v}, f_{2,v},
    \Phi_{i,v},\pi_v\otimes\mu_v)\right)\\
    &=\frac{1}{\kappa d^S_{m}(\rho_{m,r})}\sum_i\underset{s=\rho_{m,r}}{\Val}
    \left(L(s+\frac{1}{2}, \pi\otimes\mu)\prod_{v\in S}Z^*_v(s,f_{1,v}, f_{2,v},
    \Phi_{i,v},\pi_v\otimes\mu_v)\right),
\end{align*}
for some standard section $\sum_i\Phi_i$ so that each $\Phi_i$ is factorizable. Then if the left hand side is non-zero, then for some $i$,
\[
    \underset{s=\rho_{m,r}}{\Val}
    \left(L(s+\frac{1}{2}, \pi\otimes\mu)\prod_{v\in S}Z^*_v(s,f_{1,v}, f_{2,v},
    \Phi_{i,v},\pi_v\otimes\mu_v)\right)
\]
must be non-zero. By Conjecture \ref{C:GCD} (b), we know that each product
\[
    \prod_{v\in S}Z^*_v(s,f_{1,v}, f_{2,v}, \Phi_i,\pi_v\otimes\mu_v)
\]
is holomorphic at $s=\rho_{m,r}$. Also by our assumption, $L(s+\frac{1}{2}, \pi)$, which is the same as $L(s+\frac{1}{2}, \pi\otimes\mu)$, does not have a pole at this point. So all those imply $L(s+\frac{1}{2}, \pi)$ must be non-zero at $s=\rho_{m,r}$. Then we have
\[
	\Zero(\pi)\leq r-\frac{m}{2}+1.
\]
Hence the proposition follows.
\vskip 10pt

Next we consider (a). For this ``pole range", to invoke the inner product formula to the extent we need, we do not have to extend the first term identity as in Proposition \ref{P:first_term2}, but we only need the version of Proposition \ref{P:first_term}, which already works for any $\varphi\in\Sw((\V^{+}\otimes W)(\A))$. Then the proposition follows in a similar way as (b). The detail is left to the reader.
\end{proof}

\quad\\


\end{document}